\documentclass{amsart}

\usepackage{ae,aecompl}
\usepackage{esint}
\usepackage{graphicx}
\usepackage[all,cmtip]{xy}
\usepackage{amsmath, amscd}
\usepackage{booktabs}
\usepackage{amsthm}
\usepackage{amssymb}
\usepackage{amsfonts}
\usepackage{qsymbols}
\usepackage{latexsym}
\usepackage{mathrsfs}
\usepackage{cite}
\usepackage{color}
\usepackage{url}
\usepackage{enumerate}
\usepackage{undertilde}

\usepackage{verbatim}
\usepackage[draft=false, colorlinks=true]{hyperref}

\allowdisplaybreaks

\newcommand {\A}{{\mathcal{A}}}

\newcommand\ang[1]{\langle #1 \rangle}

\newcommand {\bmo}{{\mathrm{bmo}}}
\newcommand {\BMO}{{\mathrm{BMO}}}

\newcommand {\C}{{\mathbb C}}

\newcommand\cf{\alpha_{S^{n-1}}}
\newcommand {\con}{{M}}
\newcommand {\Cr}{{C^{r}_{-}}}
\newcommand {\Crtwo}{{C^{2}_{-}}}
\newcommand {\D}{D}
\newcommand {\Da}{\mathcal{J}}

\newcommand {\ud}{\mathrm{d}}

\newcommand {\eps}{\varepsilon}
\newcommand {\veps}{\varepsilon}

\newcommand {\ex}{\mathbf{e}}
\newcommand {\F}{{\mathcal{F}}}

\newcommand {\hchi}{\hat{\chi}}

\newcommand {\HT}{\mathcal{H}}
\newcommand {\Hp}{\mathcal{H}^{p}_{FIO}(\Rn)}
\newcommand {\Hps}{\mathcal{H}^{s,p}_{FIO}(\Rn)}  
\newcommand {\Hpsm}{\mathcal{H}^{s-1,p}_{FIO}(\Rn)}   
\newcommand {\Hpsmm}{\mathcal{H}^{s-2,p}_{FIO}(\Rn)} 
\newcommand {\Hpsp}{\mathcal{H}^{s(p),p}_{FIO}(\Rn)}

\newcommand {\ind}{{\mathbf{1}}}

\newcommand {\rb}{\rangle}
\newcommand {\lb}{{\langle}}
\newcommand {\La}{{\mathcal{L}}}
\newcommand {\loc}{{\mathrm{loc}}}

\newcommand {\N}{{{\mathbb N}}}

\newcommand {\ph}{{\varphi}}

\newcommand {\R}{{\mathbb R}}

\newcommand {\Rn}{{\mathbb{R}^{n}}}

\newcommand {\supp}{{\mathrm{supp}}}

\newcommand {\Sp}{S^{*}\Rn}
\newcommand {\Spp}{S^{*}_{+}\Rn}

\newcommand {\Sw}{\mathcal{S}}

\newcommand {\Tp}{T^{*}\Rn}
   
\newcommand{\Tentps}{T^{p}_{s}(\Sp)}   
\newcommand{\Tentpsm}{T^{p}_{s-1}(S^*\Rn)}

\newcommand {\w}{{\omega}}

\newcommand {\Z}{{{\mathbb Z}}}

\newcommand {\vanish}[1]{\relax}

\newcommand\xihat{{\hat \xi}}

\newcommand\Proj{P_{\xihat}^\perp}
\newcommand\bhom{a^{\mathrm{hom}}}
\newcommand\ahom{a}
\newcommand\Vhom{V^{\mathrm{hom}}}
\newcommand\Phihom{\Phi^{\mathrm{hom}}}

\newcommand\mcD{\mathcal{D}}

\newcommand{\wh}{\widehat}
\newcommand{\wt}{\widetilde}

\DeclareFontFamily{U}{mathx}{\hyphenchar\font45}
\DeclareFontShape{U}{mathx}{m}{n}{
      <5> <6> <7> <8> <9> <10>
      <10.95> <12> <14.4> <17.28> <20.74> <24.88>
      mathx10
      }{}
\DeclareSymbolFont{mathx}{U}{mathx}{m}{n}
\DeclareFontSubstitution{U}{mathx}{m}{n}
\DeclareMathAccent{\widecheck}{0}{mathx}{"71}

\newcommand {\cs}{\mathbf{c}}
\newcommand {\sn}{\mathbf{s}}

\newtheorem{theorem}{Theorem}[section]
\newtheorem{lemma}[theorem]{Lemma}
\newtheorem{proposition}[theorem]{Proposition}
\newtheorem{corollary}[theorem]{Corollary}

\theoremstyle{definition}
\newtheorem{definition}[theorem]{Definition}
\newtheorem{remark}[theorem]{Remark}

\numberwithin{equation}{section}
\protected\def\ignorethis#1\endignorethis{}
\let\endignorethis\relax

\title[$L^{p}$ and $\HT^{p}_{FIO}$ regularity for rough wave equations]{$L^{p}$ and $\HT^{p}_{FIO}$ regularity for wave equations with rough coefficients}

\author{Andrew Hassell}
\address{Mathematical Sciences Institute\\ Australian National University\\Acton ACT 2601, Australia}
\email{Andrew.Hassell@anu.edu.au}

\author{Jan Rozendaal}
\address{Institute of Mathematics, Polish Academy of Sciences\\
ul.~\'{S}niadeckich 8\\
00-656 Warsaw\\
Poland}
\email{jrozendaal@impan.pl}


\thanks{This research was supported by grant DP160100941 of the Australian Research Council.}

\begin{document}

\begin{abstract}
We consider wave equations with time-independent coefficients that have $C^{1,1}$ regularity in space.
 We show that, for nontrivial ranges of $p$ and $s$, the standard inhomogeneous initial value problem for the wave equation is well posed in Sobolev spaces $\Hps$ over the Hardy spaces $\Hp$ for Fourier integral operators introduced recently by the authors and Portal, following work of Smith.
   In spatial dimensions $n = 2$ and $n=3$, this includes the full range $1 < p < \infty$. 
As a corollary, we obtain the optimal fixed-time $L^{p}$ regularity for such equations, generalizing work of Seeger, Sogge and Stein in the case of smooth coefficients. 
\end{abstract}

\maketitle

\tableofcontents

\section{Introduction}

In this article, we develop the fixed-time $L^p$ theory of rough wave equations. 

In 1991 Seeger, Sogge and Stein \cite{SeSoSt91} determined the sharp $L^{p}$ mapping properties of Fourier integral operators, by showing that such operators lose $(n-1)|\frac{1}{2} - \frac{1}{p}|$ derivatives on $L^{p}(\R^n)$. As a consequence, they obtained the optimal fixed-time $L^{p}$ regularity for wave equations with smooth coefficients. Until now, there has been no corresponding result for wave equations with rough coefficients. 

In fact, until recently, there was no obvious strategy to approach the fixed-time $L^{p}$ theory of rough wave equations.
Existing work on wave equations with rough coefficients by Smith \cite{Smith98b,Smith06,Smith14} and Tataru \cite{Tataru00, Tataru01, Tataru02} proceeds by first 
smoothing the coefficients of the equation. An approximate solution to the smoothed equation is then constructed using microlocal analysis, and the solution of the original equation is obtained via a 
method of successive approximations. 
On $L^{p}(\Rn)$ this procedure leads to a loss of $(n-1)|\frac{1}{2}-\frac{1}{p}|$ derivatives in each iteration step, and the approach thus breaks down fundamentally for $p\neq 2$.

Recently, the authors, in collaboration with Portal and building on an earlier construction by Smith \cite{Smith98a}, introduced a scale of Hardy spaces for Fourier integral operators \cite{HaPoRo20}. These $\Hp$ spaces embed into the $L^{p}(\Rn)$-based Sobolev scale, and Fourier integral operators of order zero act on them \emph{without loss of derivatives}. This suggests that these spaces can also be used to analyze rough wave equations, by providing a framework for the method of iterative approximation. 

In the present article, we implement this strategy. We show that for every dimension $n$ there is a nontrivial range of $p$ around $2$ for which wave equations with $C^{1,1}$ coefficients are well posed on suitable Sobolev spaces $\Hps$ over $\Hp$. By combining this with the Sobolev embeddings for $\Hp$, one obtains the optimal $L^{p}$ regularity for such equations. This $L^{p}$ regularity is, however, significantly weaker than the well-posedness of these equations on Hardy spaces for Fourier integral operators, and we wish to promote the viewpoint that fixed-time $L^p$ analysis of wave equations is best carried out directly on $\Hp$; these are the `natural' spaces for the $L^{p}$ theory. The second author's recent use \cite{Rozendaal22b} of the Hardy spaces for Fourier integrals to improve the Bourgain-Demeter local smoothing estimates from \cite{Bourgain-Demeter15} provides further evidence for this viewpoint.

\subsection{Setting} 

We study the inhomogeneous initial value problem for the wave equation in $n+1$ dimensions: 
\begin{equation}\begin{gathered}
D_{t}^{2} u(t,x)-\sum_{i,j=1}^{n} D_{i} (a_{i,j}(x) D_{j} u)(t,x) = F(t,x) , \quad u:\R^{n+1} \to \C,\\
u(0,x) = f(x), \qquad D_t u(0,x) = g(x), \qquad D = -i \partial. 
\end{gathered}\label{eq:wave-eqn}\end{equation}
Here the $a_{i,j}$ are uniformly elliptic, bounded and real-valued, and `rough' in the sense that they possess a limited number of derivates.

We briefly recall part of the regularity theory of solutions to \eqref{eq:wave-eqn}, and the associated harmonic analysis. In the case of the Laplacian, where $a_{i,j}=\delta_{ij}$ for $1\leq i,j\leq n$, the solution operators $\cos(t\sqrt{-\Delta})$ and $\sin(t\sqrt{-\Delta})/\sqrt{-\Delta}$ to \eqref{eq:wave-eqn} are bounded on $L^{2}(\Rn)$ for all $t\in\R$, as follows from either Plancherel's theorem or spectral calculus. Correspondingly, \eqref{eq:wave-eqn} is well posed on $L^{2}(\Rn)$. For more general smooth coefficients $a_{i,j}$, the regularity theory of \eqref{eq:wave-eqn} reduces to the study of mapping properties of Fourier integral operators, a class of oscillatory integral operators that includes the solution operators to \eqref{eq:wave-eqn}. Since suitable Fourier integral operators of order zero are bounded on $L^{2}(\Rn)$, \eqref{eq:wave-eqn} is well posed in $L^{2}(\Rn)$ if the coefficients $a_{i,j}$ are smooth.

The regularity theory for \eqref{eq:wave-eqn} becomes more involved when considering initial data in Sobolev spaces over $L^{p}(\Rn)$ for $p\neq 2$. Indeed, an examination of the kernel of $\cos(t\sqrt{-\Delta})$ shows that this operator is not bounded on $L^{p}(\Rn)$ unless $p=2$, $n=1$ or $t=0$. In fact, it was shown by Peral~\cite{Peral80} and Miyachi~\cite{Miyachi80a} that 
\begin{equation}\label{eq:classicalwave}
\cos(t\sqrt{-\Delta}):W^{2s(p),p}(\Rn)\to L^{p}(\Rn)
\end{equation}
for $1<p<\infty$ and 
\begin{equation}
s(p):=\frac{n-1}{2}\Big|\frac{1}{2}-\frac{1}{p}\Big|,
\label{sp}\end{equation} 
and this `loss of $2s(p)$ derivatives' cannot be improved. The same mapping property \eqref{eq:classicalwave} holds more generally for compactly supported Fourier integral operators of order zero associated to a canonical transformation, as was proved by Seeger, Sogge and Stein in their celebrated paper~\cite{SeSoSt91}. Moreover, a corresponding result holds for the endpoints $p=1$ and $p=\infty$ upon replacing $L^{p}(\Rn)$ by the local Hardy space $\HT^{1}(\Rn)$ and $\bmo(\Rn)$, respectively. As a result, although \eqref{eq:wave-eqn} is not well posed on $L^{p}(\Rn)$ for $p\neq 2$ and $n>1$, for smooth coefficients the fixed-time $L^{p}$ regularity theory is well understood.

On the other hand, motivated by applications to regularity theory, a classical problem in harmonic analysis is to determine the $L^{p}$ mapping properties of the solution operators to partial differential equations with rough coefficients. For example, it is not merely the class of pseudodifferential operators, which includes the solution operators to smooth elliptic equations, that is bounded on $L^{p}(\Rn)$ for all $1<p<\infty$. It has also long been known that their rough analogues, Calderon-Zygmund operators, are bounded on $L^{p}(\Rn)$ for all $1<p<\infty$ provided they are bounded on $L^{2}(\Rn)$ \cite{Stein93}. Similarly, for the parabolic theory, the solution operators $(e^{tL})_{t \geq 0}$ to the equation $\partial_{t}u=Lu$ are bounded on $L^{p}(\R^{n})$ for a range of $p$ around $2$, if $L$ is a uniformly elliptic operator in divergence form with $L^{\infty}$ coefficients \cite{Auscher07}. 

In this respect, elliptic and parabolic theory differs from hyperbolic theory, and specifically from the theory of the wave equation \eqref{eq:wave-eqn}, in two ways. One is that the solution operators to elliptic and parabolic equations are typically bounded on $L^{p}(\Rn)$ for $1<p<\infty$, even if the coefficients of the equation have very limited smoothness. The other difference is that surprisingly little is known about the optimal $L^{p}$ regularity theory for wave equations with rough coefficients. Our goal in this article is to address this deficiency by extending the sharp $L^{p}$ regularity theory for wave equations with smooth coefficients to equations with rough coefficients.

\subsection{Hardy spaces for Fourier integral operators} 

An effective method for studying rough wave equations was developed by Smith in~\cite{Smith98b}. Combining paradifferential methods due to Bony~\cite{Bony81} with wave packet transforms, he constructed a useful microlocal parametrix on $L^{2}(\Rn)$ and then corrected to the exact solution using an iterative procedure. This strategy was exploited by Smith, as well as  Tataru, to obtain powerful results on rough wave equations, such as Strichartz estimates~\cite{Metcalfe-Tataru12,Smith98b,Tataru00,Tataru01,Tataru02}, propagation of singularities~\cite{Smith14}, the related spectral cluster estimates~\cite{Smith06}, and well-posedness of nonlinear wave equations with rough initial data \cite{Smith-Tataru05}.

It is then natural to attempt to apply these techniques to the fixed-time $L^{p}$ regularity theory for wave equations with rough coefficients. However, here one immediately runs into a formidable obstacle: no reasonable approximation of the solution operators to a rough wave equation can be expected to behave better than the solution operators to smooth wave equations, so in particular one cannot expect them to be bounded on $L^{p}(\Rn)$ for $p\neq 2$. This means that for the type of iterative constructions that are typically used to construct a parametrix, on $L^{p}(\Rn)$ one will lose a fixed number of derivatives in each iteration step, and the iterative `loop' cannot be closed.

On the other hand, in~\cite{Smith98a} Smith introduced a substitute for the classical local Hardy space $\HT^{1}(\Rn)$ which is adapted to Fourier integral operators. His space $\HT^{1}_{FIO}(\Rn)$ is invariant under compactly supported Fourier integral operators of order zero associated with a canonical transformation, and it satisfies Sobolev embeddings  that allow one to recover the results of Seeger, Sogge and Stein. Smith's construction was subsequently extended to a full range $(\Hp)_{1\leq p\leq \infty}$ of invariant spaces by the authors and Portal in~\cite{HaPoRo20}. These Hardy spaces for Fourier integral operators are invariant under suitable Fourier integral operators of order zero, and they satisfy the following Sobolev embeddings into the $L^{p}$ scale:
\begin{equation}\label{eq:Sobolevintro}
W^{s(p),p}(\Rn)\subseteq \Hp\subseteq W^{-s(p),p}(\Rn), \quad 1 < p < \infty,
\end{equation}
where $s(p)$ is as in \eqref{sp},  with the natural modifications involving $\HT^{1}(\Rn)$ and $\bmo(\Rn)$ for $p=1$ and $p=\infty$, respectively. In particular, by considering the Sobolev space $\Hpsp$ over $\Hp$, one recovers as a special case \eqref{eq:classicalwave} and the optimal $L^{p}$ regularity for Fourier integral operators. 

However, beyond merely recovering existing results, the Hardy spaces for Fourier integral operators provide us with a framework to apply the existing techniques for rough wave equations to the fixed-time $L^{p}$ regularity theory, by proving that such equations are solvable over $\Hp$ and then using \eqref{eq:Sobolevintro}.

\subsection{Main results}

The Sobolev spaces $\Hps=\lb D\rb^{-s}\Hp$ over the Hardy spaces for Fourier integral operators are introduced in Definition \ref{def:HardyFIO}. We say that \eqref{eq:wave-eqn} is well posed in $\Hps$ for given $p\in(1,\infty)$ and $s\in\R$ if, for each $f\in\Hps$, $g\in\Hpsm$ and $F\in L^{1}_{\loc}(\R;\Hpsm)$, there exists a unique
\[
u \in C(\R;\Hps)\cap C^{1}(\R;\Hpsm)\cap W^{2,1}_{\loc}(\R;\Hpsmm)
\]
such that \eqref{eq:wave-eqn} holds as an identity in $\Hpsmm$ for almost all $t\in\R$. 

By combining a slightly simplified version of our main result, Theorem \ref{thm:divwave}, with the Sobolev embeddings in \eqref{eq:Sobolevintro}, one obtains the following result.

\begin{theorem}\label{thm:mainintro}
Suppose that $a_{i,j}\in C^{1,1}(\Rn)$ is bounded and real-valued for all $1\leq i,j\leq n$, and that $\sum_{i,j=1}^{n}D_{i}(a_{i,j}D_{j})$ is uniformly elliptic. Then \eqref{eq:wave-eqn} is well posed in $\Hps$ for all $p\in(1,\infty)$ such that $2s(p)<1$ and all $-1+s(p)\leq s\leq 2-s(p)$. In particular, the solution operators $(U(t))_{t\in\R}$ to \eqref{eq:wave-eqn} satisfy
\begin{equation}\label{eq:Lpreg}
U(t):W^{s+s(p),p}(\Rn)\to W^{s-s(p),p}(\Rn)
\end{equation}
for all $t\in\R$.

If $a_{i,j}\in C^{r}(\Rn)$ for some $r>2$ and all $1\leq i,j\leq n$, then \eqref{eq:wave-eqn} is well posed in $\Hps$ for all $p\in(1,\infty)$ such that $2s(p)<r-1$ and all $-r+s(p)+1<s< r-s(p)$. In particular, for such $p$ and $r$, \eqref{eq:Lpreg} holds for all $t\in\R$.
\end{theorem}

Note that the first two statements apply to all $p\in(1,\infty)$ if $n\leq 3$. For $n\geq 4$, the last two statements apply to all $p\in(1,\infty)$ if $r\geq (n+1)/2$. Explicitly, for $n\geq4$ the condition on $p$ in the first statement is that
\[
2\frac{n-1}{n+1}<p<2\frac{n-1}{n-3}.
\]
If $r>2$ is an integer, then one may replace $C^{r}(\Rn)$ by $C^{r-1,1}(\Rn)$ in the second statement, and one may also include the endpoints of the Sobolev interval.  

We obtain a similar result as in Theorem \ref{thm:mainintro} for operators $\sum_{i,j=1}^{n}a_{i,j}D_{i}D_{j}$ in standard form, with the same assumptions on $p$, but with a Sobolev interval for $s$ that is shifted by $1$ (see Theorem \ref{thm:mainwave}). We consider only pure second order operators in this article, for simplicity, but in Remark \ref{rem:lowerorder} we indicate how our techniques can be used to include lower order terms. 
In Appendix \ref{app:regularity} we detail how one can improve the second statement, for $r>2$, by making regularity assumptions outside of the H\"{o}lder scale. This also explains why the case where $r$ is an integer is special in Theorem \ref{thm:mainintro}.

To the authors' best knowledge, Theorem \ref{thm:mainintro} is the first instance where the fixed-time $L^{p}$ regularity theory for wave equations has been extended beyond the seminal work for smooth equations by Seeger, Sogge and Stein in 1991 to a general class of rough wave equations. Results for even rougher wave equations have been available through the theory of spectral multipliers for much longer~\cite{DuOuSi02}; however, these yield a sub-optimal loss of at least $n|\frac{1}{2}-\frac{1}{p}|$ derivatives even in the case of the Laplacian, cf.~\eqref{eq:fracLaplace}. In \cite{DosSantosFerreira-Staubach14,Rodriguez-Staubach13} one may find results on the $L^{p}$ regularity of oscillatory integral operators with rough phase functions and symbols, some of which involve a loss of $(n-1)|\frac{1}{2}-\frac{1}{p}|+\veps$ derivatives. However, the very procedure which expresses the solution operators to a wave equation locally as a sum of oscillatory integral operators plus a smoothing term, namely solving an eikonal equation and then iteratively transport equations, requires the coefficients of the equation to be smooth. We also refer to \cite{Frey-Portal20} for a very recent contribution to this area, which treats equations with even rougher coefficients that are small perturbations of commuting differential operators.

It is important to stress that the assumption of $C^{1,1}$ regularity in Theorem \ref{thm:mainintro} is common in the analysis of wave equations. In fact, although both Strichartz and spectral cluster estimates hold in the classical sense for wave equations with $C^{1,1}$ coefficients~\cite{Smith98b,Tataru01,Tataru02,Smith06}, it is known that they only hold in a weaker sense for $C^{1,\alpha}$ coefficients with $\alpha<1$~\cite{Smith-Sogge92,Smith-Tataru02}. The fact that this level of regularity is critical manifests itself in multiple ways, as is explained below. 

In this article we do not treat time-dependent coefficients. This is not due to an inherent limitation of our methods. In fact, the Strichartz estimates obtained by Smith and Tataru using similar methods also hold for time-dependent coefficients, and one may even allow for time dependence of slightly lower regularity than the spatial regularity (see also~\cite{MaMeTa08}). The same applies to propagation of singularities~\cite{Smith14}, and our results can also be extended to coefficients that depend on time in a rough sense. However, doing so introduces various subtle technical issues that tend to obfuscate the main ideas for anyone who is not already familiar with them. This manifests itself in various ways, such as in the concept of a solution (see e.g.~\cite{Smith98b}), in the proof of our main results, and in Sections \ref{sec:flow} and \ref{sec:parametrix} (see e.g.~Remark \ref{rem:time-dep}). We believe that there is considerable value in illustrating the key ideas to a wider audience by considering a slightly simpler setting. In a follow-up article, we plan to deal with time-dependent coefficients, and treat propagation of singularities in $\Hps$. 

Finally, it is illustrative to compare our results to the $L^{p}$ theory for Schr\"{o}dinger equations. Invariant spaces for Schr\"{o}dinger propagators have been used in time-frequency analysis for much longer; these are the modulation spaces (see e.g.~\cite{CoFaRo10}). However, results about the fixed-time $L^{p}$ regularity for Schr\"{o}dinger equations have been obtained using more classical methods from harmonic analysis, and under the much weaker regularity assumption of Gaussian heat kernel bounds~\cite{ChDuLiYa20}. On the other hand, it is well known that the fixed-time $L^{p}$ regularity theory for wave equations is more involved than that for Schr\"{o}dinger equations, as is already evident for the flat Laplacian. Indeed, cf.~\cite{Miyachi80b}, for all $\alpha>0$ one has
\begin{equation}\label{eq:fracLaplace}
e^{it(-\Delta)^{\alpha}}:W^{2\alpha n|\frac{1}{2}-\frac{1}{p}|,p}(\Rn)\to L^{p}(\Rn)
\end{equation}
for $p\in(1,\infty)$ and $t\in\R$, and this exponent is sharp for $t\neq 0$ unless $\alpha=1/2$, in which case the improved estimate \eqref{eq:classicalwave} holds. From a microlocal viewpoint, this phenomenon can be understood by observing that for $t\neq 0$ the projection of the canonical relation $\{(y-2\alpha t|\eta|^{2\alpha-1}\hat{\eta},\eta,y,\eta)\mid (y,\eta)\in\Tp\setminus o\}$ of $e^{it(-\Delta)^{\alpha}}$ onto the base spaces $\Rn\times\Rn$ has rank $2n-1$ if $\alpha=1/2$, and rank $2n$ otherwise. Moreover, it is known that the $L^{p}$ mapping properties of oscillatory integral operators improve when the rank of this projection drops~\cite{SeSoSt91,Ruzhansky99}, with operators for which the projection has rank $n$, such as pseudodifferential operators, being bounded. Hence sharp fixed-time $L^{p}$ regularity results for wave equations need to take into account such subtle microlocal aspects, and the theory is more involved. On the other hand, the space-time Strichartz estimates have been obtained for Schr\"{o}dinger equations with rough coefficients in~\cite{Tataru08} using a similar combination of wave packet transforms and parametrices as described above.

\subsection{Overview of the proof}

The general strategy of our proof is similar to that of previous results for wave equations with rough coefficients. Here we sketch this approach, and we indicate where our implementation differs from the existing instances of this method. 

Suppose that $a_{i,j}\in C^{1,1}(\Rn)$ for all $1\leq i,j\leq n$. Using a paradifferential smoothing procedure which goes back to Bony~\cite{Bony81} (see also e.g.~Meyer~\cite{Meyer81a, Meyer81b} and Taylor~\cite{Taylor91,Taylor00}), one separates the low and the high frequencies of the coefficients $a_{i,j}$. This decomposes $L:=\sum_{i,j=1}^{n}D_{i}(a_{i,j}D_{j})$ as a sum of a smooth pseudodifferential operator $L_{1}$ with symbol in H\"{o}rmander's $S^{2}_{1,1/2}$ class, and a rough pseudodifferential operator $L_{2}$ with principal symbol in $C^{1,1}S^{1}_{1,1/2}$. The latter class consists of symbols that behave like elements of $S^{1}_{1,1/2}$, except that they have the same spatial regularity as the coefficients $a_{i,j}$. Note, however, that the operator $L_{2}$ is of order $1$. Hence it is reasonable to expect that any $u\in C(\R;\Hps)$ satisfies  $L_{2}u\in C(\R;\Hpsm)$, in which case $L_{2}u$ has the same `strength' as the inhomogeneous term $F$. Then, heuristically, we can rewrite \eqref{eq:wave-eqn} as $(D_{t}^{2}-L_{1})u=F-L_{2}u$, and use Duhamel's principle to reduce to the equation $(D_{t}^{2}-L_{1})u=0$. In fact, for technical reasons it turns out to be more convenient to solve the first order equation $(D_{t}-b(x,D))u=0$, where $b\in S^{1}_{1,1/2}$ and $b(x,D)$ is an approximate square root of $L_{1}$. 

Now, to solve $(D_{t}-b(x,D))u=0$ one cannot directly rely on standard tools from microlocal analysis. Indeed, the symbol $b$ of the pseudodifferential operator $b(x,D)$ is not homogeneous, and therefore the solution operators to this equation are not Fourier integral operators. Instead, following an idea which was already present in the work of Cordoba and Fefferman~\cite{Cordoba-Fefferman78}, one can use wave packet transforms to lift the equation to phase space, where the quantum correspondence principle suggests that the solution operators are well approximated by the bicharacteristic flow maps associated with $b$. By conjugating these flow maps with wave packet transforms, and using an iterative construction to deal with error terms that arise, one then obtains a parametrix for the equation $(D_{t}-b(x,D))u=0$ (see Remark \ref{rem:parametrix}). In turn, to make the heuristics involving $L_{2}u$ precise, one uses another iterative construction to solve the original equation \eqref{eq:wave-eqn}. 

Apart from providing a method to solve \eqref{eq:wave-eqn}, the idea of conjugating flow maps on phase space with wave packet transforms results in a parametrix that is both conceptually elegant and technically convenient. Moreover, one can show that this parametrix has various useful properties, which in turn are transferred to the original equation and allow one to derive Strichartz estimates, or propagation of singularities, for example. 

We note that, in the procedure above, the assumption that $a_{i,j}\in C^{1,1}(\Rn)$ for $1\leq i,j\leq n$ is used in three important ways. Firstly, it guarantees that the operator $L_{2}$ has order $1$. In general, if $a_{i,j}\in C^{r}(\Rn)$ for some $r>0$ (or $a_{i,j}\in C^{r-1,1}(\Rn)$ if $r\in\N$), then $L_{2}$ has order $2-r/2$.  Moreover, the regularity of the $a_{i,j}$ is used to construct the flow parametrix for $(D_{t}-b(x,D))u=0$. Indeed, the symbol $b$ in fact has slightly better properties than a general $S^{1}_{1,1/2}$ symbol, and one can take two spatial derivatives without incurring blow-up in the fiber variable. This fact ensures, secondly, that the flow associated with $b$ is well behaved (see Remark \ref{rem:C11needed}), and, thirdly, that one can obtain useful error bounds when applying $b(x,D)$ to the conjugated flow (see Remark \ref{rem:sharpreg}). In fact, the regularity of the $a_{i,j}$ also leads to the condition that $2s(p)<1$, via Theorem \ref{thm:roughpseudo}, but here the $C^{1,1}$ assumption only determines the size of the interval of $p$ in Theorem \ref{thm:mainintro}.

Our incarnation of this method differs from the existing ones in three main ways.  The \emph{first} is that we work with the Hardy spaces $\Hps$ for Fourier integral operators, so we require estimates for the rough term $L_{2}$ on these spaces. By contrast, in the existing works on rough wave equations it suffices to obtain $L^{2}(\R^n)$ estimates (e.g.~by a $TT^{*}$ argument). Estimates for rough pseudodifferential operators on $L^{p}(\R^n)$ are classical~\cite{Bourdaud82,Marschall88,Taylor91}, but they are new for $\Hp$. These estimates are proved in the companion papers \cite{Rozendaal20,Rozendaal22} by the second author, and they are of independent interest. They can be found in Theorem \ref{thm:roughpseudo} of the present article, and we stress that these mapping properties are proved in a very different manner than the results in this article, relying in particular on Littlewood--Paley theory, equivalent characterizations of $\Hp$ (see~\cite{Rozendaal21,FaLiRoSo19} and \eqref{eq:equivchar}), and interpolation of rough symbol classes. It should also be noted that it is the rough operator $L_{2}$ which leads to the restrictions on $p$ and $s$ in Theorem \ref{thm:mainintro}, whereas our results for the smooth term hold for all $p\in[1,\infty)$ and $s\in\R$ (we will not consider $p=1$ or $p=\infty$ in this article). 
On the other hand, by making regularity assumptions outside of the H\"{o}lder scale, slightly less than $C^{1,1}$ regularity in fact suffices to deal with this operator (see Appendix \ref{app:regularity}). 

In the present article we solve the smoothed equation $(D_{t}-b(x,D))u=0$. Then we combine this solution with the rough term, treated as an additional inhomogeneity, to prove solvability for \eqref{eq:wave-eqn} on $\Hps$. For the latter step we rely in an essential manner on the fact that $\Hps$ is invariant under the solution operators to $(D_{t}-b(x,D))u=0$. This allows for iterative arguments that would not be possible by directly working on $L^{p}(\Rn)$, and it also explains why it is necessary to obtain estimates for the rough term $L_{2}$ on $\Hps$.

To construct a parametrix for the equation $(D_{t}-b(x,D))u=0$, we conjugate flow maps with wave packet transforms. However, unlike many of the existing methods, we cannot rely on isotropic wave packet transforms such as the FBI transform~\cite{Tataru99}, as doing so would not lead to a sharp loss of derivatives. The \emph{second} way in which our method differs from most previous ones is that we use an anisotropic wave packet transform $W$ which already appeared in our previous work~\cite{HaPoRo20}. This transform has its roots in~\cite{Smith98a} and can also be found in~\cite{Geba-Tataru07}. It captures the dyadic-parabolic, or second dyadic, decomposition of phase space which goes back to Fefferman~\cite{Fefferman71,Fefferman73b} and which is a key tool in the work of Segger, Sogge and Stein~\cite{SeSoSt91} on the $L^{p}$ regularity of Fourier integral operators.

The \emph{third} way in which our method differs from the existing ones is that we combine flow maps on phase space with the theory of tent spaces. Tent spaces, introduced in~\cite{CoMeSt85}, provide a powerful framework in harmonic analysis that has proved to be particularly useful when dealing with rough elliptic and parabolic equations~\cite{HoLuMiMiYa11}. Tent spaces $T^{p}(\Sp)$ over the cosphere bundle $\Sp=\Rn\times S^{n-1}$ are function spaces on phase space, $T^* \R^{n}$, and they incorporate the conical square functions that were originally used by Smith in~\cite{Smith98a} to define $\HT^{1}_{FIO}(\Rn)$. In our previous work~\cite{HaPoRo20}, $\Hp$ was defined by embedding it into $T^{p}(\Sp)$ using the wave packet transform $W$, and to deal with the Sobolev spaces $\Hps$ we work in the present article with weighted tent spaces $T^{p}_{s}(\Sp)$. As a result, to prove mapping properties of operators on $\Hps$, one can conjugate them with $W$ and prove estimates in the $T^{p}_{s}(\Sp)$ norm. In the setting of the present article, this means that one has to prove that bicharacteristic flow maps are bounded on weighted tent spaces, and one has to obtain appropriate kernel bounds on phase space for error terms in the parametrix. This leads to additional difficulties that are not present in the $L^{2}$ theory. For example, although bicharacteristic flow maps, being symplectomorphisms, are automatically bounded on $L^{2}(\Tp)$, it is more difficult to prove boundedness of flow maps on tent spaces. The situation is further complicated by the fact that the symbol $b$ is not homogeneous.

It should be noted that one can also solve the equation $(D_{t}-b(x,D))u=0$ using methods from either~\cite{Smith98b} or~\cite{Geba-Tataru07}. However, our approach differs from these works in several ways. Namely, although the wave packet transform used by Smith in~\cite{Smith98b} is also anisotropic, it is a discrete transform that decomposes functions using a curvelet basis and views operators as infinite matrices with respect to this basis, as opposed to lifting functions and operators to phase space (see also~\cite{Candes-Demanet03}). Accordingly, the flow parametrix does not arise from an actual bicharacteristic flow on phase space. We believe that it is useful to construct an anisotropic flow parametrix on phase space involving genuine bicharacteristic flows, both for technical reasons and for conceptual simplicity, as the author himself did using isotropic transforms in his later work~\cite{Smith06,Smith14}. 

On the other hand, in~\cite{Geba-Tataru07} Geba and Tataru directly obtain appropriate kernel bounds for the solution operators to the equation $(D_{t}-b(x,D))u=0$. However, a key idea in our previous work~\cite{HaPoRo20} is to apply the existing theory of tent spaces to the $L^{p}$ theory for Fourier integral operators, and in the present article we develop this idea further by building parametrices using flows maps on tent spaces. Since the theory of tent spaces has so far mostly been restricted to (rough) elliptic and parabolic equations, where propagation of singularities plays no role, we believe that there is value in demonstrating how tent spaces can be combined with flow parametrices to study the fixed-time $L^{p}$ theory of rough wave equations.

\subsection{Organization of this article} This article is organized as follows. In Section \ref{sec:tentspaces} we collect some background on the weighted tent spaces $T^{p}_{s}(\Sp)$. Apart from the introduction of weights, one difference with respect to our earlier work in~\cite{HaPoRo20} is that we take an alternative viewpoint on tent spaces over the cosphere bundle, which effectively corresponds to parametrizing phase space using Cartesian coordinates as opposed to spherical coordinates (see Remark \ref{rem:differenttent}). The former is more useful when dealing with bicharacteristic flows, particularly for symbols that are not homogeneous. In Section \ref{sec:HpFIO} we introduce our wave packet transform and the Hardy spaces for Fourier integral operators. We also take a different viewpoint on this transform and the Hardy spaces for FIOs, although the resulting spaces are the same (see Remark \ref{rem:otherHpFIO}). 

In Section \ref{sec:roughsymb} we introduce the relevant rough symbol classes, and the symbol smoothing procedure which plays a key role in this article. In Section \ref{subsec:smoothand rough} we collect our main results for the smooth and rough terms in this symbol decomposition, Theorem \ref{thm:mainsmooth} and Theorem \ref{thm:roughpseudo}. Theorem \ref{thm:mainsmooth} is proved in Section \ref{sec:firstorder}, and Theorem \ref{thm:roughpseudo} is proved in~\cite{Rozendaal20,Rozendaal22}. Section \ref{sec:divform} then contains our main results for operators in divergence form, and in particular Theorem \ref{thm:mainintro}, and Section \ref{sec:standardform} contains our results for operators in standard form. The proofs of both results are very similar, with some minor technical differences.

Section \ref{sec:firstorder} reduces the proof of our main result for smooth first order equations to Theorem \ref{thm:parametrix}, which asserts that a suitable parametrix exists for such equations. Theorem \ref{thm:parametrix} is in turn proved in Sections \ref{sec:flow} and \ref{sec:parametrix}, with Section \ref{sec:flow} showing that the relevant flow maps are bounded on tent spaces, and Section \ref{sec:parametrix} dealing with the error bounds in the parametrix.

Finally, Appendix \ref{app:kernel} contains a statement about kernel bounds for oscillatory integral operators which is used frequently, and Appendix \ref{app:regularity} explains how subtle regularity assumptions on the coefficients can be used to improve some of our results. 

\subsection{Notation}\label{subsec:notation}

The natural numbers are $\N=\{1,2,\ldots\}$, and $\Z_{+}:=\N\cup\{0\}$. Throughout this article we fix $n\in\N$ with $n\geq2$. Our techniques can also be applied for $n=1$, but in that case the results can be improved, cf.~\cite{Frey-Portal20}. 

For $1\leq j\leq n$ we denote the $j$-th standard basis vector of $\Rn$ by $e_{j}$. For $\xi\in\Rn$ we set $\lb\xi\rb=(1+|\xi|^{2})^{1/2}$, and $\hat{\xi}=\xi/|\xi|$ if $\xi\neq0$. We use multi-index notation, where $\partial_{\xi}:=(\partial_{\xi_{1}},\ldots,\partial_{\xi_{n}})$, $\partial^{\alpha}_{\xi}=\partial^{\alpha_{1}}_{\xi_{1}}\ldots\partial^{\alpha_{n}}_{\xi_{n}}$ and $\xi^{\alpha}=\xi_{1}^{\alpha_{1}}\ldots\xi_{n}^{\alpha_{n}}$ for $\xi=(\xi_{1},\ldots,\xi_{n})\in\Rn$ and $\alpha=(\alpha_{1},\ldots,\alpha_{n})\in\Z_{+}^{n}$. Moreover, $D_{j}:=-i\partial_{x_{j}}$ and $D_{t}:=-i\partial_{t}$.

The duality between a Schwartz function $f\in\Sw(\Rn)$ and a tempered distribution $g\in\Sw'(\Rn)$ is denoted by $\lb f,g\rb$. 
The Fourier transform of $f\in\Sw'(\Rn)$ is denoted by $\F f$ or $\widehat{f}$, and the Fourier multiplier with symbol $\ph\in\Sw'(\Rn)$ is denoted by $\ph(D)$. 

The H\"{o}lder conjugate of $p\in[1,\infty]$ is $p'$, and $s(p)=\frac{n-1}{2}|\frac{1}{2}-\frac{1}{p}|$. The volume of a measurable subset $B$ of a measure space $(\Omega,\mu)$ is $|B|$. For $F\in L^{1}(B)$, we set
\[
\fint_{B}F(x)\ud\mu(x)=\frac{1}{|B|}\int_{B}F(x)\ud\mu(x)
\]
if $|B|<\infty$. The space of bounded linear operators on a Banach space $X$ is $\La(X)$. 

We write $f(s)\lesssim g(s)$ to indicate that $f(s)\leq Mg(s)$ for all $s$ and a constant $\con>0$ independent of $s$, and similarly for $f(s)\gtrsim g(s)$ and $g(s)\eqsim f(s)$.

\section{Tent spaces}\label{sec:tentspaces}

In this section we collect some background on weighted tent spaces.

\subsection{Definitions}\label{subsec:definitions}

Our tent spaces are function spaces on phase space. In~\cite{HaPoRo20} these spaces were defined using an implicit spherical coordinate system on phase space, as arises from the classical theory of tent spaces in e.g.~\cite{CoMeSt85, Amenta14}. More precisely, one parametrizes phase space using the cosphere bundle, which is convenient for dealing with flows that are associated with homogeneous symbols. However, in the present article we consider flows associated with symbols that are merely asymptotically homogeneous, and such flows do not project down to the cosphere bundle. Hence it is more convenient to use a Cartesian coordinate system. This leads to spaces that look different from those in~\cite{HaPoRo20}, but this difference is only apparent, cf.~Remark \ref{rem:differenttent}. We also work in the more general setting of weighted tent spaces, which allows us to deal with Sobolev spaces over the Hardy spaces for Fourier integral operators.  

Let $\Tp$ be the cotangent bundle of $\Rn$, identified with $\Rn\times\Rn$ and endowed with the symplectic form $d\xi \cdot dx$ and the Liouville measure $\ud x\ud\xi$.  
Let $o:=\Rn\times\{0\}$ be the zero section in $\Tp$, and let $\Sp=\Rn\times S^{n-1}$ be the cosphere bundle over $\Rn$. We shall denote elements of $S^{n-1}$ by either $\hat \xi=\xi/|\xi|$, for $\xi \in \Rn \setminus \{ 0 \}$, or by $\omega$ or $\nu$. We also denote the standard Riemannian metric on $S^{n-1}$ by $g_{S^{n-1}}$, and the standard measure on $S^{n-1}$ by $\ud \omega$. 
We will endow $\Sp$ with a metric $d$ that arises from contact geometry. For readers unfamiliar with contact geometry, we note that we mostly use only two properties of the metric:
\begin{itemize}
\item that it has a convenient equivalent expression, in \eqref{eq:metric};
\item that $(\Sp,d,\ud x\ud\w)$ is a doubling metric measure space, cf.~\eqref{eq:doubling}.
\end{itemize}  
The cosphere bundle $\Sp$ is a contact manifold with respect to the standard contact form $\alpha_{S^{n-1}}:= \hat \xi \cdot dx$, the kernel of which is a smooth distribution of codimension $1$ subspaces of the tangent bundle $T(\Sp)$ of $\Sp$. Set
\begin{equation}\label{eq:defd}
d((x, \omega), (y, \nu)) := \inf_{\gamma} \int_0^1 |\gamma'(s)|\ud s
\end{equation}
for $(x,\w),(y,\nu)\in\Sp$, where $|\gamma'(s)|$ is the length of the tangent vector $\gamma'(s)$ with respect to the product metric $dx^2 + g_{S^{n-1}}$, and the infimum is taken over all horizontal Lipschitz\footnote{
In \cite{HaPoRo20} the infimum was taken over all piecewise $C^{1}$ curves, but this modification makes no difference, and it will be convenient for us in Section \ref{sec:flow} (see footnote \ref{foot:Lipschitz}).\label{foot:curves}} curves $\gamma : [0,1] \to \Sp$ such that $\gamma(0) = (x, \omega)$ and $\gamma(1) = (y, \nu)$. Here horizontal means that $\cf(\gamma'(s)) = 0$ for almost all $s \in [0,1]$. Throughout, we let $B_{\tau}(x,\w)\subseteq\Sp$ be the open ball around $(x,\w)\in\Sp$ of radius $\tau>0$ with respect to $d$. 

We will often use an equivalent analytic expression for $d$. By~\cite[Lemma 2.1]{HaPoRo20},
\begin{equation}\label{eq:metric}
d((x,\w),(y,\nu))\eqsim \big(|x-y|^{2}+|\w\cdot (x-y)|+|\w-\nu|^{2}\big)^{1/2}
\end{equation}
for all $(x,\w),(y,\nu)\in\Sp$, where the implicit constants only depend on $n$. By~\cite[Lemma 2.3]{HaPoRo20}, there exists an $\con>0$ such that for all $(x,\w)\in\Sp$ one has
\begin{equation}\label{eq:volume}
\frac{1}{\con}\tau^{2n}\leq |B_{\tau}(x,\w)|\leq \con\tau^{2n}
\end{equation}
if $0<\tau<1$, and
\begin{equation}\label{eq:volume2}
\frac{1}{\con}\tau^{n}\leq |B_{\tau}(x,\w)|\leq \con\tau^{n}
\end{equation}
if $\tau\geq 1$. In particular, 
\begin{equation}\label{eq:doubling}
|B_{c\tau}(x,\w)|\leq \con^{2}c^{2n}|B_{\tau}(x,\w)|
\end{equation}
for all $\tau>0$ and $c\geq1$, and $(\Sp,d,\ud x\ud\w)$ is a doubling metric measure space. Moreover, the volume $|B_{\tau}(x,\w)|>0$ of the ball $B_{\tau}(x,\w)$ only depends on $\tau$, as follows from the invariance of $d$ under translations and rotations in $\Sp$. Hence 
\begin{equation}\label{eq:defmu}
\mu(\lambda):=|B_{\lambda^{-1/2}}(x,\w)|^{-1}\quad(\lambda>0)
\end{equation}
is independent of the choice of $(x,\w)\in\Sp$.

For $(x,\w)\in\Sp$ set
\[
\Gamma(x,\w):=\{(y,\eta)\in\Tp\setminus o\mid (y,\hat{\eta})\in B_{|\eta|^{-1/2}}(x,\w)\},
\]
and for a ball $B\subseteq\Sp$ set
\[
T(B):=\{(y,\eta)\in\Tp\setminus o\mid d((y,\hat{\eta}),B^{c})\geq |\eta|^{-1/2}\}.
\]
Note that the projection of $\Gamma$ onto the $\eta$ variable is approximately a paraboloid in the direction of $\w$, and the projection of $\Gamma\cap \{(y,\eta)\in \Tp\mid |\eta|=c\}$ onto the $x$ variable is an anisotropic ball around $x$, with anisotropy in the direction of $\w$, for each $c>0$. Similar statements hold for $T(B)$. Next, for $s\in\R$ and $F\in L^{2}_{\loc}(\Tp)$ set 
\begin{equation}\label{eq:Afunctional}
\begin{aligned}
\A_{s} F(x,\w):=&\Big(\int_{\Gamma(x,\w)}|F(y,\eta)|^{2}\frac{|\eta|^{2s}\ud y\ud\eta}{|B_{|\eta|^{-1/2}}(x,\w)|}\Big)^{1/2}\\
=&\Big(\int_{\Gamma(x,\w)}|F(y,\eta)|^{2}\mu(|\eta|)|\eta|^{2s}\ud y\ud\eta\Big)^{1/2},
\end{aligned}
\end{equation}
where $\mu$ is as in \eqref{eq:defmu}, and
\begin{equation}\label{eq:Cfunctional}
\mathcal{C}_{s}F(x,\w):=\sup_{B}\Big(\frac{1}{|B|}\int_{T(B)}|F(y,\eta)|^{2}|\eta|^{2s}\ud y\ud \eta\Big)^{1/2},
\end{equation}
where the supremum is taken over all balls $B\subseteq \Sp$ containing $(x,\w)$. 

\begin{definition}\label{def:tentspaces}
For $p\in[1,\infty)$, $s\in\R$, the \emph{(weighted) tent space} $T^{p}_{s}(\Sp)$ consists of all $F\in L^{2}_{\loc}(\Tp)$ such that $\A_{s} F\in L^{p}(\Sp)$, endowed with the norm
\[
\|F\|_{T^{p}_{s}(\Sp)}:=\|\A_{s} F\|_{L^{p}(\Sp)}.
\]
Also, $T^{\infty}_{s}(\Sp)$ consists of all $F\in L^{2}_{\loc}(\Tp)$ such that $\mathcal{C}_{s}F\in L^{\infty}(\Sp)$, with 
\[
\|F\|_{T^{\infty}_{s}(\Sp)}:=\|\mathcal{C}_{s}F\|_{L^{\infty}(\Sp)}.
\]
\end{definition}

\begin{remark}\label{rem:differenttent}
In~\cite{HaPoRo20} the tent space $T^{p}(\Sp)$, for $p\in[1,\infty)$, was defined to consist of all $F\in L^{2}_{\loc}(\Sp\times(0,\infty),\ud x\ud\w\ud\sigma/\sigma)$ such that 
\[
\Big(\int_{\Sp}\Big(\int_{0}^{\infty}\fint_{B_{\sqrt{\sigma}}(x,\w)}|F(y,\nu,\sigma)|^{2}\ud y\ud\nu\frac{\ud\sigma}{\sigma}\Big)^{p/2}\ud x\ud\w\Big)^{1/2}<\infty.
\]
At first sight this seems quite different from Definition \ref{def:tentspaces}, but by changing to a spherical coordinate system via $\pi(y,\eta):= (y,\hat{\eta},|\eta|^{-1})$, one obtains an isometric isomorphism $\pi^{*}:T^{p}(\Sp)\to T^{p}_{-n/2}(\Sp)$ via pull-back. A similar remark applies for $p=\infty$. Hence the spaces in Definition \ref{def:tentspaces} are weighted tent spaces over the spaces from~\cite{HaPoRo20}. Definition \ref{def:tentspaces} is more convenient for our purposes than the one in~\cite{HaPoRo20}, because in Section \ref{sec:parametrix} we will work with bicharacteristic flows associated to symbols that are merely asymptotically homogeneous. It is also natural from a phase space perspective, in that $T^{2}_{0}(\Sp)=L^{2}(\Tp)$ isometrically. Indeed, 
\begin{align*}
&\|F\|_{T^{2}_{0}(\Sp)}^{2}=\int_{\Sp}\int_{\Gamma(x,\w)}|F(y,\eta)|^{2}\mu(|\eta|)\ud y\ud\eta\ud x\ud\w\\
&=\int_{\Sp}\int_{\Tp}\ind_{B_{|\eta|^{-1/2}}(x,\w)}(y,\hat{\eta})|F(y,\eta)|^{2}|B_{|\eta|^{-1/2}}(y,\hat{\eta})|^{-1}\ud y\ud\eta\ud x\ud\w\\
&=\int_{\Tp}\int_{\Sp}\ind_{B_{|\eta|^{-1/2}}(y,\hat{\eta})}(x,\w)|B_{|\eta|^{-1/2}}(y,\hat{\eta})|^{-1}\ud x\ud\w|F(y,\eta)|^{2}\ud y\ud\eta\\
&=\int_{\Tp}|F(y,\eta)|^{2}\ud y\ud\eta=\|F\|_{L^{2}(\Tp)}^{2}.
\end{align*}
\end{remark}

For more on weighted tent spaces we refer the reader to~\cite{Amenta18}. There the spaces are weighted in a slightly different manner which in our setting would amount to replacing the factor $|\eta|^{2s}$ in \eqref{eq:Afunctional} by $|B_{|\eta|^{-1/2}}(x,\w)|^{-2s/n}$. The difference between these two choices is not significant, and for each $\kappa>0$ the two choices coincide on
\[
\{F\in T^{p}_{s}(\Sp)\mid F(y,\eta)=0\text{ for all }(y,\eta)\in\Tp\text{ with }|\eta|<\kappa\},
\]
due to \eqref{eq:volume} and \eqref{eq:volume2}. 

A minor role will be played in this article by a class of test functions on $\Tp$ and the associated distributions. Let $\Da(\Tp)$ consist of those $F\in L^{\infty}(\Tp)$ such that $[(x,\xi)\mapsto (1+|x|+\max(|\xi|,|\xi|^{-1}))^{N}F(x,\xi)]\in L^{\infty}(\Tp)$ for all $N\geq0$, endowed with the topology generated by the corresponding weighted $L^{\infty}$ norms. Let $\Da'(\Tp)$ be the space of continuous linear $G:\Da(\Tp)\to \C$, endowed with the topology induced by $\Da(\Tp)$. We denote the duality between $F\in\Da(\Tp)$ and $G\in\Da'(\Tp)$ by $\lb F,G\rb_{\Tp}$. If $G\in L^{1}_{\loc}(\Tp)$ is such that $F\mapsto \int_{\Tp}F(x,\xi)\overline{G(x,\xi)}\ud x\ud\xi$ defines an element of $\Da'(\Tp)$, then we write $G\in\Da'(\Tp)$. One has 
\begin{equation}\label{eq:testfunctions}
\Da(\Sp)\subseteq T^{p}_{s}(\Sp)\subseteq\Da'(\Tp)
\end{equation}
for all $p\in[1,\infty]$ and $s\in\R$, and the first embedding is dense for $p<\infty$. This follows from~\cite[Lemma 2.10]{HaPoRo20} and Remark \ref{rem:differenttent} for $s=-n/2$, and the statement for general $s$ then follows directly.

\subsection{Results about tent spaces}\label{subsec:tentspaceresults}

First, for use in Section \ref{sec:flow} we generalize \eqref{eq:Afunctional} slightly. For $\alpha>0$ and $(x,\w)\in\Sp$ we set
\begin{equation}
\Gamma_{\alpha}(x,\w):=\{(y,\eta)\in\Tp\setminus o\mid (y,\hat{\eta})\in B_{\alpha|\eta|^{-1/2}}(x,\w)\}.
\label{Gammadefn}\end{equation}
For $s\in\R$ and $F\in L^{2}_{\loc}(\Tp)$, let\[
\A_{s}^{\alpha} F(x,\w):=\Big(\int_{\Gamma_{\alpha}(x,\w)}|F(y,\eta)|^{2}\mu(|\eta|)|\eta|^{2s}\ud y\ud \eta\Big)^{1/2}.
\]
The following lemma shows that this modification of \eqref{eq:Afunctional} does not lead to any significant changes to the theory.

\begin{lemma}\label{lem:aperture}
Let $p\in[1,\infty)$, $s\in\R$ and $\alpha>0$. Then there exists an $\con>0$ such that the following holds for all $F\in L^{2}_{\loc}(\Tp)$. One has $F\in T^{p}_{s}(\Sp)$ if and only if $\A^{\alpha}_{s}F\in L^{p}(\Sp)$, with
\[
\frac{1}{\con}\|\A^{\alpha}_{s}F\|_{L^{p}(\Sp)}\leq \|F\|_{T^{p}_{s}(\Sp)}\leq \con\|\A^{\alpha}_{s}F\|_{L^{p}(\Sp)}.
\]
\end{lemma}
For a collection of $\alpha$ which are uniformly bounded from above and below, the corresponding constants $M=M_{\alpha}$ are uniformly bounded.

\begin{proof}
The case where $s=-n/2$ is contained in~\cite[Proposition 2.1]{Amenta14} (see also~\cite[Lemma 2.2]{Rozendaal21}), by Remark \ref{rem:differenttent} and \eqref{eq:doubling}. For a general $s\in\R$ one applies this special case with $F$ replaced by $F_{s}(x,\xi):=F(x,\xi)|\xi|^{s+\frac{n}{2}}$ for $(x,\xi)\in\Tp\setminus o$.
\end{proof}

One frequently needs to prove boundedness of integral operators on tent spaces. To this end, the following proposition is a powerful tool. The kernel bounds in this proposition were called \emph{off-singularity bounds} in~\cite{HaPoRo20}, where they were expressed in spherical coordinates, and similar bounds can be found in~\cite[Lemma 3.14]{Smith98a} and~\cite[Definition 5.3]{Geba-Tataru07}. We will only use the case where $\hat{\chi}$ is the identity map, although the more general case can be used for an alternative approach to the results about flows in Section \ref{sec:flow}. Here and throughout, we write
\begin{equation}\label{eq:Upsilon}
\Upsilon(t):=\min(t,t^{-1}),\quad t>0.
\end{equation}

\begin{proposition}\label{prop:offsingbound}
Let $\hat{\chi}:\Sp\to\Sp$ be bi-Lipschitz with respect to the metric $d$, and let $K:\Tp\times\Tp\to \C$ be measurable and such that for all $N\geq0$ there exists an $\con\geq0$ with
\[
|K(x,\xi,y,\eta)|\leq \con\Upsilon\big(\tfrac{|\xi|}{|\eta|}\big)^{N}\big(1+\rho^{-1}d((x,\hat{\xi}),\hchi(y,\hat{\eta}))^{2}\big)^{-N}
\]
for all $(x,\xi),(y,\eta)\in\Tp\setminus o$, where $\rho:=\min(|\xi|^{-1},|\eta|^{-1})$. Then the integral operator with kernel $K$ is bounded on $T^{p}_{s}(\Sp)$ for all $p\in[1,\infty]$ and $s\in\R$.
\end{proposition}

We note that $RF(x,\xi):=\int_{\Tp}K(x,\xi,y,\eta)F(y,\eta)\ud y\ud\eta$ is initially well defined for $F\in \Da(\Tp)$. A more precise wording of Proposition \ref{prop:offsingbound} would be that, using \eqref{eq:testfunctions}, $R$ extends uniquely to a bounded operator on $T^{p}_{s}(\Sp)$ for all $p\in[1,\infty)$ and $s\in\R$, and then by duality to $T^{\infty}_{s}(\Sp)$.

\begin{proof}
We first consider $s=-n/2$. By Remark \ref{rem:differenttent}, it suffices to show that $\pi_{*}R\pi^{*}$ is bounded on the space $T^{p}(\Sp)$ from~\cite{HaPoRo20}, where $\pi_{*}:=(\pi^{*})^{-1}:T^{p}_{-n/2}(\Sp)\to T^{p}(\Sp)$ is push-forward via $\pi$. One can check that $\pi_{*}R\pi^{*}$ is an integral operator with kernel $\tilde{K}$ given by
\[
\tilde{K}(x,\w,\sigma,y,\nu,\tau):=\tau^{-n}K(x,\w\sigma^{-1},y,\nu\tau^{-1})
\] 
for all $(x,\w,\sigma),(y,\nu,\tau)\in \Sp\times(0,\infty)$. Then
\[
|\tilde{K}(x,\w,\sigma,y,\nu,\tau)|\lesssim \tilde{\rho}^{-n}\Upsilon(\tfrac{\sigma}{\tau})^{N}(1+\tilde{\rho}^{-1}d((x,\w),\hchi(y,\nu))^{2})^{-N}
\]
for each $N\geq0$, where $\tilde{\rho}:=\min(\sigma,\tau)$. Now~\cite[Theorem 3.7]{HaPoRo20} concludes the proof for $s=-n/2$.

For general $s\in\R$ we note that the integral operator $R$ with kernel $K$ is bounded on $T^{p}_{s}(\Sp)$ if and only if the integral operator with kernel $K_{s}$ is bounded on $T^{p}_{n/2}(\Sp)$, where $K_{s}(x,\xi,y,\eta):=K(x,\xi,y,\eta)\big(\tfrac{|\xi|}{|\eta|}\big)^{s+\frac{n}{2}}$ for $(x,\xi),(y,\eta)\in\Tp\setminus o$. Moreover, for each $N\geq0$ one has
\begin{align*}
|K_{s}(x,\xi,y,\eta)|&\leq \con\Upsilon\big(\tfrac{|\xi|}{|\eta|}\big)^{N+|s+\frac{n}{2}|}(1+\rho^{-1}d((x,\hat{\xi}),\hchi(y,\hat{\eta}))^{2})^{-N-|s+\frac{n}{2}|}\big(\tfrac{|\xi|}{|\eta|}\big)^{s+\frac{n}{2}}\\
&\leq \con\Upsilon\big(\tfrac{|\xi|}{|\eta|}\big)^{N}(1+\rho^{-1}d((x,\hat{\xi}),\hchi(y,\hat{\eta}))^{2})^{-N},
\end{align*}
since $\Upsilon\big(\frac{|\xi|}{|\eta|}\big)\leq \frac{|\eta|}{|\xi|}$, which concludes the proof.
\end{proof}

\section{Wave packet transforms and Hardy spaces for Fourier integral operators}\label{sec:HpFIO}

In this section we introduce the wave packets and the wave packet transform that will be used in this article, and then we use them to define the Hardy spaces for Fourier integral operators.

\subsection{The wave packet transform}\label{subsec:wavepackettransform}

In this subsection we first introduce the relevant wave packets for this article, and we derive some of their properties. These wave packets coincide with those in~\cite{HaPoRo20} up to a normalization factor, which in turn arises by switching from spherical to Cartesian coordinates, and similar wave packets can be found in~\cite{Geba-Tataru07}. Then we introduce our wave packet transform, which in turn is a modified version of the transform from~\cite{HaPoRo20}. A similar transform was also used in~\cite{Geba-Tataru07}.

Throughout, fix a real-valued radial $\ph\in\Sw(\Rn)$ such that $\ph(\zeta)=1$ in a neighborhood of zero, and $\ph(\zeta)=0$ if $|\zeta|\geq1$. For $\xi\in\Rn$ set
\begin{equation}\label{eq:csigma}
c_{|\xi|}:=\Big(\int_{S^{n-1}}\ph(|\xi|^{1/2}(e_{1}-\nu))^{2}\ud\nu\Big)^{-1/2}.
\end{equation}
Note that the first basis vector $e_{1}$ can be replaced by any element of $S^{n-1}$. Fix a real-valued radial $\Psi\in C^{\infty}_{c}(\Rn)$ such that $\supp(\Psi)\subseteq\{\zeta\in\Rn\mid |\zeta|\in[\frac{1}{2},2]\}$ and 
\begin{equation}\label{eq:Psiintegral}
\int_{0}^{\infty}\Psi(\sigma\zeta)^{2}\frac{\ud \sigma}{\sigma}=1\quad(\zeta\neq0).
\end{equation}
Now, for $\xi,\zeta\in\Rn\setminus \{0\}$ let\begin{equation*}
\psi_{\xi}(\zeta)=|\xi|^{-n/2}c_{|\xi|}\Psi(|\xi|^{-1}\zeta)\ph(|\xi|^{1/2}(\hat{\zeta}-\hat{\xi}))
\end{equation*}
and $\psi_{\xi}(0):=0$. Finally, let
\begin{equation}
q(\zeta):=\Big(\frac{1}{|B_{1}(0)\setminus B_{1/2}(0)|}\int_{B_{1}(0)}\psi_{\xi}(\zeta)^{2}\ud\xi\Big)^{1/2}
\label{r-defn}\end{equation}
for $\zeta\in\Rn$. Then 
\begin{equation}
q(\zeta) = 0 \text{ if } |\zeta| \geq 2\quad\text{and}\quad q(\zeta) = |B_{1}(0)\setminus B_{1/2}(0)|^{-1/2}  \text{ if } |\zeta| \leq \tfrac{1}{2},
\label{r-support}\end{equation}
as follows from \eqref{r-defn} and the support property of $\Psi$. Moreover, $q\in C^{\infty}_{c}(\Rn)$, as can be checked by showing that the derivatives of $q$ vanish where $q(\zeta)=0$, using that $q$ is radial. We collect some properties of these wave packets, one of which involves the vector quantity $\Omega_{\xi}(\zeta)$, whose $j$th component ($1 \leq j \leq n$) is defined by 
\begin{equation}
(\Omega_{\xi}(\zeta))_j:=\partial_{\xi_j}\psi_{\xi}(\zeta)+\tfrac{|\zeta|}{|\xi|}\partial_{\zeta_j}\psi_{\xi}(\zeta)
\label{Omega-defn}\end{equation}
for $\xi\in\Rn\setminus 0$ and $\zeta\in\Rn$. Recall from Section \ref{subsec:notation} that we write $\partial_{\xi}=(\partial_{\xi_{1}},\ldots,\partial_{\xi_{n}})$ and $ \partial^{\alpha}_{\xi}=\partial^{\alpha_{1}}_{\xi_{1}}\ldots\partial^{\alpha_{n}}_{\xi_{n}}$, for $\xi=(\xi_{1},\ldots,\xi_{n})\in\Rn$ and $\alpha=(\alpha_{1},\ldots,\alpha_{n})\in\Z_{+}^{n}$.

\begin{lemma}\label{lem:packets}
For all $\xi\in\Rn\setminus\{0\}$ one has $\psi_{\xi}\in C^{\infty}_{c}(\Rn)$. Each $\zeta\in\supp(\psi_{\xi})$ satisfies $\frac{1}{2}|\xi|\leq |\zeta|\leq 2|\xi|$ and $|\hat{\zeta}-\hat{\xi}|\leq |\xi|^{1/2}$. Moreover, for each $\zeta\in\Rn\setminus\{0\}$ one has
\[
\int_{\Rn}\psi_{\xi}(\zeta)^{2}\ud\xi=1.
\]
For all $\alpha\in\Z_{+}^{n}$ and $\beta\in\Z_{+}$ there exists an $\con\geq0$ such that one has
\begin{align}
\label{eq:packetbounds1}|(\hat{\xi}\cdot \partial_{\zeta})^{\beta}\partial_{\zeta}^{\alpha}\psi_{\xi}(\zeta)|&\leq \con|\xi|^{-\frac{n+1}{4}-\frac{|\alpha|}{2}-\beta},\\
\label{eq:packetbounds2}|(\hat{\xi}\cdot \partial_{\zeta})^{\beta}\partial_{\zeta}^{\alpha}\partial_{\xi}\psi_{\xi}(\zeta)|&\leq \con|\xi|^{-\frac{n+3}{4}-\frac{|\alpha|}{2}-\beta},\\
\label{eq:packetbounds3}|(\hat{\xi}\cdot\partial_{\zeta})^{\beta}\partial_{\zeta}^{\alpha}\Omega_{\xi}(\zeta)|&\leq \con|\xi|^{-\frac{n+1}{4}-1-\frac{|\alpha|}{2}-\beta},
\end{align}
for all $\xi\in\Rn\setminus\{0\}$ and $\zeta\in\Rn$. 
\end{lemma}
\begin{proof}
The second statement follows immediately from the support conditions on $\Psi$ and $\ph$. The first statement follows subsequently from the smoothness of $\Psi$ and $\ph$. For the third statement, switch to spherical coordinates and use \eqref{eq:Psiintegral} to write
\begin{align*}
&\int_{\Rn}\psi_{\xi}(\zeta)^{2}\ud\xi=\int_{0}^{\infty}\int_{S^{n-1}}\psi_{\tau\nu}(\zeta)^{2}\tau^{n}\ud\nu\frac{\ud\tau}{\tau}=\int_{0}^{\infty}\int_{S^{n-1}}\psi_{\sigma^{-1}\nu}(\zeta)^{2}\sigma^{-n}\ud\nu\frac{\ud\sigma}{\sigma}\\
&=\int_{0}^{\infty}\Psi(\sigma\zeta)^{2}c_{\sigma^{-1}}^{2}\int_{S^{n-1}}\ph(\sigma^{-1/2}(\hat{\zeta}-\nu))^{2}\ud\nu\frac{\ud\sigma}{\sigma}=\int_{0}^{\infty}\Psi(\sigma\zeta)^{2}\frac{\ud\sigma}{\sigma}=1
\end{align*}
for $\zeta\neq0$. 

For \eqref{eq:packetbounds1}, \eqref{eq:packetbounds2} and \eqref{eq:packetbounds3}, by the second statement of the lemma, we may fix $\xi,\zeta\in\Rn\setminus\{0\}$ and suppose in the remainder that $\tfrac{1}{2}|\xi|\leq |\zeta|\leq 2|\xi|$ and $|\hat{\zeta}-\hat{\xi}|\leq |\xi|^{-1/2}$. We also note, for later use, that
\begin{equation}\label{eq:diffidentities}
\partial_{\zeta}|\zeta|=\frac{\zeta}{|\zeta|}=\hat{\zeta}\quad\text{and}\quad\partial_{\zeta_{i}}\hat{\zeta}_{j}=\partial_{\zeta_{i}}\frac{\zeta_{j}}{|\zeta|}=\frac{\delta_{ij}-\hat{\zeta}_{i}\hat{\zeta}_{j}}{|\zeta|}
\end{equation}
for $1\leq i,j\leq n$, where $\delta_{ij}$ is the Kronecker delta, and that
\begin{equation}\label{eq:hatidentity}
|\xi||\hat{\zeta}-\hat{\xi}|^{2}=2|\xi|(1-\hat{\zeta}\cdot\hat{\xi})=2(|\xi|-\hat{\zeta}\cdot\xi).
\end{equation}

Now, the bounds in \eqref{eq:packetbounds1} in fact follow directly from~\cite[Lemma 4.1]{HaPoRo20}, but we  sketch the argument here since it will be useful for \eqref{eq:packetbounds2} and \eqref{eq:packetbounds3} as well. Let $\kappa>0$ be such that $\ph(\eta)=1$ for $|\eta|\leq \kappa$, and set $E_{\xi}:=\{\nu\in S^{n-1}\mid |e_{1}-\nu|\leq \kappa|\xi|^{-1/2}\}$ and $F_{\xi}:=\{\nu\in S^{n-1}\mid |e_{1}-\nu|\leq |\xi|^{-1/2}\}$. Then 
\begin{equation}
|\xi|^{-\frac{n-1}{2}}\eqsim \int_{E_{\xi}}\ud\nu\leq \int_{S^{n-1}}\ph\big(|\xi|^{1/2}(e_{1}-\nu))^{2}\ud\nu\lesssim\int_{F_{\xi}}\ud\nu\eqsim |\xi|^{-\frac{n-1}{2}}.
\end{equation}
Hence $c_{|\xi|}\eqsim |\xi|^{(n-1)/4}$. Upon multiplication with $|\xi|^{-n/2}$, this yields the factor $|\xi|^{-(n+1)/4}$ in \eqref{eq:packetbounds1}. Moreover, 
\[
|\partial_{\zeta}^{\alpha}\Psi(|\xi|^{-1}\zeta)|=|\xi|^{-|\alpha|}|(\partial_{\zeta}^{\alpha}\Psi)(|\xi|^{-1}\zeta)|\lesssim |\xi|^{-|\alpha|}.
\]
Finally, for the derivatives of $\ph(|\xi|^{1/2}(\hat{\zeta}-\hat{\xi}))$ one writes $\hat{\xi}\cdot\partial_{\zeta}=( \hat{\xi}-\hat{\zeta})\cdot\partial_{\zeta}+\hat{\zeta}\cdot\partial_{\zeta}$. Since $\ph(|\xi|^{1/2}(\hat{\zeta}-\hat{\xi}))$ is positively homogeneous of degree $0$ in $\zeta$, one has $\hat{\zeta}\cdot\partial_{\zeta}\ph(|\xi|^{1/2}(\hat{\zeta}-\hat{\xi}))=0$. Hence \eqref{eq:diffidentities} and the assumptions on $\xi$ and $\zeta$ yield
\begin{align*}
&|(\hat{\xi}\cdot\partial_{\zeta})^{\beta}\partial_{\zeta}^{\alpha}\ph(|\xi|^{1/2}(\hat{\zeta}-\hat{\xi}))|\\
&\lesssim \sum_{\alpha'\leq \alpha,\beta'\leq \beta}|\xi|^{\frac{|\alpha'|+\beta'}{2}}\big|\big(((\hat{\xi}-\hat{\zeta})\cdot\partial_{\zeta})^{\beta'}\partial_{\zeta}^{\alpha'}\ph\big)(|\xi|^{1/2}(\hat{\zeta}-\hat{\xi}))\big|\,|\zeta|^{-|\alpha|-\beta}\lesssim |\xi|^{-\frac{|\alpha|}{2}-\beta}.
\end{align*}
All of this combined proves \eqref{eq:packetbounds1}.

For \eqref{eq:packetbounds2} and \eqref{eq:packetbounds3}, set $\tilde{\Psi}(t):=\Psi\big(\tfrac{1}{t}e_{1}\big)$ for $t>0$, and $\tilde{\ph}(t):=\ph(\sqrt{2t}e_{1})$ for $t\geq0$. Since $\Psi$ and $\ph$ are radial, we can then write $\partial_{\xi}\psi_{\xi}(\zeta)=I_{1}+I_{2}+I_{3}$, where
\begin{align*}
I_{1}&:=\partial_{\xi}(|\xi|^{-n/2}c_{|\xi|})\Psi(|\xi|^{-1}\zeta)\ph(|\xi|^{1/2}(\hat{\zeta}-\hat{\xi})),\\
I_{2}&:=|\xi|^{-n/2}c_{|\xi|}\partial_{\xi}\big(\tilde{\Psi}\big(\tfrac{|\xi|}{|\zeta|}\big)\big)\ph(|\xi|^{1/2}(\hat{\zeta}-\hat{\xi})),\\
I_{3}&:=|\xi|^{-n/2}c_{|\xi|}\Psi(|\xi|^{-1}\zeta)\partial_{\xi}(\tilde{\ph}(\tfrac{1}{2}|\xi||\hat{\zeta}-\hat{\xi}|^{2})).
\end{align*}
We will prove \eqref{eq:packetbounds2} by obtaining estimates for each of these terms separately.

First consider $I_{1}$. We use \eqref{eq:diffidentities} to write 
\[
I_{1}=-\tfrac{n}{2}|\xi|^{-1}\psi_{\xi}(\zeta)\hat{\xi}+|\xi|^{-n/2}\partial_{\xi}(c_{|\xi|})\Psi(|\xi|^{-1}\zeta)\ph(|\xi|^{1/2}(\hat{\zeta}-\hat{\xi})).
\]
It directly follows from \eqref{eq:packetbounds1} that
\begin{equation}\label{eq:boundI1one}
\begin{aligned}
&\big|(\hat{\xi}\cdot\partial_{\zeta})^{\beta}\partial_{\zeta}^{\alpha}\big(\partial_{\xi}(|\xi|^{-n/2})c_{|\xi|}\Psi(|\xi|^{-1}\zeta)\ph(|\xi|^{1/2}(\hat{\zeta}-\hat{\xi}))\big)\big|\\
&=\big|(\hat{\xi}\cdot\partial_{\zeta})^{\beta}\partial_{\zeta}^{\alpha}\big(\tfrac{n}{2}|\xi|^{-1}\psi_{\xi}(\zeta)\hat{\xi}\big)\big|\lesssim |\xi|^{-\frac{n+1}{4}-1-\frac{|\alpha|}{2}-\beta}
\end{aligned}
\end{equation}
for all $\alpha\in\Z_{+}^{n}$ and $\beta\in\Z_{+}$. Moreover, differentiating \eqref{eq:csigma} and using \eqref{eq:hatidentity} yields
\[
\partial_{\xi}(c_{|\xi|})=\tfrac{1}{2}c_{|\xi|}^{3}\int_{S^{n-1}}\ph(|\xi|^{1/2}(e_{1}-\nu))(\partial_{\xi}\ph)(|\xi|^{1/2}(e_{1}-\nu))\cdot (e_{1}-\nu)\ud\nu |\xi|^{-1/2}\hat{\xi}.
\]
As before, with $F_{\xi}:=\{\nu\in S^{n-1}\mid |e_{1}-\nu|\leq |\xi|^{-1/2}\}$, we obtain
\begin{align*}
|\partial_{\xi}(c_{|\xi|})|&\lesssim c_{|\xi|}^{3}\int_{F_{\xi}}|\ph(|\xi|^{1/2}(e_{1}-\nu))(\partial_{\xi}\ph)(|\xi|^{1/2}(e_{1}-\nu))|\,|e_{1}-\nu|\ud\nu |\xi|^{-1/2}\\
&\lesssim c_{|\xi|}^{3}\int_{F_{\xi}}|\xi|^{-1/2}\ud\nu|\xi|^{-1/2}\eqsim c_{|\xi|}^{3}|\xi|^{-(n-1)/2}|\xi|^{-1}.
\end{align*}
We already showed that $c_{|\xi|}\eqsim |\xi|^{(n-1)/4}$, so \eqref{eq:packetbounds1} yields
\begin{align*}
&\big|(\hat{\xi}\cdot\partial_{\zeta})^{\beta}\partial_{\zeta}^{\alpha}\big(|\xi|^{-n/2}\partial_{\xi}(c_{|\xi|})\Psi(|\xi|^{-1}\zeta)\ph(|\xi|^{1/2}(\hat{\zeta}-\hat{\xi}))\big)\big|\\
&=\frac{|\partial_{\xi}(c_{|\xi|})|}{c_{|\xi|}}|(\hat{\xi}\cdot\partial_{\zeta})^{\beta}\partial_{\zeta}^{\alpha}\psi_{\xi}(\zeta)\big|\lesssim |\xi|^{-\frac{n+1}{4}-1-\frac{|\alpha|}{2}-\beta}.
\end{align*}
Combined with \eqref{eq:boundI1one}, this shows that $|(\hat{\xi}\cdot\partial_{\zeta})^{\beta}\partial_{\zeta}^{\alpha}I_{1}|\lesssim |\xi|^{-\frac{n+1}{4}-1-\frac{|\alpha|}{2}-\beta}$.

For $I_{2}$ we use \eqref{eq:diffidentities}:
\begin{equation}\label{eq:I2}
I_{2}=|\xi|^{-n/2}c_{|\xi|}\tilde{\Psi}'\big(\tfrac{|\xi|}{|\zeta|}\big)\ph(|\xi|^{1/2}(\hat{\zeta}-\hat{\xi}))|\zeta|^{-1}\hat{\xi}.
\end{equation}
Now one can repeat the arguments that yielded \eqref{eq:packetbounds1}. The only difference is that $\Psi(|\xi|^{-1}\zeta)$ is replaced by $\tilde{\Psi}'\big(\frac{|\xi|}{|\zeta|}\big)$, and there is an additional factor of $|\zeta|^{-1}\hat{\xi}$. Differentiating $\tilde{\Psi}'\big(\frac{|\xi|}{|\zeta|}\big)$ yields factors of $-\frac{|\xi|}{|\zeta|^{2}}\hat{\zeta}$, and since $|\zeta|\eqsim |\xi|$, we obtain 
\begin{equation}\label{eq:I2bound}
|(\hat{\xi}\cdot\partial_{\zeta})^{\beta}\partial_{\zeta}^{\alpha}I_{2}|\lesssim |\xi|^{-\frac{n+1}{4}-1-\frac{|\alpha|}{2}-\beta}.
\end{equation}

Finally, for $I_{3}$ we use \eqref{eq:hatidentity} to write
\begin{align*}
I_{3}&=|\xi|^{-n/2}c_{|\xi|}\Psi(|\xi|^{-1}\zeta)\partial_{\xi}(\tilde{\ph}(|\xi|-\hat{\zeta}\cdot\xi))\\
&=|\xi|^{-n/2}c_{|\xi|}\Psi(|\xi|^{-1}\zeta)\tilde{\ph}'(|\xi|-\hat{\zeta}\cdot\xi)(\hat{\xi}-\hat{\zeta}).
\end{align*}
The only difference between this expression and $\psi_{\xi}(\zeta)$ is that $\ph(|\xi|^{1/2}(\hat{\zeta}-\hat{\xi}))$ is replaced by $\tilde{\ph}'(|\xi|-\hat{\zeta}\cdot\xi)(\hat{\xi}-\hat{\zeta})$. Nonetheless, one can repeat the arguments that yielded \eqref{eq:packetbounds1}, and the term $\hat{\xi}-\hat{\zeta}$ is responsible for an additional decay factor of $|\xi|^{-1/2}$. The resulting estimate is $|(\hat{\xi}\cdot\partial_{\zeta})^{\beta}\partial_{\zeta}^{\alpha}I_{3}|\lesssim |\xi|^{-\frac{n+3}{4}-\frac{|\alpha|}{2}-\beta}$, and combined with the estimates for $I_{1}$ and $I_{2}$, this yields \eqref{eq:packetbounds2}.

It remains to prove \eqref{eq:packetbounds3}. To this end, we decompose $\partial_{\zeta}\ph_{\xi}(\zeta)$ in a similar manner as was done for $\partial_{\xi}\ph_{\xi}(\zeta)$. One has $\partial_{\zeta}\psi_{\xi}(\zeta)=J_{1}+J_{2}$, where
\begin{align*}
J_{1}:=&|\xi|^{-n/2}c_{|\xi|}\partial_{\zeta}\big(\tilde{\Psi}\big(\tfrac{|\xi|}{|\zeta|}\big)\big)\ph(|\xi|^{1/2}(\hat{\zeta}-\hat{\xi})),\\
J_{2}:=&|\xi|^{-n/2}c_{|\xi|}\Psi(|\xi|^{-1}\zeta)\partial_{\zeta}(\tilde{\ph}(\tfrac{1}{2}|\xi||\hat{\zeta}-\hat{\xi}|^{2})).
\end{align*}
We calculate, using \eqref{eq:diffidentities}, 
\[
\frac{|\zeta|}{|\xi|}J_{1}=-|\xi|^{-n/2}c_{|\xi|}\tilde{\Psi}'\big(\tfrac{|\xi|}{|\zeta|}\big)\ph(|\xi|^{1/2}(\hat{\zeta}-\hat{\xi}))|\zeta|^{-1}\hat{\zeta}.
\]
The only difference between this and $-I_{2}$ is that the final $\hat{\xi}$ in \eqref{eq:I2} has been replaced by $\hat{\zeta}$. Nonetheless, it satisfies the same estimates as $I_{2}$ in \eqref{eq:I2bound}, namely $\big|(\hat{\xi}\cdot\partial_{\zeta})^{\beta}\partial_{\zeta}^{\alpha}\big(\frac{|\zeta|}{|\xi|}J_{1}\big)\big|\lesssim |\xi|^{-\frac{n+1}{4}-1-\frac{|\alpha|}{2}-\beta}$. Next, using \eqref{eq:diffidentities} and \eqref{eq:hatidentity} we calculate
\begin{align*}
\frac{|\zeta|}{|\xi|}J_{2}&=\frac{|\zeta|}{|\xi|}|\xi|^{-n/2}c_{|\xi|}\Psi(|\xi|^{-1}\zeta)\partial_{\zeta}(\tilde{\ph}(|\xi|-\hat{\zeta}\cdot\xi))\\
&=-\frac{|\zeta|}{|\xi|}|\xi|^{-n/2}c_{|\xi|}\Psi(|\xi|^{-1}\zeta)\tilde{\ph}'(|\xi|-\hat{\zeta}\cdot\xi)\frac{|\xi|}{|\zeta|}(\hat{\xi}-(\hat{\xi}\cdot\hat{\zeta})\hat{\zeta})\\
&=-I_{3}+|\xi|^{-n/2}c_{|\xi|}\Psi(|\xi|^{-1}\zeta)\tilde{\ph}'(|\xi|-\hat{\zeta}\cdot\xi)(1-\hat{\xi}\cdot\hat{\zeta})\hat{\zeta}.
\end{align*}
Denote the second term on the final line by $J_{3}$. Then one has $|(\hat{\xi}\cdot\partial_{\zeta})^{\beta}\partial_{\zeta}^{\alpha}J_{3}|\lesssim |\xi|^{-\frac{n+1}{4}-1-\frac{|\alpha|}{2}-\beta}$. This is shown in the same way as \eqref{eq:packetbounds1}, with the improvement of a factor of $|\xi|^{-1}$ over \eqref{eq:packetbounds1} coming from the term $(1-\hat{\xi}\cdot\hat{\zeta})\hat{\zeta}=\frac{1}{2}|\hat{\xi}-\hat{\zeta}|^{2}\lesssim |\xi|^{-1}$. So
\begin{align*}
\Omega_{\xi}(\zeta)&=I_{1}+I_{2}+I_{3}+\frac{|\zeta|}{|\xi|}J_{1}+\frac{|\zeta|}{|\xi|}J_{2}=I_{1}+I_{2}+\frac{|\zeta|}{|\xi|}J_{1}+J_{3},
\end{align*}
and our estimates for the four terms on the right conclude the proof.
\end{proof}

We are now ready to introduce our wave packet transform $W$. For $f\in \Sw'(\Rn)$ and $(x,\xi)\in\Tp$, it is given by
\begin{equation}\label{eq:deftransform}
Wf(x,\xi):=\begin{cases}\psi_{\xi}(D)f(x)&\text{if }|\xi|>1,\\
\ind_{[1/2,1]}(|\xi|)q(D)f(x)&\text{if }|\xi|
\leq 1\end{cases}
\end{equation}
where $q$ is defined in \eqref{r-defn}. 
Then $W:L^{2}(\Rn)\to L^{2}(\Tp)$ is an isometry, as follows from Lemma \ref{lem:packets}:
\begin{align*}
&\|Wf\|_{L^{2}(\Tp)}^{2}=\int_{\Rn}\int_{\Rn}|Wf(x,\xi)|^{2}\ud x\ud\xi\\
&=\int_{\Rn}|\wh{f}(\zeta)|^{2}\Big(\int_{\Rn\setminus B_{1}(0)}|\psi_{\xi}(\zeta)|^{2}\ud\xi+\int_{B_{1}(0)\setminus B_{1/2}(0)}|q(\xi)|^{2}\ud\xi\Big)\ud\zeta\\
&=\int_{\Rn}|\wh{f}(\zeta)|^{2}\int_{\Rn}|\psi_{\xi}(\zeta)|^{2}\ud\xi\ud\zeta=\int_{\Rn}|\wh{f}(\zeta)|^{2}\ud\zeta=\|f\|_{L^{2}(\Rn)}^{2}
\end{align*}
for all $f\in L^{2}(\Rn)$. Moreover, it follows from~\cite[Proposition 4.3]{HaPoRo20} that $W$ maps $\Sw(\Rn)$ into the class of test functions $\Da(\Tp)$ from Section \ref{subsec:definitions}, and that $W:\Sw'(\Rn)\to\Da'(\Tp)$. The adjoint $W^{*}$ of $W$ is given by
\begin{equation}\label{eq:defadjoint}
\begin{aligned}
W^{*}F(x)&=\!\int_{\Rn\setminus B_{1}(0)}\psi_{\eta}(D)F(\cdot,\eta)(x)\ud\eta+\!\int_{B_{1}(0)\setminus B_{1/2}(0)}q(D)F(\cdot,\eta)(x)\ud\eta
\end{aligned}
\end{equation}
for $F\in L^{2}(\Spp)$ and $x\in\Rn$. It has similar mapping properties with respect to test functions and distributions, and one has
\begin{equation}\label{eq:WstarW}
W^{*}Wf=f\quad(f\in\Sw'(\Rn).
\end{equation}

\subsection{Hardy spaces for Fourier integral operators}\label{subsec:Hardyspaces}

In this subsection we define the Hardy spaces for Fourier integral operators, as well as the associated Sobolev spaces, and we collect some of their basic properties. Recall the definition of the tent space $T^{p}_{0}(\Sp)$ from Definition \ref{def:tentspaces}.

\begin{definition}\label{def:HardyFIO}
Let $p\in[1,\infty]$. Then $\Hp$ consists of all $f\in\Sw'(\Rn)$ such that $Wf\in T^{p}_{0}(\Sp)$, with
\[
\|f\|_{\Hp}:=\|Wf\|_{T^{p}_{0}(\Sp)}\quad(f\in \Hp).
\]
Moreover, for $s\in\R$ the space $\Hps:=\lb D\rb^{-s}\Hp$ consists of all $f\in\Sw'(\Rn)$ such that $\lb D\rb^{s}f\in\Hp$, endowed with the norm
\[
\|f\|_{\Hps}:=\|\lb D\rb^{s}f\|_{\Hp}\quad(f\in \Hps).
\] 
\end{definition}

\begin{remark}\label{rem:otherHpFIO}
In~\cite{HaPoRo20} the Hardy spaces for Fourier integral operators were defined in a manner that, at first sight, appears different from the present definition. Nonetheless, the resulting spaces coincide. Indeed, in~\cite{HaPoRo20} the Hardy spaces for Fourier integral operators were defined to consist of all $f\in\Sw'(\Rn)$ such that $W'f\in T^{p}(\Sp)$, endowed with the norm $\|W'f\|_{T^{p}(\Sp)}$. Here $T^{p}(\Sp)$ is as in Remark \ref{rem:differenttent}, and for $(x,\w,\sigma)\in\Sp\times(0,\infty)$ one has
\[
W'f(x,\w,\sigma)=\begin{cases}\sigma^{-n/2}\psi_{\sigma^{-1}\w}(D)f(x)&\text{if }\sigma<1,\\
\frac{|B_{1}(0)\setminus B_{1/2}(0)|^{1/2}}{|S^{n-1}|^{1/2}}\ind_{[1,e]}(\sigma)q(D)f(x)&\text{if }\sigma\geq1.\end{cases}
\]
Using Remark \ref{rem:differenttent}, it is straightforward to check that the resulting space coincides with $\Hp$ as introduced in Definition \ref{def:HardyFIO}, with equivalence of norms. In fact, for any $f\in\Sw'(\Rn)$ with $\supp(\wh{f}\,)\subseteq\{\xi\in\Rn\mid |\xi|\geq 2\}$ one has $\|Wf\|_{T^{p}_{0}(\Sp)}=\|W'f\|_{T^{p}(\Sp)}$. 
\end{remark}

It follows from~\cite[Theorem 7.4]{HaPoRo20} that the following Sobolev embeddings hold for all $p\in[1,\infty]$ and $s\in\R$, with $s(p)=\frac{n-1}{2}|\frac{1}{2}-\frac{1}{p}|$:
\begin{equation}\label{eq:Sobolev}
\HT^{s+s(p),p}(\Rn)\subseteq\Hps\subseteq \HT^{s-s(p),p}(\Rn),
\end{equation}
and the exponents in these embeddings are sharp (see~\cite[Remark 7.9]{HaPoRo20} and~\cite[Remark 6.5]{FaLiRoSo19}). Here and throughout, for simplicity of notation we write
\begin{equation}\label{eq:classical}
\HT^{s,p}(\Rn):=
\begin{cases}
W^{s,p}(\Rn)&\text{if }p\in(1,\infty),\\
\lb D\rb^{-s}\HT^{1}(\Rn)&\text{if }p=1,\\
\lb D\rb^{-s}\bmo(\Rn)&\text{if }p=\infty,
\end{cases}
\end{equation}
with the associated norm $\|f\|_{\HT^{s,p}(\Rn)}:=\|\lb D\rb^{s}f\|_{\HT^{0,p}(\Rn)}$. By $\HT^{1}(\Rn)$ we denote the local Hardy space from~\cite{Goldberg79}. Given a function $q_{0}\in C^{\infty}_{c}(\Rn)$ such that $q_{0}(\zeta)=1$ for $|\zeta|\leq 2$, and writing $H^{1}(\Rn)$ for the classical Hardy space, it consists of those $f\in\Sw'(\Rn)$ such that $q_{0}(D)f\in L^{1}(\Rn)$ and $(1-q_{0})(D)f\in H^{1}(\Rn)$, with norm
\[
\|f\|_{\HT^{1}(\Rn)}:=\|q_{0}(D)f\|_{L^{1}(\Rn)}+\|(1-q_{0})(D)f\|_{H^{1}(\Rn)}.
\]
Moreover, recall that $\bmo(\Rn)$ consists of those $f\in\Sw'(\Rn)$ such that $q_{0}(D)f\in L^{\infty}(\Rn)$ and $(1-q_{0})(D)f\in \BMO(\Rn)$, endowed with the norm
\[
\|f\|_{\bmo(\Rn)}:=\|q_{0}(D)f\|_{L^{\infty}(\Rn)}+\|(1-q_{0})(D)f\|_{\BMO(\Rn)}.
\]

For later use we note that the following duality property holds: 
\begin{equation}\label{eq:duality}
(\Hps)^{*}=\HT^{-s,p'}_{FIO}(\Rn)
\end{equation}
for all $p\in[1,\infty)$ and $s\in\R$, where the duality pairing is given by the standard duality pairing between $f\in\Sw(\Rn)\subseteq \Hps$ and $g\in\HT^{-s,p'}_{FIO}(\Rn)\subseteq\Sw'(\Rn)$. This is shown in~\cite[Proposition 6.8]{HaPoRo20} for $s=0$, from which the case of general $s\in\R$ follows directly.

Next, we relate the Sobolev spaces $\Hps$ over the Hardy spaces for Fourier integral operators to the weighted tent spaces from Section \ref{sec:tentspaces}.

\begin{proposition}\label{prop:transformsandweights}
Let $p\in[1,\infty]$ and $s\in\R$. Then there exists a constant $\con>0$ such that the following holds for all $f\in\Sw'(\Rn)$. One has $f\in\Hps$ if and only if $Wf\in T^{p}_{s}(\Sp)$, and
\[
\frac{1}{\con}\|Wf\|_{T^{p}_{s}(\Sp)}\leq \|f\|_{\Hps}\leq \con\|Wf\|_{T^{p}_{s}(\Sp)}
\]
for all $f\in\Hps$. Moreover, $W^{*}:T^{p}_{s}(\Sp)\to\Hps$ is bounded and surjective.
\end{proposition}
\begin{proof}
First note that, by Proposition \ref{prop:oscilbdd} (see Remark \ref{rem:transformbounded}), 
\[
R_{1}F(x,\xi):=|\xi|^{-s}W\lb D\rb^{s}W^{*}F(x,\xi)
\]
and 
\[
R_{2}F(x,\xi):=|\xi|^{s}W\lb D\rb^{-s}W^{*}F(x,\xi)
\]
define bounded operators $R_{1},R_{2}:T^{p}_{s}(\Sp)\to T^{p}_{s}(\Sp)$. Now, by definition, an $f\in\Sw'(\Rn)$ satisfies $f\in \Hps$ if and only if $W\lb D\rb^{s}f\in T^{p}_{0}(\Sp)$, which in turn holds if and only if the function $(x,\xi)\mapsto |\xi|^{-s}W\lb D\rb^{s}f(x,\xi)$ is an element of $T^{p}_{s}(\Sp)$. Using \eqref{eq:WstarW}, for all $(x,\xi)\in\Tp$ we can write
\begin{equation}\label{eq:R1expression}
|\xi|^{-s}W\lb D\rb^{s}f(x,\xi)=|\xi|^{-s}\big(W\lb D\rb^{s}W^{*}W f\big)(x,\xi)=R_{1}Wf(x,\xi),
\end{equation}
and the boundedness of $R_{1}$ implies half of the first statement. On the other hand, $\lb D\rb^{s}$ commutes with $W^{*}W$, as follows from \eqref{eq:deftransform} and \eqref{eq:defadjoint}. Hence \eqref{eq:WstarW} yields
\[
Wf(x,\xi)=|\xi|^{s}|\xi|^{-s}\big(WW^{*}W\lb D\rb^{-s}\lb D\rb^{s}W^{*}Wf\big)(x,\xi)=R_{2}R_{1}Wf(x,\xi).
\]
Now the boundedness of $R_{2}$ and \eqref{eq:R1expression} imply the other half of the first statement.

Finally, the boundedness of $W^{*}:T^{p}_{s}(\Sp)\to \Hps$ follows from what we have already shown, since $WW^{*}\in\La(T^{p}_{s}(\Sp))$ by Corollary \ref{cor:transformbounded}. Similarly, for each $f\in\Hps$ one has $Wf\in T^{p}_{s}(\Sp)$ and $W^{*}Wf=f$, by \eqref{eq:WstarW}.
\end{proof}

As a corollary we obtain the following extension of~\cite[Proposition 6.7]{HaPoRo20}. It shows that the Sobolev spaces over the Hardy spaces for Fourier integral operators form a complex interpolation scale.

\begin{corollary}\label{cor:interpolation}
Let $p_{1},p_{2}\in[1,\infty]$, $s_{1},s_{2}\in\R$ and $\theta\in[0,1]$ be given. Let $p\in[1,\infty]$ be such that $\frac{1}{p}=\frac{1-\theta}{p_{1}}+\frac{\theta}{p_{2}}$, and set $s:=(1-\theta)s_{1}+\theta s_{2}$. Then
\begin{equation}\label{eq:interpolation}
[\HT^{s_{1},p_{1}}_{FIO}(\Rn),\HT^{s_{2},p_{2}}_{FIO}(\Rn)]_{\theta}=\Hps
\end{equation}
with equivalent norms.
\end{corollary}
\begin{proof}
As in~\cite[Theorem 2.1]{Amenta18}, using also Remark \ref{rem:differenttent}, one can show that 
\begin{equation}\label{eq:tentinter}
[T^{p_{1}}_{s_{1}}(\Sp),T^{p_{2}}_{s_{2}}(\Sp)]_{\theta}=T^{p}_{s}(\Sp).
\end{equation}
We use a slightly different weight than in~\cite{Amenta18}, but this does not affect the argument, which relies on viewing $T^{p}_{s}(\Sp)$ as a subspace of a weighted $L^{p}$ space with values in a Banach space, and on results about interpolation of weighted $L^{p}$ spaces. 

Next, it follows from Proposition \ref{prop:transformsandweights} that $WW^{*}:T^{p_{0}}_{t}(\Sp)\to T^{p_{0}}_{t}(\Sp)$ is bounded for all $p_{0}\in[1,\infty]$ and $t\in\R$, and that the range of $WW^{*}$ in $T^{p_{0}}_{t}(\Sp)$ is isomorphic to $\HT^{t,p_{0}}_{FIO}(\Rn)$. Now one can rely on known results about interpolation of complemented subspaces (see~\cite[Theorem 1.17.1.1]{Triebel78}) to conclude the proof.
\end{proof}

To conclude this section we note that the Hardy spaces for Fourier integral operators can be characterized in an alternative manner, using parabolic frequency cut-offs. For $w\in S^{n-1}$ and $\zeta\in\Rn$ set $\ph_{\w}(\zeta):=\int_{0}^{4}\psi_{\tau^{-1}\w}(\zeta)\frac{\ud\tau}{\tau^{1+n/2}}$. These functions localize to a paraboloid in the direction of $\w$ (see~\cite[Remark 3.3]{Rozendaal21}). Let $q_{0}\in C^{\infty}_{c}(\Rn)$ be such that $q_{0}(\zeta)=1$ for $|\zeta|\leq 2$. It was shown in \cite{Rozendaal21} for $1<p<\infty$, and in \cite{FaLiRoSo19} for $p=1$, that $f\in\Hps$ if and only if $q_{0}(D)f\in L^{p}(\Rn)$, $\ph_{\w}(D)f\in \HT^{s,p}(\Rn)$ for almost all $\w\in S^{n-1}$, and $(\int_{S^{n-1}}\|\ph_{\w}(D)f(x)\|^{p}_{\HT^{s,p}(\Rn)}\ud\w)^{1/p}<\infty$.
Moreover, 
\begin{equation}\label{eq:equivchar}
\|f\|_{\Hps}\eqsim \|q_{0}(D)f\|_{L^{p}(\Rn)}+\Big(\int_{S^{n-1}}\|\ph_{\w}(D)f(x)\|^{p}_{\HT^{s,p}(\Rn)}\ud\w\Big)^{1/p}
\end{equation}
for implicit constants independent of $f\in\Hps$. This characterization in turn allows one to apply tools from Littlewood-Paley theory to the Hardy spaces for Fourier integral operators. It will not play an explicit role in this article, but it was crucial in the proof in~\cite{Rozendaal20,Rozendaal22} of Theorem \ref{thm:roughpseudo} below, and that theorem is in turn essential for the main results of this article.

\section{Rough symbols}\label{sec:roughsymb}

In this section we collect some background on the rough symbol classes that are used throughout this article. We also introduce a symbol smoothing procedure which plays a key role in the proofs of our main results. Then we state our main results for the smooth and rough terms in this symbol smoothing procedure.

\subsection{Symbol classes and symbol smoothing}\label{subsec:roughsymbols}

Throughout this article, we work with symbols $a(x,\xi)$ that have limited smoothness in the $x$-variable. We will measure the degree of smoothness of such a symbol using spaces of continuous functions, and to simplify notation it will be convenient to make the following definition.

\begin{definition}\label{def:functions}
For $r>0$, write $r=l+s$ with $l\in\Z_{+}$ and $s\in(0,1]$. Then $\Cr(\Rn)$ consists of those $l$ times continuously differentiable $f:\Rn\to\C$ such that $\partial_{x}^{\alpha}f$ is bounded for each $\alpha\in\Z_{+}^{n}$ with $|\alpha|\leq l$, and if $|\alpha|=l$ then $\partial^{\alpha}_{x}f$ is additionally H\"{o}lder continuous with exponent $s$. For $f\in \Cr(\Rn)$ we set
\[
\|f\|_{\Cr(\Rn)}:=\max_{|\alpha|\leq l}\sup_{x\in\Rn}|\partial_{x}^{\alpha}f(x)|+\max_{|\alpha|=l}\sup_{x\neq y}\frac{|\partial_{x}^{\alpha}f(x)-\partial_{x}^{\alpha}f(y)|}{|x-y|^{s}}.
\]
\end{definition}

For $r\in\N$, $\Cr(\Rn)$ can be equivalently described as the space of $r$ times weakly differentiable $f:\Rn\to\C$ such that $\partial_{x}^{\alpha}f\in L^{\infty}(\Rn)$ for all $\alpha\in\Z_{+}^{n}$ with $|\alpha|\leq r$. This notation is inspired by the fact that 
\[
C^{r}(\Rn)\cap L^{\infty}(\Rn)\subseteq \Cr(\Rn)\subseteq C^{r-\veps}(\Rn)\cap L^{\infty}(\Rn)
\]
for all $\veps>0$, where the first inclusion is strict for $r\in\N$.

We can now introduce the rough symbol classes that appear in this article. Recall that, for $m\in\R$ and $\delta\in[0,1]$, $S^{m}_{1,\delta}$ is the class of $a\in C^{\infty}(\R^{2n})$ such that for all $\alpha,\beta\in\Z_{+}^{n}$ there exists an $\con_{\alpha,\beta}\geq0$ such that
\[
|\partial_{\xi}^{\alpha}\partial_{x}^{\beta}a(x,\xi)|\leq \con_{\alpha,\beta}\lb \xi\rb^{m-|\alpha|+|\beta|\delta}\quad(x,\xi\in\Rn).
\]

\begin{definition}\label{def:rough}
Let $r>0$, $m\in\R$ and $\delta\in[0,1]$. Then $\Cr S^{m}_{1,\delta}$ is the collection of $a:\R^{2n}\to\C$ such that for each $\alpha\in\Z_{+}^{n}$ there exists an $\con_{\alpha}\geq 0$ with the following properties:
\begin{enumerate}
\item\label{it:symbol1} For all $x,\xi\in\Rn$ one has $a(x,\cdot)\in C^{\infty}(\Rn)$ and $|\partial_{\eta}^{\alpha}a(x,\xi)|\leq \con_{\alpha}\lb\xi\rb^{m-|\alpha|}$. 
\item\label{it:symbol2} For all $\xi\in\Rn$ one has $\partial_{\xi}^{\alpha}a(\cdot,\xi)\in \Cr(\Rn)$ and
\[
\|\partial_{\xi}^{\alpha}a(\cdot,\xi)\|_{\Cr(\Rn)}\leq \con_{\alpha}\lb\xi\rb^{m-|\alpha|+r\delta},
\]
and for each $j\in\Z_{+}$ with $0\leq j\leq r$ one has
\[
\|\partial_{\xi}^{\alpha}a(\cdot,\xi)\|_{C^{j}_{-}(\Rn)}\leq \con_{\alpha}\lb\xi\rb^{m-|\alpha|+j\delta}.
\]
\end{enumerate} 
Moreover, $\A^{r}S^{m}_{1,\delta}$ consists of those $a\in S^{m}_{1,\delta}$ such that for all $\alpha\in\Z_{+}^{n}$ and $s\geq0$ there exists an $\con_{\alpha,s}\geq 0$ such that
\[
\|\partial_{\xi}^{\alpha}a(\cdot,\xi)\|_{C^{r+s}_{-}(\Rn)}\leq \con_{\alpha,s}\lb\xi\rb^{m-|\alpha|+\delta s}
\] 
for all $\xi\in\Rn$.
\end{definition}

Note that $\A^{r}S^{m}_{1,\delta}\subseteq S^{m}_{1,\delta}\cap \Cr S^{m}_{1,0}$ for all $r>0$, $m\in\R$ and $\delta\in[0,1]$. These symbols arise naturally when applying a symbol smoothing procedure to functions of limited regularity, cf.~Lemma \ref{lem:smoothing} and the proof of our main result in Section \ref{sec:divform}.

 Next, the pseudodifferential operator $a(x,D):\Sw(\Rn)\to\Sw'(\Rn)$ with symbol $a\in \Cr S^{m}_{1,\delta}$ is given by
\[
a(x,D)f(x):=\frac{1}{(2\pi)^{n}}\int_{\Rn}e^{ix\xi}a(x,\xi)\wh{f}(\xi)\ud\xi
\]
for $f\in\Sw(\Rn)$ and $x\in\Rn$. One says that $a\in \Cr S^{m}_{1,\delta}$ is \emph{elliptic} if there exist $\kappa,R>0$ such that $|a(x,\xi)|>\kappa|\xi|^{m}$ for all $x,\xi\in\Rn$ with $|\xi|\geq R$.

An $a\in \Cr S^{m}_{1,\delta}$ is \emph{homogeneous (of degree $m$) for $|\xi|\geq1$} if $a(x,\lambda\xi)=\lambda^{m}a(x,\xi)$ for all $x,\xi\in\Rn$ and $\lambda\geq1$ with $|\xi|\geq1$. 
Note that this holds if and only if $(\xi\cdot \partial_{\xi}-m)a(x,\xi)=0$ for $|\xi|\geq 1$. 
We find this notion of homogeneity more useful than complete homogeneity for all $\xi \neq 0$, as this is too rigid a requirement. For example, if $m=1$, then the only functions homogeneous of degree $1$ in $\Crtwo S^1_{1,1/2}$ are functions linear in $\xi$, while if $1 < m < 2$, there are no nonzero functions homogeneous of degree $m$, due to the blowup of second $\xi$-derivatives at $\xi = 0$. 

The following weaker version of homogeneity plays a key role in this article.

\begin{definition}\label{def:asymphom}
Let $r>0$, $m\in\R$, $\delta\in[0,1]$ and $b\in S^{m}_{1,\delta}$, and let $a\in \Cr S^{m}_{1,0}$ be homogeneous of degree $m$ for $|\xi|\geq1$. Then $b\in S^{m}_{1,\delta}$ is \emph{asymptotically homogeneous of degree $m$ with limit} $a$ if the following conditions hold:
\begin{subequations}\label{eq:asymphom}
\begin{align}
& [(x,\xi)\mapsto (\xi\cdot \partial_{\xi}-m)b(x,\xi)]\in S^{m-1}_{1,\delta};  \label{it:asymphom1}  \\
& a-b\in \Cr S^{m-1}_{1,\delta}.  \label{it:asymphom2} 
\end{align}
\end{subequations}
Moreover, $b$ is \emph{asymptotically homogeneous of degree $m$} if there exists an $a\in \Cr S^{m}_{1,0}$ that is homogeneous for $|\xi|\geq1$ such that $b$ is asymptotically homogeneous of degree $m$ with limit $a$.
\end{definition}

\begin{remark}\label{rem:homlimit}
Given \eqref{it:asymphom1}, one can show that the limit 
\begin{equation}
\lim_{t \to \infty} \frac{b(x, t\xi)}{t^m} \quad (|\xi| \geq 1)
\label{a-limit}\end{equation}
exists pointwise and defines a function $\tilde a(x, \xi)$ that is homogeneous of degree $m$ for $|\xi|\geq 1$, and that  
\[
\tilde a - b = O(\ang{\xi}^{m-1}).
\]
Since there can be at most one function $a$ that is homogeneous of degree $m$ and such that 
$a-b = O(\ang{\xi}^{m-1})$ for $|\xi|\geq1$, as implied by \eqref{it:asymphom2},  it follows that $a$ is given by \eqref{a-limit} for $|\xi| \geq 1$. It is in this sense that $a$ is a `limit' of the function $b$.  
\end{remark}

\begin{remark}\label{rem:asymphom} 
It may appear at first that either condition in this definition implies the other, but that is not the case. In one direction, although, using the previous remark, $a$ is determined for $|\xi| \geq 1$ by \eqref{a-limit}, it need not be the case that this limit, for given $b\in S^{m}_{1,\delta}$ should lie in the space $\Cr S^{m}_{1,0}$. In the other direction, \eqref{it:asymphom2} does not imply \eqref{it:asymphom1}, even though $(\xi\cdot \partial_{\xi}-m)a(x,\xi)=0$ for $|\xi|\geq 1$, since \eqref{it:asymphom2} only provides information about a finite number of $x$-derivatives of $a-b$. The subtlety is that, although $b$ is infinitely differentiable, the $a$ to which it is asymptotic may be rough. In fact, that is exactly what happens in the case of interest; see Proposition~\ref{prop:squareroot}. 
\end{remark}

We next describe a symbol smoothing procedure that decomposes a rough symbol into a sum of smooth part and a rough part with additional decay. Let $(\psi_{j})_{j=0}^{\infty}\subseteq C^{\infty}_{c}(\Rn)$ be a fixed Littlewood--Paley decomposition. That is, one has
\begin{equation}\label{eq:LitPaldecomp}
\sum_{j=0}^{\infty}\psi_{j}(\xi)=1\quad(\xi\in\Rn),
\end{equation}
$\psi_{0}(\xi)=0$ for $|\xi|>1$, $\psi_{1}(\xi)=0$ if $|\xi|\notin [1/2,2]$, and $\psi_{j}(\xi)=\psi_{1}(2^{-j+1}\xi)$ for all $j>1$ and $\xi\in\Rn$. In fact, we may suppose that 
\begin{equation}\label{eq:LitPalchoice}
\psi_{j}(\xi)=\psi_{0}(2^{-j}\xi)-\psi_{0}(2^{1-j}\xi)
\end{equation}
for all $j\geq1$ and $\xi\in\Rn$.  Let $a\in \Cr S^{m}_{1,0}$ for $r>0$ and $m\in\R$, and recall from Section \ref{subsec:wavepackettransform} that $\ph\in C^{\infty}_{c}(\Rn)$ satisfies $\phi\equiv1$ near zero. For $x,\xi\in\Rn$ set
\begin{equation}\label{eq:asharp}
a^{\sharp}(x,\xi):=\sum_{k=0}^{\infty}\big(\ph(2^{-k/2}D)a(\cdot,\xi)\big)(x)\psi_{k}(\xi)
\end{equation}
and
\[
a^{\flat}(x,\xi):=a(x,\xi)-a^{\sharp}(x,\eta)=\sum_{k=0}^{\infty}\big((1-\ph)(2^{-k/2}D)a(\cdot,\xi)\big)(x)\psi_{k}(\xi).
\]
The decomposition $a=a^{\sharp}+a^{\flat}$ then has the following properties.

\begin{lemma}\label{lem:smoothing}
Let $r>0$, $m\in\R$ and $a\in \Cr S^{m}_{1,0}$. Then $a^{\sharp}\in \A^{r}S^{m}_{1,1/2}$ and $a^{\flat}\in \Cr S^{m-r/2}_{1,1/2}$. If $a$ is elliptic then $a^{\sharp}$ is elliptic as well. If $a$ is homogeneous of degree $m$ for $|\xi|\geq1$ and if $r\geq2$, then $a^{\sharp}$ is asymptotically homogeneous of degree $m$ with limit $a$.
\end{lemma}
\begin{proof}
For the fact that $a^{\sharp}\in \A^{r}S^{m}_{1,1/2}$ see~\cite[Proposition 3.1.C]{Taylor91}, where the notation $\A^{r}_{0}S^{m}_{1,1/2}$ is used. To see that $a^{\flat}\in \Cr S^{m-r/2}_{1,1/2}$, one uses the estimate
\begin{equation}\label{eq:infestimate}
\begin{aligned}
&\|(1-\ph)(2^{-k/2}D)a(\cdot,\xi)\|_{L^{\infty}}\leq\!\sum_{j\geq-M+k/2}\!\|(1-\ph)(2^{-k/2}D)\psi_{j}(D)a(\cdot,\xi)\|_{L^{\infty}}\\
&\lesssim\!\sum_{j\geq-M+k/2}\!\|\psi_{j}(D)a(\cdot,\xi)\|_{L^{\infty}}\lesssim\!\sum_{j\geq-M+k/2}2^{-jr}\!\|a(\cdot,\xi)\|_{\Cr}\lesssim 2^{-kr/2}\lb\xi\rb^{m}
\end{aligned}
\end{equation} 
for some $M>0$ and all $k\geq0$, $\xi\in\Rn$. Choosing $k$ such that $\lb\xi\rb\eqsim 2^{k}$, and using similar estimates for the derivatives of $a$, one obtains the desired result.
For more details see~\cite[Proposition 1.3.E]{Taylor91}. 

If $a$ is elliptic then, since $a^{\flat}\in \Cr S^{m-r/2}_{1,1/2}$, it follows by writing $a^{\sharp}=a-a^{\flat}$ that $|a^{\sharp}(x,\xi)|\gtrsim \lb\xi\rb^{m}$ for $x,\xi\in\Rn$ with $|\xi|$ sufficiently large, i.e.~that $a^{\sharp}$ is elliptic. 

Finally, suppose that $a$ is homogeneous of degree $m$ for $|\xi|\geq1$, and that $r\geq2$. By what we have already shown, $a-a^{\sharp}=a^{\flat}\in \Cr S^{m-r/2}_{1,1/2}\subseteq \Cr S^{m-1}_{1,1/2}$, so it remains to show that $(x,\xi)\mapsto (\xi \cdot \partial_\xi - m) a^\sharp(x, \xi)$ is a symbol in $S^{m-1}_{1, 1/2}$. By \eqref{eq:LitPalchoice},
\[
\xi \cdot \partial_{\xi} \psi_{k}(\xi) = \xi \cdot \partial_{\xi} \psi_{0}(2^{-k} \xi) - \xi \cdot \partial_{\xi} \psi_{0}(2^{-(k-1)} \xi)
\]
for all $k\geq1$ and $\xi\in\Rn$. Let $x,\xi\in\Rn$ with $|\xi|\geq1$. Then, by assumption, $(\xi \cdot \partial_\xi - m)a(x,\xi)=0$. Hence, by definition of $a^{\sharp}$, 
\[
(\xi \cdot \partial_{\xi} - m) a^\sharp(x, \xi) = \sum_{k=0}^{\infty} \big((\ph(2^{-k/2}D) - \ph(2^{-(k+1)/2} D))a(\cdot,\xi)\big)(x) \xi \cdot \partial_{\xi} \big(\psi_{0}(2^{-k}\xi)\big).
\]
Now, using that $r\geq2$, it follows in a similar manner as in \eqref{eq:infestimate} that
\[
\|(\ph(2^{-k/2}D) - \ph(2^{-(k+1)/2} D))a(\cdot,\xi)\|_{L^{\infty}}\lesssim 2^{-k}\lb \xi\rb^{m}
\]
for all $k\geq0$. To conclude the proof, one again chooses $k$ such that $\lb\xi\rb\eqsim 2^{k}$ and uses similar estimates for the derivatives of $a$. Each $\xi$-derivative yields a factor $2^{-k}\eqsim\lb\xi\rb^{-1}$, and each $x$-derivative, applied to $2^{kn/2}\F^{-1}(\ph)(2^{k/2}\cdot)$, yields a factor $2^{k/2}\eqsim \lb\xi\rb^{1/2}$. The relevant estimates for $|\xi|<1$ are automatic.
\end{proof}

\subsection{Results for the smooth and rough terms}\label{subsec:smoothand rough}

This subsection contains our main results for the smooth and rough terms in the symbol splitting. We also prove a proposition that will allow us to work with first order equations instead of second order equations in the next sections, and we collect some lemmas on pseudodifferential calculus. 

The following theorem is our main result for smooth first order equations. It will be proved in Section \ref{sec:firstorder}, based in turn on results from Sections \ref{sec:flow} and \ref{sec:parametrix}.

\begin{theorem}\label{thm:mainsmooth}
Let $b\in \A^{2}S^{1}_{1,1/2}$ be real-valued, elliptic and asymptotically homogeneous of degree $1$ with real-valued limit $a\in \Crtwo S^{1}_{1,0}$. Then there exists a unique family $(\ex_{t})_{t\in\R}$ of operators on $\Sw'(\Rn)$ such that, for all $p\in[1,\infty)$, $s\in\R$, $k\in\Z_{+}$ and $t_{0}>0$, there exists an $\con>0$ such that the following properties hold for all $f\in\Hps$: 
\begin{enumerate}
\item\label{it:smooth1} $[t\mapsto \ex_{t}f]\in C^{k}(\R;\HT^{s-k,p}_{FIO}(\Rn))$;
\item\label{it:smooth2} $\|\partial_{t}^{k}\ex_{t}f\|_{\HT^{s-k,p}_{FIO}(\Rn)}\leq \con\|f\|_{\Hps}$ for all $t\in[-t_{0},t_{0}]$;
\item\label{it:smooth3} $\ex_{0}f=f$, and $D_{t}\ex_{t}f=b(x,D)\ex_{t}f$ for all $t\in\R$.
\end{enumerate}
\end{theorem}

Note that the $\ex_{t}$, $t\in\R$, are the solution operators to the equation $D_{t}w(t)=b(x,D)w(t)$. We prove uniqueness of the solutions to this equation because it is of independent interest. However, for the results in Sections \ref{sec:divform} and \ref{sec:standardform} we only need the existence of such a family.

The following theorem is our main result for pseudodifferential operators with rough coefficients. It was proved in~\cite{Rozendaal22}, based in turn on results from \cite{Rozendaal20}. Recall that $s(p)=\frac{n-1}{2}|\frac{1}{2}-\frac{1}{p}|$ for $p\in[1,\infty]$.

\begin{theorem}\label{thm:roughpseudo}
Let $r>0$, $m\in\R$, $p\in(1,\infty)$ and $a\in \Cr S^{m}_{1,1/2}$. For $\veps\in(0,r/2]$, set
\[
\rho:=\begin{cases}
0&\text{if }2s(p)<r/2,\\
2s(p)-r/2+\veps&\text{if }2s(p)\geq r/2.
\end{cases}
\]
Then 
\begin{equation}\label{eq:roughpseudo}
a(x,D):\HT^{s+m+\rho,p}_{FIO}(\Rn)\to \Hps
\end{equation}
for $-r/2+s(p)-\rho<s<r-s(p)$. If $r\in\N$, then \eqref{eq:roughpseudo} also holds for $s=r-s(p)$.
\end{theorem}
\begin{proof}
The result follows from~\cite[Theorem 5.1] {Rozendaal22} upon noting that, by \eqref{eq:Zygmund1} and \eqref{eq:Zygmund2}, $\Cr S^{m}_{1,1/2}\subseteq C^{r}_{*}S^{m}_{1,1/2}$ for all $r>0$, and $\Cr S^{m}_{1,1/2}\subseteq \HT^{r,\infty}S^{m}_{1,1/2}$ if $r\in\N$. 
\end{proof}

Next, we prove a proposition, about approximate square roots of the smooth term in the symbol splitting, which will be used in Sections \ref{sec:divform} and \ref{sec:standardform} to move from a second order operator to a first order operator to which Theorem \ref{thm:mainsmooth} applies.

\begin{proposition}\label{prop:squareroot}
Let $A\in \Crtwo S^{2}_{1,0}$ be non-negative, elliptic and homogeneous of degree $2$ for $|\xi|\geq1$. Then there exist a real-valued elliptic $b\in \A^{2}S^{1}_{1,1/2}$ and an $e\in S^{1}_{1,1/2}$ with the following properties:
\begin{enumerate}
\item\label{it:squareroot1} $A^{\sharp}(x,D)=b(x,D)^{2}+e(x,D)$.
\item\label{it:squareroot2} $b$ is asymptotically homogeneous of degree $1$ with real-valued limit $a\in \Crtwo S^{1}_{1,0}$ such that $a(x,\xi)=\sqrt{A(x,\xi)}$ for all $x,\xi\in\Rn$ with $|\xi|\geq 1$.
\end{enumerate}
\end{proposition}
\begin{proof}
Recall the formula for the product of two pseudodifferential operators (see e.g.~\cite[Proposition 0.3.C]{Taylor91}): for $m_{1},m_{2}\in\R$ and $a_{1}\in S^{m_{1}}_{1,1/2}$, $a_{2}\in S^{m_{2}}_{1,1/2}$, one has $a_{1}(x,D)a_{2}(x,D)=a_{3}(x,D)$, where $a_{3}\in S^{m_{1}+m_{2}}_{1,1/2}$ has the following asymptotic expansion:
\begin{align}\label{eq:asymprod}
&a_{3}(x,\xi) =  \sum_{\alpha\in\Z_{+}^{n},  |\alpha| \leq N}\frac{i^{|\alpha|}}{\alpha!}D_{\xi}^{\alpha}a_{1}(x,\xi)D_{x}^{\alpha}a_{2}(x,\xi) + e_N(x, \xi), \\
\nonumber e_N(x, \xi) &= \frac{i^{N+1}}{N!} \Big( \int_0^1 (1-t)^N e^{itD_\xi \cdot D_y} \ud t \,  (D_\xi \cdot D_y)^{N+1} \Big) a_1(x, \xi) a_2(y, \eta) \Big|_{x=y, \xi = \eta}
\end{align}
for all $N\in\Z_{+}$, where the remainder term $e_N$ is in general in $S^{m_{1} + m_{2} - (N+1)/2}_{1,1/2}$. By Lemma \ref{lem:smoothing}, $A^{\sharp}\in \A^{2}S^{2}_{1,1/2}$ is elliptic, and because $A^{\flat}\in \Crtwo S^{1}_{1,1/2}$ is lower order and $A$ is non-negative, there exist $\kappa,R>0$ such that $A^{\sharp}(x,\xi)=A(x,\xi)-A^{\flat}(x,\xi)\geq \kappa|\xi|^{2}$ for $x,\xi\in\Rn$ with $|\xi|>R$. Let $\chi\in C^{\infty}(\Rn)$ be real-valued and such that $\chi(\xi)=1$ for $|\xi|>2R$, and $\chi(\xi)=0$ for $|\xi|\leq R$. Now, for $x,\xi\in\Rn$ set 
\[
b(x,\xi):=\begin{cases}
\sqrt{A^{\sharp}(x,\xi)}\chi(\xi)&\text{if }|\xi|>R,\\
0&\text{if }|\xi|\leq R.
\end{cases}
\]
It is straightforward to check that $b\in A^{2}S^{1}_{1,1/2}$, and \eqref{eq:asymprod} with $N=0$ implies that in fact $A^{\sharp}(x,D)-b_{1}(x,D)^{2}=e(x,D)$ for some $e\in S^{1}_{1,1/2}$. Clearly $b$ is real-valued and elliptic. 

It remains to prove \eqref{it:squareroot2}. Since $A$ is non-negative, elliptic and homogeneous of degree $2$ for $|\xi|\geq1$, one has $A(x,\xi)\geq A(x,\hat{\xi})>0$ for all $x,\xi\in\Rn$ with $|\xi|\geq 1$. Hence $a(x,\xi):=\sqrt{A(x,\xi)}$ is well defined, and one can cut off $\sqrt{A(x,\xi)}$ smoothly for $|\xi|<1$ to extend $a$ to a real-valued element of $\Crtwo S^{1}_{1,0}$ that is homogeneous of degree $1$ for $|\xi|\geq1$ by definition. 

We have to show that $b$ is asymptotically homogeneous of degree $1$ with limit $a$. First note that for $x,\xi\in\Rn$ with $|\xi|>R$ one has, since $A^{\sharp}$ is asymptotically homogeneous of degree $2$ by Lemma \ref{lem:smoothing}, and because $A^{\sharp}(x,\xi)\geq \kappa \lb\xi\rb^{2}$,
\begin{align*}
|(\xi\cdot \partial_{\xi}-1)b(x,\xi)|&=\tfrac{1}{2}\Big|(A^{\sharp}(x,\xi))^{-1/2}\xi\cdot \partial_{\xi}A^{\sharp}(x,\xi)-\sqrt{A^{\sharp}(x,\xi)}\Big|\\
&=\tfrac{1}{2}|(A^{\sharp}(x,\xi))^{-1/2}(\xi\cdot \partial_{\xi}-2)A^{\sharp}(x,\xi)|\lesssim 1.
\end{align*}
Similar estimates hold for the derivatives of $(\xi\cdot \partial_{\xi}-1)b(x,\xi)$. This suffices to show that $(x,\xi)\mapsto (\xi\cdot \partial_{\xi}-1)b(x,\xi)$ is a symbol in $S^{0}_{1,1/2}$, as is needed for Definition \ref{def:asymphom} \eqref{it:asymphom1}.

Finally, to see that $a-b\in \Crtwo S^{0}_{1,\delta}$, it suffices to obtain the appropriate estimates for $x,\xi\in\Rn$ with $|\xi|>\max(R,1)$. Then we can use the definition of $a$ and $b$, the fact that $A^{\flat}\in \Crtwo S^{1}_{1,1/2}$ by Lemma \ref{lem:smoothing}, and that $A$ and $A^{\sharp}$ are elliptic, to estimate
\[
|a(x,\xi)-b(x,\xi)|=\Big|\sqrt{A(x,\xi)}-\sqrt{A^{\sharp}(x,\xi)}\Big|=\frac{|A^{\flat}(x,\xi)|}{\sqrt{A(x,\xi)}+\sqrt{A^{\sharp}(x,\xi)}}\lesssim 1.
\]
Similar estimates hold for the derivatives of $a-b$. This concludes the proof.
\end{proof}

We conclude this section with two basic lemmas about pseudodifferential calculus. The proofs are standard, but we include them for the convenience of the reader and because we work in a slightly nonstandard setting. The first lemma concerns the invertibility of elliptic pseudodifferential operators on Hardy spaces for Fourier integral operators. 

\begin{lemma}\label{lem:inverse}
Let $m\in\R$, and let $b\in S^{m}_{1,1/2}$ be real-valued and elliptic. Then for all $p\in[1,\infty]$ and $s\in\R$ there exists a $c>0$ such that $b(x,D)+ic:\HT^{s+m,p}_{FIO}(\Rn)\to\Hps$ is invertible.
\end{lemma}
\begin{proof}
It is a consequence (see~\cite[Lemma 3.2 and Remark 4.2]{Rozendaal20}) of Theorem 6.10 in~\cite{HaPoRo20} and the remark following it that each $a\in S^{m}_{1,1/2}$ satisfies $a(x,D):\HT^{s+m,p}_{FIO}(\Rn)\to\Hps$, and that the operator norm of $a(x,D)$ is bounded by finitely many of the $S^{m}_{1,1/2}$ seminorms of $a$. Since $b$ is real-valued and elliptic, $\{(b+ic)^{-1}\mid c\geq1\}$ is a uniformly bounded collection in $S^{-m}_{1,1/2}$, and $\{(b+ic)^{-1}(x,D)\mid c\geq1\}\subseteq\La(\Hps,\HT^{s+m,p}_{FIO}(\Rn))$ is uniformly bounded. Moreover, for each $c\geq 1$ it follows from \eqref{eq:asymprod} that 
\[
(b(x,D)+ic)(b+ic)^{-1}(x,D)=I+b_{c}(x,D)
\]
for some $b_{c}\in S^{0}_{1,1/2}$ with an asymptotic expansion given by
\begin{equation}\label{eq:asymptoticb}
b_{c}(x,\xi) = \sum_{\alpha\in\Z_{+}^{n}, 0 < |\alpha| \leq N}\frac{i^{|\alpha|}}{\alpha!}D_{\xi}^{\alpha}b(x,\xi)D_{x}^{\alpha}(b(x,\xi)+ic)^{-1} + e_{c, N}(x, \xi)
\end{equation}
where $e_{c, N}$ is an error term as in \eqref{eq:asymprod}. 
Now, for each $\alpha\in\Z_{+}$ with $|\alpha|=1$ we can use the ellipticity of $b\in S^{m}_{1,1/2}$ to write
\begin{align*}
&|D^{\alpha}_{\xi}b(x,\xi)D^{\alpha}_{x}(b(x,\xi)+ic)^{-1}|=\frac{|D^{\alpha}_{\xi}b(x,\xi)D^{\alpha}_{x}b(x,\xi)|}{|b(x,\xi)+ic|^{2}}\\
&\lesssim \frac{\lb \xi\rb^{m-1}\lb\xi\rb^{m+\frac{1}{2}}}{|b(x,\xi)+ic|^{2-\frac{1}{2m}}|b(x,\xi)+ic|^{\frac{1}{2m}}}\lesssim \frac{\lb \xi\rb^{2m-\frac{1}{2}}}{\lb\xi\rb^{2m-\frac{1}{2}}c^{\frac{1}{2m}}}=c^{-\frac{1}{2m}}
\end{align*}
for an implicit constant independent of $x,\xi\in\Rn$ and $c\geq 1$. Similar estimates hold for the derivatives of the terms with $|\alpha|=1$, and for the remaining terms in \eqref{eq:asymptoticb}. In particular, an analysis of the remainder term in \eqref{eq:asymptoticb} shows that the $S^{0}_{1,1/2}$ seminorms of $b_{c}$ tend to zero as $c\to\infty$. Hence $\|b_{c}(x,D)\|_{\La(\Hps)}\to0$ as $c\to\infty$, and the proof is concluded by choosing $c$ such that $\|b_{c}(x,D)\|_{\La(\Hps)}<1$.
\end{proof}

The following lemma will play a minor role in Section \ref{sec:firstorder}.

\begin{lemma}\label{lem:Aduality}
Let $r>0$, $m\in\R$, $\delta\in[0,1)$ and $b\in \A^{r}S^{m}_{1,\delta}$. Then there exists a $b_{1}\in \A^{r}S^{m}_{1,\delta}$ such that $b(x,D)^{*}=b_{1}(x,D)$. \end{lemma}
\begin{proof}
By~\cite[Proposition 0.3.B]{Taylor91}, there exists a $b_{1}\in S^{m}_{1,\delta}$ with $b(x,D)^{*}=b_{1}(x,D)$ and with an asymptotic expansion given by
\[
b_{1}(x,\xi)\sim \sum_{\alpha\in\Z_{+}^{n}}\frac{i^{|\alpha|}}{\alpha!}D^{\alpha}_{\xi}D^{\alpha}_{x}\overline{b(x,\xi)}.
\]
For $j\leq r$, the terms in this expansion with $|\alpha|=j$ belong to $\A^{r-j}S^{m-j}_{1,\delta}\subseteq \A^{r}S^{m}_{1,\delta}$, and for $j>r$ the terms with $|\alpha|=j$ belong to $S^{m-j+\delta(j-r)}_{1,\delta}\subseteq \A^{r}S^{m}_{1,\delta}$. For $j$ large enough, the remainder term is also an element of $S^{m-(1-\delta)(j+1)}_{1,\delta}\subseteq\A^{r}S^{m}_{1,\delta}$.
\end{proof}

\section{Operators in divergence form}\label{sec:divform}

In this section we state and prove our main result for operators in divergence form. From this main result we derive a corollary for equations on Sobolev spaces over $L^{p}(\Rn)$, and one about oscillatory operators which arise from spectral calculus. 

We consider the following differential operator in divergence form: 
\[
Lf(x)=\sum_{i,j=1}^{n}D_{i}(a_{i,j}D_{j}f)(x).
\]
Here $a_{i,j}:\Rn\to\R$ is a bounded real-valued function for each $1\leq i,j\leq n$, and there exists a $\kappa_{0}>0$ such that
\[
\sum_{i,j=1}^{n}a_{i,j}(x)\eta_{i}\eta_{j}\geq \kappa_{0}|\eta|^{2}
\]
for all $x,\eta\in\Rn$. Crucially, we assume that there exists an $r\geq 2$ such that $a_{i,j}\in \Cr(\Rn)$ for all $1\leq i,j\leq n$. We consider the following Cauchy problem:
\begin{equation}\label{eq:divwave}
\begin{aligned}
(D_{t}^{2}-L)u(t,x)&=F(t,x),\\
u(0,x)&=u_{0}(x),\\
\partial_{t}u(0,x)&=u_{1}(x).
\end{aligned}
\end{equation}
We will prove existence and uniqueness of solutions to \eqref{eq:divwave}. To this end, a key tool is the symbol smoothing procedure from Section \ref{subsec:roughsymbols}. More precisely, as in Lemma \ref{lem:smoothing}, for $1\leq i,j\leq n$ we write $a_{i,j}=a_{i,j}^{\sharp}+a_{i,j}^{\flat}$ with $a_{i,j}^{\sharp}\in \A^{r}S^{0}_{1,1/2}\subseteq \A^{2}S^{0}_{1,1/2}$ and $a_{i,j}^{\flat}\in \Cr S^{-r/2}_{1,1/2}$. Write $L=L_{1}+L_{2}$, where
\begin{equation}\label{eq:Ldecomp}
L_{1}:=\sum_{i,j=1}^{n}D_{i}a^{\sharp}_{i,j}(x,D)D_{j}\text{ and }L_{2}:=\sum_{i,j=1}^{n}D_{i}a^{\flat}_{i,j}(x,D)D_{j}.
\end{equation}
Also let
\begin{equation}\label{eq:defA}
A(x,\xi):=\sum_{i,j=1}^{n}a_{i,j}(x)\xi_{i}\xi_{j}\quad(x,\xi\in\Rn).
\end{equation}
The following proposition records some important properties of these operators.

\begin{proposition}\label{prop:Lproperties}
Let $p\in(1,\infty)$ be such that $2s(p)<r-1$. Then the following statements hold.
\begin{enumerate}
\item\label{it:Lprop1} There exist a real-valued elliptic $b\in\A^{2}S^{1}_{1,1/2}$ and an $e\in S^{1}_{1,1/2}$ such that 
\[
L_{1}=b(x,D)^{2}+e(x,D).
\]
One may choose $b$ to be asymptotically homogeneous of degree $1$ with real-valued limit $a$ given by $a(x,\xi)=\sqrt{A(x,\xi)}$ for $x,\xi\in\Rn$ with $|\xi|\geq1$.
\item\label{it:Lprop2} For all $0\leq  s< r-s(p)$ one has
\begin{equation}\label{eq:Aflatbound}
L_{2}:\Hps\to\Hpsm.
\end{equation}
If $r\in\N$, then \eqref{eq:Aflatbound} also holds for $s=r-s(p)$.
\item\label{it:Lprop3} For all $-r+s(p)+1< s< r-s(p)+1$ one has
\begin{equation}\label{eq:Lmap}
L:\Hps\to\Hpsmm.
\end{equation}
If $r\in\N$, then \eqref{eq:Aflatbound} also holds for $s=-r+s(p)+1$ and $s=r-s(p)+1$.
\end{enumerate}
\end{proposition}
\begin{proof}
\eqref{it:Lprop1}: The symbol of $L_{1}$ is given by
\[
(x,\xi)\mapsto\sum_{i,j=1}^{n}(D_{i}a^{\sharp}_{i,j})(x,\xi)\xi_{j}+\sum_{i,j=1}^{n}a^{\sharp}_{i,j}(x,\xi)\xi_{i}\xi_{j},
\]
where the first term is a symbol in $\A^{r-1}S^{1}_{1,1/2}\subseteq S^{1}_{1,1/2}$, and the second term is $A^{\sharp}(x,\xi)$ for $A$ as in \eqref{eq:defA}. It thus suffices to find a $b\in \A^{2}S^{1}_{1,1/2}$ with the stated properties, and an $e_{1}\in S^{1}_{1,1/2}$, such that $A^{\sharp}(x,D)=b(x,D)^{2}+e_{1}(x,D)$. But this follows directly from Proposition \ref{prop:squareroot}.

\eqref{it:Lprop2}: As noted before, by Lemma \ref{lem:smoothing} one has $a_{i,j}^{\flat}\in \Cr S^{-r/2}_{1,1/2}$ for all $1\leq i,j\leq n$. Now let $\rho$ and $\veps$ be as in Theorem \ref{thm:roughpseudo}, and note that $s-1\geq s-r/2+\rho$ for sufficiently small $\veps$. Here we used the assumption that $2s(p)<r-1$. Moreover, $-r/2+s(p)-\rho<0$. Hence Theorem \ref{thm:roughpseudo} yields
\[
a^{\flat}_{i,j}(x,D):\Hpsm\subseteq \HT^{s-r/2+\rho,p}_{FIO}(\Rn) \to\Hps
\]
for $s\in[0,r-s(p))\subseteq (-r/2+s(p)-\rho,r-s(p))$, and also for $s=r-s(p)$ if $r\in\N$. This yields the required statement, since $D_{i},D_{j}:\Hps\to \Hpsm$.

\eqref{it:Lprop3}: Combining \eqref{it:Lprop1} with Theorem \ref{thm:roughpseudo}, we obtain that $L_{1}:\Hps\to\Hpsmm$ for all $s\in\R$. By duality, cf.~\eqref{eq:duality}, one then also has $L_{1}^{*}:\Hps\to\Hpsmm$ for all $s\in\R$. Moreover, again by duality and because $s(p)=s(p')$, \eqref{it:Lprop2} yields $L_{2}^{*}:\Hps\to\Hpsm$ for $-r+s(p)+1< s\leq 1$. Since $L$ is self adjoint, we can now write $L=L_{1}+L_{2}=L_{1}^{*}+L_{2}^{*}$ to obtain \eqref{eq:Lmap} for all $-r+s(p)+1< s<r-s(p)$. To deal with the case where $r-s(p)<s< r-s(p)+1$, one additionally uses \eqref{it:Lprop2} to write
\[
L_{2}:\Hps\subseteq\HT^{r-s(p),p}_{FIO}(\Rn)\to\HT^{r-s(p)-1,p}_{FIO}(\Rn)\subseteq \Hpsmm.
\]
It follows from \eqref{it:Lprop2} that for $r\in\N$ one can include the endpoints $s=-r+s(p)+1$ and $s=r-s(p)+1$.
\end{proof}

We are now ready to prove our main result on existence and uniqueness of solutions to \eqref{eq:divwave}. 

\begin{theorem}\label{thm:divwave}
Let $p\in(1,\infty)$ be such that $2s(p)<r-1$, and let $-r+s(p)+1< s< r-s(p)$. 
Then for all $u_{0}\in\Hps$, $u_{1}\in \Hpsm$ and $F\in L^{1}_{\loc}(\R;\Hpsm)$, there exists a unique
\begin{equation}\label{eq:regularityu}
u\in C(\R;\Hps)\cap C^{1}(\R;\Hpsm)\cap W^{2,1}_{\loc}(\R;\Hpsmm)
\end{equation}
such that $u(0)=u_{0}$, $\partial_{t}u(0)=u_{1}$, and
\begin{equation}\label{eq:eqdistr}
(D_{t}^{2}-L)u(t)=F(t)
\end{equation}
in $\Hpsmm$ for almost all $t\in \R$. If $r\in\N$, then this statement holds for $-r+s(p)+1\leq s\leq r-s(p)$.
\end{theorem}

As is explained in Appendix \ref{app:regularity}, for $r>2$ the regularity assumptions on the coefficients can in fact be weakened slightly, by considering more refined forms of regularity. This is the reason that the result can be strengthened for $r\in\N$. Note that the conditions in Theorem \ref{thm:mainwave} are satisfied for all $p\in(1,\infty)$ whenever $r\geq \frac{n+1}{2}$. In particular, if $n\leq 3$ then the conclusion of Theorem \ref{thm:divwave} holds for all $p\in(1,\infty)$, given that $r\geq2$. The condition on $p$ arises from Theorem \ref{thm:roughpseudo}, through Proposition \ref{prop:Lproperties} \eqref{it:Lprop2}. It follows from \eqref{eq:regularityu} and Proposition \ref{prop:Lproperties} \eqref{it:Lprop3} that \eqref{eq:eqdistr} is well defined in $\Hpsmm$.

\begin{proof} We need to prove existence and uniqueness of solutions. We first prove existence for the interval $0\leq s< r-s(p)$ from Proposition \ref{prop:Lproperties} \eqref{it:Lprop2}. This is the main part of the proof. We then prove uniqueness for the smaller interval $1/2\leq  s< r-s(p)$, by solving an adjoint problem. Since $L$ is self adjoint, here we can rely on what we have already shown, and working with a smaller interval allows us to circumvent some difficulties related to the roughness of the coefficients. To get the full range of $s$, in the final step we use duality and some semigroup theory. 

\subsubsection{Existence for $0\leq s< r-s(p)$} The idea for this part of the proof is as follows. Write $L=L_{1}+L_{2}$ with $L_{1}$ and $L_{2}$ as in \eqref{eq:Ldecomp}. By Proposition \ref{prop:Lproperties} \eqref{it:Lprop2} one has $L_{2}:\Hps\to\Hpsm$, which will allow us to incorporate the term $L_{2}u$ into the inhomogeneous term in \eqref{eq:eqdistr}. Heuristically, by Duhamel's principle we may then consider the homogeneous second order equation $D_{t}^{2}w(t)=L_{1}w(t)$. Write $L_{1}=b(x,D)^{2}+e(x,D)$ as in Proposition \ref{prop:Lproperties} \eqref{it:Lprop1}. One has $e(x,D):\Hps\to\Hpsm$, so that we can again view $e(x,D)w$ as an inhomogeneous term. We can then use the solution operators $\{\ex_{t}\mid t\in\R\}$ to the first order equation associated with $b(x,D)$, arising from Theorem \ref{thm:mainsmooth}, to build a solution to our second order equation $D_{t}^{2}u(t)=Lu(t)$. However, it turns out to be more convenient to work with the operators $\tilde{\ex}_{t}=e^{-ct}\ex_{t}$, which satisfy $D_{t}\tilde{\ex}_{t}=(b(x,D)+ic)\tilde{\ex}_{t}$, where $c>0$ is such that $b(x,D)+ic$ is invertible, as in Lemma \ref{lem:inverse}. One encounters error terms when incorporating various factors into the inhomogeneous term; these error terms are dealt with using an iterative approximation scheme.

To make this idea precise, we first do some preliminary work. Write 
\begin{equation}\label{eq:Ldecomp2}
L=L_{1}+L_{2}=b(x,D)^{2}+e(x,D)+L_{2}
\end{equation}
as in Proposition \ref{prop:Lproperties}, with $b\in \A^{2}S^{1}_{1,1/2}$ real-valued, elliptic and asymptotically homogeneous of degree $1$ with real-valued limit $a\in \Crtwo S^{1}_{1,0}$, and $e\in S^{1}_{1,1/2}$. Let $\{\ex_{t}\mid t\in\R\}$ be the family of operators from Theorem \ref{thm:mainsmooth}, satisfying $D_{t}\ex_{t}f=b(x,D)\ex_{t}f$ for all $f\in\Hps$ and $t\in\R$. As in Lemma \ref{lem:inverse}, let $c>0$ be such that 
\[
\tilde{b}(x,D):=b(x,D)+ic:\Hps\to\Hpsm
\]
is invertible. Set $\tilde{\ex}_{t}:=e^{-ct}\ex_{t}$, so that
\begin{equation}\label{eq:deretilde}
\wt{e}_{0}f=f\quad\text{and}\quad D_{t}\tilde{\ex}_{t}f=(b(x,D)+ic)\tilde{\ex}_{t}f=\tilde{b}(x,D)\tilde{\ex}_{t}f.
\end{equation}
Using \eqref{eq:Ldecomp2}, we can now write
\begin{equation}\label{eq:factorA}
\begin{aligned}
\big(D_{t}^{2}-L\big)\tilde{\ex}_{t}f&=\big(D_{t}^{2}-b(x,D)^{2}-e(x,D)-L_{2}\big)\tilde{\ex}_{t}f\\
&=\big((D_{t}+b(x,D))(D_{t}-b(x,D))-e(x,D)-L_{2}\big)\tilde{\ex}_{t}f\\
&=\big((D_{t}+\tilde{b}(x,D))(D_{t}-\tilde{b}(x,D))+\tilde{e}(x,D)\big)\tilde{\ex}_{t}f=\tilde{e}(x,D)\tilde{\ex}_{t}f,
\end{aligned}
\end{equation}
where 
\begin{equation}\label{eq:etilde}
\tilde{e}(x,D):=-e(x,D)+2icb(x,D)-c^{2}-L_{2}.
\end{equation}
One obtains
\begin{equation}\label{eq:factorA2}
\begin{aligned}
\big(D_{t}^{2}-L\big)\tilde{\ex}_{-t}f&=\big((D_{t}-\tilde{b}(x,D))(D_{t}+\tilde{b}(x,D))+\tilde{e}(x,D)\big)\tilde{\ex}_{-t}f\\
&=\tilde{e}(x,D)\tilde{\ex}_{-t}f
\end{aligned}
\end{equation}
in the same manner.

By Proposition \ref{prop:Lproperties} \eqref{it:Lprop2}, $L_{2}:\Hps\to\Hpsm$. The same mapping property holds for $-e+2icb-c^{2}\in S^{1}_{1,1/2}$, without restrictions on $p$ and $s$, by Theorem \ref{thm:roughpseudo}. Hence
\begin{equation}\label{eq:mappinge}
\tilde{e}(x,D):\Hps\to\Hpsm
\end{equation}
as well. Note that \eqref{eq:mappinge} also holds for $s=r-s(p)$ if $r\in\N$, by Proposition \ref{prop:Lproperties} \eqref{it:Lprop2}.

Next, we construct approximate solutions to our equation. For $t\in\R$ set
\begin{equation}\label{eq:defcossin}
\cs_{t}:=\frac{\tilde{\ex}_{t}+\tilde{\ex}_{-t}}{2}\quad\text{and}\quad\sn_{t}:=\frac{\tilde{\ex}_{t}-\tilde{\ex}_{-t}}{2i}\tilde{b}(x,D)^{-1}.
\end{equation}
By \eqref{eq:deretilde}, for all $f\in\Hps$ and $g\in\Hpsm$ one has 
\begin{equation}\label{eq:cossinprop}
\cs_{0}f=f,\quad \partial_{t}\cs_{t}f|_{t=0}=0,\quad \sn_{0}g=0,\quad \partial_{t}\sn_{t}g|_{t=0}=g.
\end{equation}
Moreover, by \eqref{eq:factorA} and \eqref{eq:factorA2}, for all $t\in\R$ one has
\begin{equation}\label{eq:dercossin}
(D_{t}^{2}-L)\cs_{t}f=\tilde{e}(x,D)\cs_{t}f\quad\text{and}\quad(D_{t}^{2}-L)\sn_{t}g=\tilde{e}(x,D)\sn_{t}g.
\end{equation}
Finally, by \eqref{it:smooth1} and \eqref{it:smooth2} of Theorem \ref{thm:mainsmooth}, for all $t_{0}>0$ and $k\geq0$ one has
\begin{equation}\label{eq:cossinbounds}
\sup_{|t|\leq t_{0}}\|\partial_{t}^{k}\cs_{t}\|_{\La(\Hps,\HT^{s-k,p}_{FIO}(\Rn))}+\|\partial_{t}^{k}\sn_{t}\|_{\La(\Hpsm,\HT^{s-k,p}_{FIO}(\Rn))}<\infty,
\end{equation}
where the derivatives are taken in the strong operator topology, and these derivatives are strongly continuous as operators between the appropriate spaces. 

Next, we deal with the error term in our approximation and derive some of its properties. For $t\in\R$, set $v_{0}(t):=\tilde{e}(x,D)(\cs_{t}u_{0}+\sn_{t}u_{1})-F(t)$ and, recursively,
\[
v_{k+1}(t):=\int_{0}^{t}\tilde{e}(x,D)\sn_{t-\tau}v_{k}(\tau)\ud \tau
\] 
for $k\geq0$. By \eqref{eq:mappinge} and \eqref{eq:cossinbounds}, for each $t_{0}>0$ one has
\begin{align*}
&\|v_{0}\|_{L^{1}([-t_{0},t_{0}];\Hpsm)}\\
&\leq 2t_{0}\con_{0}(\|u_{0}\|_{\Hps}+\|u_{1}\|_{\Hpsm})+\|F\|_{L^{1}([-t_{0},t_{0}];\Hpsm)}
\end{align*}
for 
\[
\con_{0}:=\sup_{|t|\leq t_{0}}\|\tilde{e}(x,D)\cs_{t}\|_{\La(\Hps,\Hpsm)}+\|\tilde{e}(x,D)\sn_{t}\|_{\La(\Hpsm)}.
\]
Using that $\|v_{k+1}(t)\|_{\Hps}\leq \con_{0}\int_{0}^{t}\|v_{k}(\tau)\|_{\Hpsm}\ud \tau$ for all $k\geq0$ and $t\in[-t_{0},t_{0}]$, it follows by induction that
\[
\|v_{k+1}(t)\|_{\Hpsm}\leq \frac{\con_{0}^{k+1}|t|^{k}}{k!}\|v_{0}\|_{L^{1}([-t_{0},t_{0}];\Hpsm)}.
\]
In particular, $v:=\sum_{k=0}^{\infty}v_{k}$ is a well defined element of $L^{1}_{\loc}(\R;\Hpsm)$, with
\begin{equation}\label{eq:boundv}
\begin{aligned}
&\|v\|_{L^{1}([-t_{0},t_{0}];\Hpsm)}\leq \sum_{k=0}^{\infty}\|v_{k}\|_{L^{1}([-t_{0},t_{0}];\Hpsm)}\\
&\leq \con_{1}(\|F\|_{L^{1}([-t_{0},t_{0}];\Hpsm)}+\|u_{0}\|_{\Hps}+\|u_{1}\|_{\Hpsm})
\end{aligned}
\end{equation}
for each $t_{0}>0$ and some $\con_{1}=\con_{1}(t_{0})>0$. 

We are now ready to define our candidate solution $u$ and show that it has the required regularity. For $t\in\R$ set
\begin{equation}\label{eq:defu}
u(t):=\cs_{t}u_{0}+\sn_{t}u_{1}+\int_{0}^{t}\sn_{t-\tau}v(\tau)\ud \tau.
\end{equation}
It follows (see e.g.~\cite[Proposition 1.3.4]{ArBaHiNe11}) from the dominated convergence theorem, using \eqref{eq:boundv} and the strong continuity of $t\mapsto \cs_{t}$ and $t\mapsto \sn_{t}$, that $u\in C(\R;\Hps)$. Moreover, for $t\in\R$ and $h>0$ one can use the smoothness of $\sigma\mapsto \sn_{\sigma}$, and the fact that $\sn_{0}=0$, to write
\begin{align*}
&\frac{1}{h}\Big(\int_{0}^{t+h}\sn_{t+h-\tau}v(\tau)\ud\tau-\int_{0}^{t}\sn_{t-\tau}v(\tau)\ud\tau\Big)\\
&=\int_{0}^{t}\frac{\sn_{t+h-\tau}-\sn_{t-\tau}}{h}v(\tau)\ud\tau+\frac{1}{h}\int_{t}^{t+h}\sn_{t+h-\tau}v(\tau)\ud\tau\\
&=\int_{0}^{t}\frac{\sn_{t+h-\tau}-\sn_{t-\tau}}{h}v(\tau)\ud\tau+\int_{t}^{t+h}\frac{1}{h}\int_{0}^{t+h-\tau}\partial_{\sigma}\sn_{\sigma}v(\tau)\ud\sigma\ud\tau.
\end{align*}
Since $v\in L^{1}_{\loc}(\R;\Hpsm)$, \eqref{eq:cossinbounds} now implies that $u\in C^{1}(\R;\Hpsm)$ with
\begin{equation}\label{eq:firstderu}
\partial_{t}u(t)=\partial_{t}\cs_{t}u_{0}+\partial_{t}\sn_{t}u_{1}+\int_{0}^{t}\partial_{t}\sn_{t-\tau}v(\tau)\ud \tau.
\end{equation}
For the second derivative of $u$ we argue in a different manner, due to the low regularity of $v$. Let $h\in C^{\infty}_{c}(\R;C^{\infty}_{c}(\Rn))$. Then $t\mapsto \lb \partial_{t}\sn_{t-\tau}v(\tau),h(t)\rb_{\Rn}$ is a differentiable map for all $\tau\in\R$, and
\[
\partial_{t}\lb \partial_{t}\sn_{t-\tau}v(\tau),h(t)\rb_{\Rn}=\lb \partial_{t}\sn_{t-\tau}v(\tau),\partial_{t}h(t)\rb_{\Rn}+\lb \partial_{t}^{2}\sn_{t-\tau}v(\tau),h(t)\rb_{\Rn}.
\]
Write $\overline{v}(t):=\int_{0}^{t}\partial_{t}\sn_{t-\tau}v(\tau)\ud \tau$ for $t\in\R$. Then
\begin{align*}
&\lb\partial_{t}\overline{v},h\rb_{\R^{n+1}}=-\lb\overline{v},\partial_{t}h\rb_{\R^{n+1}}=-\int_{\R}\int_{0}^{t}\lb\partial_{t}\sn_{t-\tau}v(\tau),\partial_{t}h(t)\rb_{\Rn}\ud\tau\ud t\\
&=\int_{\R}\int_{0}^{t}\lb\partial_{t}^{2}\sn_{t-\tau}v(\tau),h(t)\rb_{\Rn}\ud\tau\ud t-\int_{\R}\int_{0}^{t}\partial_{t}\lb\partial_{t}\sn_{t-\tau}v(\tau),h(t)\rb_{\Rn}\ud\tau\ud t.
\end{align*}
Interchanging the order of integration in the final term and using the fundamental theorem of calculus, that $h\in C^{\infty}_{c}(\R;C^{\infty}_{c}(\Rn))$ and that $\partial_{t}\sn_{t}|_{t=0}=I$, we obtain that $\partial_{t}\overline{v}(t)=\int_{0}^{t}\partial_{t}^{2}\sn_{t-\tau}v(\tau)\ud \tau+v(t)$ as distributions on $\R\times\Rn$. By \eqref{eq:cossinbounds} with $k=2$, the right-hand side is an element of $L^{1}_{\loc}(\R;\Hpsmm)$. Hence 
\begin{equation}\label{eq:secondderu}
\partial_{t}^{2}u(t)=\partial_{t}^{2}\cs_{t}u_{0}+\partial_{t}^{2}\sn_{t}u_{1}+v(t)+\int_{0}^{t}\partial_{t}^{2}\sn_{t-\tau}v(\tau)\ud \tau
\end{equation}
in $L^{1}_{\loc}(\R;\Hpsmm)$, and $u\in W^{2,1}_{\loc}(\R;\Hpsmm)$. Note that at this point we cannot differentiate any more, due to the low regularity of $u$.

Finally, we show that $u$ solves \eqref{eq:divwave}. By \eqref{eq:cossinprop} and \eqref{eq:firstderu}, one has $u(0)=u_{0}$ and $\partial_{t} u(t)|_{t=0}=u_{1}$. Moreover, \eqref{eq:secondderu}, \eqref{eq:dercossin} and the definitions of $u$ and $v$ yield
\begin{align*}
(D_{t}^{2}-L)u(t)&=(D_{t}^{2}-L)\cs_{t}u_{0}+(D_{t}^{2}-L)\sn_{t}u_{1}-v(t)+\int_{0}^{t}(D_{t}^{2}-L)\sn_{t-\tau}v(\tau)\ud \tau\\
&=\tilde{e}(x,D)(\cs_{t}u_{0}+\sn_{t}u_{1})-\sum_{k=0}^{\infty}v_{k}(t)+\sum_{k=0}^{\infty}\int_{0}^{t}\tilde{e}(x,D)\sn_{t-\tau}v_{k}(\tau)\ud \tau\\
&=\tilde{e}(x,D)(\cs_{t}u_{0}+\sn_{t}u_{1})-v_{0}(t)=F(t)
\end{align*}
for almost all $t\in\R$. Thus $u$ is a solution with the required properties. 

If $r\in\N$, then for $s=r-s(p)$ the exact same construction, combined with the fact that \eqref{eq:mappinge} now holds for $s=r-s(p)$ as well, shows that $u$ is well defined and satisfies \eqref{eq:regularityu} and \eqref{eq:eqdistr}.

\subsubsection{Uniqueness for $1/2\leq  s< r-s(p)$} Next, we prove uniqueness, by solving an adjoint problem. 

We first reduce the problem somewhat. By linearity, we may suppose that
\[
u\in C(\R;\Hps)\cap C^{1}(\R;\Hpsm)\cap W^{2,1}_{\loc}(\R;\Hpsmm)
\]
satisfies $u(0)=\partial_{t}u(0)=0$ and $(D_{t}^{2}-L)u(t)=0$ in $\Hpsmm$ for almost all $t\in\R$, and show that $u=0$. For $t\in\R$ set
\[
u_{+}(t):=\ind_{(-\infty,0]}(t)u(t)\text{ and }u_{-}(t):=\ind_{[0,-\infty)}(t)u(t).
\]
By the assumptions on $u$, one has $u=u_{+}+u_{-}$,
\[
u_{+},u_{-}\in C(\R;\Hps)\cap C^{1}(\R;\Hpsm)\cap W^{2,1}_{\loc}(\R;\Hpsmm)
\]
and $(D_{t}^{2}-L)u_{+}(t)=(D_{t}^{2}-L)u_{-}(t)=0$ in $\Hpsmm$ for almost all $t\in\R$. In fact, the latter identity and Proposition \ref{prop:Lproperties} \eqref{it:Lprop3} imply that $u_{+},u_{-}\in C^{2}(\R;\Hpsmm)$. We need to show that $u_{+}=u_{-}=0$, and by symmetry, replacing $u(t)$ by $u(-t)$, it suffices to consider only $u_{+}$. We want to show that
\begin{equation}\label{eq:uniquetoshow}
\int_{\R}\lb u_{+}(t),F(t)\rb_{\Rn}\ud t=0
\end{equation}
for each $F\in C^{\infty}_{c}(\R;C^{\infty}_{c}(\Rn))$, and we will use that 
\begin{equation}\label{eq:uniquenesstouse}
\int_{\R}\lb u_{+}(t),(D_{t}^{2}-L)w(t)\rb_{\Rn}\ud t=\int_{\R}\lb (D_{t}^{2}-L)u_{+}(t),w(t)\rb_{\Rn}\ud t=0
\end{equation}
for all 
\begin{equation}\label{eq:regularitywneeded}
w\in C(\R;\HT^{2-s,p'}_{FIO}(\Rn))\cap W^{2,1}_{\loc}(\R;\HT^{-s,p'}_{FIO}(\Rn))
\end{equation}
with compact support in time, as follows from the properties of $u_{+}$. In fact, since $u_{+}(t)=0$ for $t\leq 0$, it suffices to assume only that there exists a $t_{0}>0$ such that $w(t)=0$ for all $t\geq t_{0}$. Note that the regularity of $w$ in \eqref{eq:regularitywneeded} is needed because of the limited regularity of $u_{+}$, and the support condition on $w$ is used in \eqref{eq:uniquenesstouse} to deal with the fact that $u$ need not be bounded or compactly supported in time.

For \eqref{eq:uniquetoshow} note that, by the first part of the proof, there exists a 
\begin{equation}\label{eq:regularityw}
w\in C(\R;\HT^{\sigma,p'}_{FIO}(\Rn))\cap C^{1}(\R;\HT^{\sigma-1,p'}_{FIO}(\Rn))\cap W^{2,1}_{\loc}(\R;\HT^{\sigma-2,p'}_{FIO}(\Rn))
\end{equation}
such that $(D_{t}^{2}-L)w(t)=F(t)$ in $\HT^{\sigma-2,p'}_{FIO}(\Rn)$, where $\sigma:=r-s(p)-\veps$ for some small $\veps>0$. Moreover, by the assumptions that $2s(p)<r-1$, $r\geq 2$ and $s\geq 1/2$, one has 
\begin{equation}\label{eq:sinequality}
r-s(p)=\tfrac{r}{2}-\tfrac{1}{2}-s(p)+\tfrac{r}{2}+\tfrac{1}{2}>\tfrac{r}{2}+\tfrac{1}{2}\geq \tfrac{3}{2}\geq 2-s.
\end{equation}
Hence, for $\eps$ sufficiently small, \eqref{eq:regularitywneeded} holds. Via \eqref{eq:uniquenesstouse}, this would seem to imply \eqref{eq:uniquetoshow}. However, a priori it is not clear whether $w(t)=0$ for some $t_{0}>0$ and all $t\geq t_{0}$, and the construction of a solution in \eqref{eq:defu} would not yield this. Hence we modify the argument slightly. 

Let $t_{0}>0$ be such that $F(t)=0$ for $t\geq t_{0}$. Using notation as in the first part of the proof, for $t\in\R$ set $v_{0}(t):=F(t)$ and, recursively,
\[
v_{k+1}(t):=-\int_{t}^{t_{0}}\tilde{e}(x,D)\sn_{t-\tau}v_{k}(\tau)\ud \tau
\] 
for $k\geq0$. As before, $v:=\sum_{k=0}^{\infty}v_{k}$ is a well defined element of $L^{1}_{\loc}(\R;\HT^{\sigma-1,p'}_{FIO}(\Rn))$, and one has $v(t)=0$ for $t\geq t_{0}$. Set $w(t):=\int_{t}^{t_{0}}\sn_{t-\tau}v(\tau)\ud\tau$ for $t\in\R$. As in the first part of the proof, it follows that $w$ satisfies \eqref{eq:regularityw}, $w(t)=0$ for $t\geq t_{0}$, and 
\[
(D_{t}^{2}-L)w(t)=v(t)+\int_{t}^{t_{0}}\wt{e}(x,D)\sn_{t-\tau}v(\tau)\ud\tau=F(t)
\]
for almost all $t$. Now \eqref{eq:regularityw} implies \eqref{eq:regularitywneeded}, and then \eqref{eq:uniquenesstouse} yields \eqref{eq:uniquetoshow}. 

Finally, since $\HT^{s_{1},p}_{FIO}(\Rn)\subseteq \HT^{s_{2},p}_{FIO}(\Rn)$ for $s_{1}\geq s_{2}$, if $r\in\N$ then one immediately obtains uniqueness for $s=r-s(p)$ from what we have already shown.

\subsubsection{Existence and uniqueness for all $-r+s(p)+1< s< r-s(p)$} We will use that the solution operators that we have constructed form a strongly continuous group which is independent of the choices of $p$ and $s$ considered until now. Then duality shows that these operators extend boundedly to the other values of $s$ as well, and basic results from semigroup theory conclude the proof. See~\cite{Engel-Nagel00} for some of the basics of semigroup theory that are used here.  

We first formulate our problem in terms of an evolution equation. For $-r+s(p)+1< s< r-s(p)$, set
\[
D_{s,p}:=\Big\{\!\left(\!\begin{array}{c}
\!w_{0}\!\\
\!w_{1}\!
\end{array}\!\right)\in \Hps\times\Hpsm\Big| w_{1}\in \Hps, Lw_{0}\in \Hpsm\Big\}.
\]
We claim that we must show that, for all $u_{0}\in\Hps$, $u_{1}\in \Hpsm$ and $F\in L^{1}_{\loc}(\R;\Hpsm)$, there exists a unique 
\begin{equation}\label{eq:regularityw2}
w\in C\big(\R;\Hps\times\Hpsm\big)
\end{equation}
such that for all $t\in\R$ one has 
\begin{equation}\label{eq:propertiesw1}
\int_{0}^{t}w(\tau)\ud \tau\in D_{s,p}
\end{equation}
and
\begin{equation}\label{eq:propertiesw2}
w(t)=\left(\!\begin{array}{c}
\!u_{0}\!\\
\!u_{1}\!
\end{array}\!\right)+\left(\!\begin{array}{cc}
\!0&I\!\\
\!-L&0\!
\end{array}\!\right)\int_{0}^{t}w(\tau)\ud\tau-\int_{0}^{t}\left(\!\begin{array}{c}
\!0\!\\
\!F(\tau)\!
\end{array}\!\right)\ud\tau
\end{equation}
in $\Hps\times\Hpsm$. 

To prove this claim, suppose that $w$ satisfies \eqref{eq:regularityw2}, \eqref{eq:propertiesw1} and \eqref{eq:propertiesw2}. Using the mapping properties of $L$ from Proposition \ref{prop:Lproperties} \eqref{it:Lprop3}, differentiate \eqref{eq:propertiesw2} to see that
\begin{equation}\label{eq:regularityw3}
w\in C\big(\R;\Hps\times\Hpsm\big)\cap W^{1,1}_{\loc}\big(\R;\Hpsm\times\Hpsmm\big)
\end{equation}
with 
\begin{equation}\label{eq:propertiesw4}
\partial_{t}w(t)=\left(\!\begin{array}{cc}
\!0&I\!\\
\!-L&0\!
\end{array}\!\right)w(t)-\left(\!\begin{array}{c}
\!0\!\\
\!F(t)\!
\end{array}\!\right)
\end{equation}
in $\Hpsm\times\Hpsmm$ for almost all $t\in\R$. Now let $u\in C(\R;\Hps)$ equal the first coordinate of $w$. Then $u(0)=u_{0}$ by \eqref{eq:propertiesw2}. Moreover, \eqref{eq:regularityw3}, \eqref{eq:propertiesw4} and \eqref{eq:propertiesw2} show that 
\[
u\in C(\R;\Hps)\cap C^{1}(\R;\Hpsm)\cap W^{2,1}_{\loc}(\R;\Hpsmm),
\]
$\partial_{t}u(0)=u_{1}$, and $(D_{t}^{2}-L)u(t)=F(t)$ in $\Hpsmm$ for almost all $t\in \R$.

Conversely, suppose that $u$ is as in the statement of the theorem, and set $w(t):=\left(\!\begin{array}{c}
\!u(t)\!\\
\!\partial_{t}u(t)\!
\end{array}\!\right)$ for $t\in\R$. Then $w$ satisfies \eqref{eq:regularityw3} and \eqref{eq:propertiesw4}. By integrating the latter identity one obtains \eqref{eq:propertiesw2}. Finally, \eqref{eq:propertiesw1} follows by rewriting \eqref{eq:propertiesw2} as
\[
w(t)-\left(\!\begin{array}{c}
\!u_{0}\!\\
\!u_{1}\!
\end{array}\!\right)+\int_{0}^{t}\left(\!\begin{array}{c}
\!0\!\\
\!F(\tau)\!
\end{array}\!\right)\ud\tau=\left(\!\begin{array}{cc}
\!0&I\!\\
\!-L&0\!
\end{array}\!\right)\int_{0}^{t}w(\tau)\ud\tau
\]
and noting that the left-hand side is an element of $\Hps\times\Hpsm$. This proves the claim. By Proposition \ref{prop:Lproperties} \eqref{it:Lprop3}, for $r\in\N$ the same argument works for $s= -r+s(p)+1$ and $s=r-s(p)$.

Having reformulated our problem in terms of an evolution equation, we now construct an appropriate strongly continuous group. For $1/2\leq s<r-s(p)$, $u_{0}\in\Hps$ and $u_{1}\in\Hpsm$, let
\[
u\in C(\R;\Hps)\cap C^{1}(\R;\Hpsm)\cap W^{2,1}_{\loc}(\R;\Hpsmm)
\]
be as in \eqref{eq:defu} for $F\equiv 0$. In this case one in fact has
\begin{equation}\label{eq:betterregu}
u\in C(\R;\Hps)\cap C^{1}(\R;\Hpsm)\cap C^{2}(\R;\Hpsmm),
\end{equation}
since $(D_{t}^{2}-L)u(t)=0$ in $\Hpsmm$ and because $Lu\in C(\R;\Hpsmm)$, by Proposition \ref{prop:Lproperties} \eqref{it:Lprop3}. Next, for $t\in\R$ set 
\[
S(t)\left(\!\begin{array}{c}
\!u_{0}\!\\
\!u_{1}\!
\end{array}\!\right):=\left(\!\begin{array}{c}
\!u(t)\!\\
\!\partial_{t}u(t)\!
\end{array}\!\right).
\]
It follows from \eqref{eq:cossinbounds}, \eqref{eq:boundv}, \eqref{eq:firstderu} and \eqref{eq:betterregu} that $(S(t))_{t\in\R}$ is a strongly continuous family of bounded linear operators on $X_{s,p}:=\Hps\times\Hpsm$. Moreover, by the uniqueness statement that we have already shown, $(S(t))_{t\in\R}$ is a strongly continuous group. This group does not depend on the choice of $p$ or $s$ within the allowed range, since the operators that occur in \eqref{eq:defu} are independent of those choices. The generator of $(S(t))_{t\in\R}$ is $\mathcal{A}:=\left(\!\begin{array}{cc}
\!0&I\!\\
\!-L&0\!
\end{array}\!\right)$, with domain
\begin{equation}\label{eq:domain}
\begin{aligned}
D(\mathcal{A})&=\Big\{\!\left(\!\begin{array}{c}
\!u_{0}\!\\
\!u_{1}\!
\end{array}\!\right)\in X_{s,p}\Big|\lim_{h\to0}\tfrac{1}{h}(S(h)-I)\left(\!\begin{array}{c}
\!u_{0}\!\\
\!u_{1}\!
\end{array}\!\right)\text{ exists in }X_{s,p}\Big\}\\
&=\Big\{\!\left(\!\begin{array}{c}
\!u_{0}\!\\
\!u_{1}\!
\end{array}\!\right)\in X_{s,p}\Big| u_{1}\in \Hps, Lu_{0}\in \Hpsm\Big\}=D_{s,p}
\end{aligned}
\end{equation}
on $X_{s,p}$, as follows from \eqref{eq:betterregu} and from the definition of $(S(t))_{t\in\R}$.

By duality (see~\cite[Sections I.5.14, II.2.5 and II.3.11]{Engel-Nagel00} and \eqref{eq:duality}), the transpose operators $(S(t)')_{t\in\R}$ form a strongly continuous group on $X^{*}_{s,p}=\HT^{-s,p'}_{FIO}(\Rn)\times\HT^{1-s,p'}_{FIO}(\Rn)$ with generator $\mathcal{A}'=\left(\!\begin{array}{cc}
\!0&-L\!\\
\!I&0\!
\end{array}\!\right)$
with domain
\[
D_{-s,p'}':=D(\mathcal{A}')=\Big\{\!\left(\!\begin{array}{c}
\!v_{0}\!\\
\!v_{1}\!
\end{array}\!\right)\in X^{*}_{s,p}\Big| v_{0}\in \HT^{1-s,p'}_{FIO}(\Rn), Lv_{1}\in \HT^{-s,p'}_{FIO}(\Rn)\Big\}.
\] 
By Proposition \ref{prop:Lproperties} \eqref{it:Lprop3}, $\Sw(\Rn)\times \Sw(\Rn)\subseteq\D_{s,p}\cap D_{-s,p'}'$, and for all $u_{0},u_{1}\in\Sw(\Rn)$ and $t\in\R$ one has $S(t)'\left(\!\begin{array}{c}
\!u_{0}\!\\
\!u_{1}\!
\end{array}\!\right)=\left(\!\begin{array}{c}
\!\partial_{t}u(t)\!\\
\!u(t)\!
\end{array}\!\right)$. Since $\Sw(\Rn)\times\Sw(\Rn)$ is dense in $X_{s,p}$ and $X_{s,p}^{*}$, it follows by interchanging the factors in $X^{*}_{s,p}=\HT^{-s,p'}_{FIO}(\Rn)\times\HT^{1-s,p'}_{FIO}(\Rn)$, as well as the roles of $p$ and $p'$ and of $s$ and $1-s$, that $(S(t))_{t\in\R}$ extends to a strongly continuous group on $\Hps\times\Hpsm$ for $-r+s(p)+1<s\leq 1/2$. Moreover, it follows from what we have already shown that, for all $-r+s(p)+1< s<r-s(p)$, its generator is $\mathcal{A}$ with domain $D_{s,p}$. 

Now~\cite[Proposition I.3.1.16]{ArBaHiNe11} implies that for all $-r+s(p)+1< s< r-s(p)$, $u_{0}\in\Hps$, $u_{1}\in \Hpsm$ and $F\in L^{1}_{\loc}(\R;\Hpsm)$, there exists a unique $w$ satisfying \eqref{eq:regularityw2}, \eqref{eq:propertiesw1} and \eqref{eq:propertiesw2}, given by
\begin{equation}\label{eq:defw}
w(t)=S(t)\left(\!\begin{array}{c}
\!u_{0}\!\\
\!u_{1}\!
\end{array}\!\right)+\int_{0}^{t}S(t-\tau)\left(\!\begin{array}{c}
\!0\!\\
\!F(\tau)\!
\end{array}\!\right)\ud\tau
\end{equation}
for $t\in\R$. This concludes the proof for $r\notin\N$. For $r\in\N$ the exact same arguments yield the required statement for $-r+s(p)+1\leq s\leq r-s(p)$.
\end{proof}

\begin{remark}\label{rem:additionalreg}
If, in Theorem \ref{thm:mainwave}, one additionally has $F\in L^{q}_{\loc}(\R;\Hpsmm)$ for some $q\in(1,\infty]$, then $u\in W^{2,q}_{\loc}(\R;\Hpsmm)$. Similarly, if $F\in C(\R;\Hpsmm)$ then $u\in C^{2}((-t_{0},t_{0});\Hpsmm)$. This follows from the identity $D_{t}^{2}u(t)=Lu(t)+F(t)$ and the fact that $Lu\in C(\R;\Hpsmm)$, by Proposition \ref{prop:Lproperties} \eqref{it:Lprop3}.
\end{remark}

\begin{remark}\label{rem:formula}
Note that we have in fact proved a stronger statement than Theorem \ref{thm:divwave}. Namely, we have shown that there exists a single collection $(S(t))_{t\in\R}$ that forms a strongly continuous group on $\Hps\times \Hpsm$ for all $p$ and $s$ from the theorem, such that for all $u_{0}\in\Hps$, $u_{1}\in \Hpsm$ and $F\in L^{1}_{\loc}(\R;\Hpsm)$, the function in \eqref{eq:defw} is the unique $w$ satisfying \eqref{eq:regularityw3} and \eqref{eq:propertiesw4}. Apart from merely proving existence and uniqueness of solutions to \eqref{eq:divwave}, this stronger statement also yields norm bounds for the solutions and consistency for initial data contained in $\Hps$ for multiple values of $p$ and $s$. 

Moreover, since $u$ is defined in \eqref{eq:defu} (see also \eqref{eq:defcossin}) in terms of the operators $\ex_{t}$ from Theorem \ref{thm:mainsmooth}, and because these operators can be expressed through a flow parametrix (see Remark \ref{rem:parametrix}), we have in fact also obtained a parametrix for \eqref{eq:eqdistr}.
\end{remark}

We can combine Theorem \ref{thm:divwave} with the Sobolev embeddings for $\Hp$ to obtain a corollary about wave equations with rough coefficients in the $L^{p}$-scale.

\begin{corollary}\label{cor:divLp}
Let $p\in(1,\infty)$ be such that $2s(p)<r-1$, and let $-r+s(p)+1< s< r-s(p)$. Then for all $u_{0}\in W^{s+s(p),p}(\Rn)$, $u_{1}\in W^{s-1+s(p),p}(\Rn)$ and $F\in L^{1}_{\loc}(\R;W^{s-1+s(p),p}(\Rn))$, there exists a
\[
u\in C(\R;W^{s-s(p),p}(\Rn))\cap C^{1}(\R;W^{s-1-s(p),p}(\Rn))\cap W^{2,1}_{\loc}(\R;W^{s-2-s(p),p}(\Rn))
\]
such that $u(0)=u_{0}$, $\partial_{t}u(0)=u_{1}$, and
\[
(D_{t}^{2}-L)u(t)=F(t)
\]
for almost all $t\in \R$. If, additionally, $s>1-r+3s(p)$, then such a $u$ is unique. If $r\in\N$, then the first statement holds for $-r+s(p)+1\leq s\leq r-s(p)$, and the second statement holds for $s\geq 1-r+3s(p)$.
\end{corollary}
\begin{proof}
The existence of such a $u$ is an immediate consequence of Theorem \ref{thm:divwave} and the Sobolev embeddings for $\Hp$ in \eqref{eq:Sobolev}. The uniqueness statement follows in the same manner, since such a $u$ satisfies
\[
u\in  C(\R;\HT^{s-2s(p),p}_{FIO}(\Rn))\cap C^{1}(\R;\HT^{s-1-2s(p),p}_{FIO}(\Rn))\cap W^{2,1}_{\loc}(\R;\HT^{s-2-2s(p),p}_{FIO}(\Rn)).
\]
For $-r+s(p)+1< s-2s(p)$, one now obtains uniqueness from Theorem \ref{thm:divwave}.
\end{proof}

\begin{remark}\label{rem:uniqueLp}
It is not clear whether the extra condition $3s(p)<r-1+s$ is needed for uniqueness in Corollary \ref{cor:divLp}. We choose not to pursue this matter any further, since under the assumptions of the corollary, Theorem \ref{thm:divwave} in fact yields a solution that has the stronger regularity property in \eqref{eq:regularityu}, and such a $u$ is unique.
\end{remark}

From Theorem \ref{thm:divwave} we can also derive a corollary about the boundedness of certain operators that arise in the spectral calculus of $L$. More precisely, $L$ is a positive operator on $L^{2}(\Rn)$, and therefore the operators $\cos(t\sqrt{L})$ and $\sin(t\sqrt{L})L^{-1/2}$ are well defined and bounded on $L^{2}(\Rn)=\HT^{2}_{FIO}(\Rn)$ for all $t\in\R$. Since $\Sw(\Rn)\subseteq \Hps\cap L^{2}(\Rn)$ is dense in $\Hps$ for all $p\in[1,\infty)$ and $s\in\R$, cf.~\cite[Proposition 6.6]{HaPoRo20}, it is natural to wonder whether these operators extend boundedly to $\Hps$ for $p\neq 2$. The following corollary gives conditions on $L$, $p$ and $s$ such that this is the case. 

\begin{corollary}\label{cor:spectral}
Let $p\in(1,\infty)$ be such that $2s(p)<r-1$. Then for all $-r+s(p)< s< r-s(p)$ and $t\in\R$, the operator
\begin{equation}\label{eq:cosbounded}
\cos(t\sqrt{L}):\Hps\to\Hps
\end{equation}
is bounded, with locally uniform bounds in $t$. Moreover, for all $-r+s(p)+1< s< r-s(p)$ and $t\in\R$, the operator
\begin{equation}\label{eq:sinbounded}
\sin(t\sqrt{L})L^{-1/2}:\Hpsm\to\Hps
\end{equation}
is bounded, with locally uniform bounds in $t$. If $r\in\N$, then \eqref{eq:cosbounded} holds for $-r+s(p)\leq s\leq r-s(p)$, and \eqref{eq:sinbounded} holds for $-r+s(p)+1\leq s\leq r-s(p)$.
\end{corollary}
\begin{proof}
It follows from abstract form theory (see e.g.~\cite[Chapter 6]{Kato95}) that $\sqrt{L}:L^{2}(\Rn)\to W^{-1,2}(\Rn)$ boundedly. Moreover, $L:L^{2}(\Rn)\to W^{-2,2}(\Rn)$, as follows directly as well as from Proposition \ref{prop:Lproperties} \eqref{it:Lprop3}. Now the spectral calculus shows that, for $f\in \Sw(\Rn)$, the function $t\mapsto u(t):=\cos(t\sqrt{L})f$ satisfies
\[
u\in C(\R;L^{2}(\Rn))\cap C^{1}(\R;W^{-1,2}(\Rn))\cap C^{2}(\R;W^{-2,2}(\Rn)),
\]
$u(0)=f$, $\partial_{t}u(0)=0$ and $(D_{t}^{2}-L)u(t)=0$. For $-r+s(p)+1\leq s\leq r-s(p)$ it thus follows from Remarks \ref{rem:additionalreg} and \ref{rem:formula} that
\[
u\in C(\R;\Hps)\cap C^{1}(\R;\Hpsm)\cap C^{2}(\R;\Hpsmm),
\]
and that for each $t_{0}>0$ there exists a $\con>0$, independent of $f$, such that $\|u(t)\|_{\Hps}\leq \con\|f\|_{\Hps}$ for all $t\in[-t_{0},t_{0}]$. Since $\Sw(\Rn)\subseteq\Hps$ is dense, this yields \eqref{eq:cosbounded} for $-r+s(p)+1< s< r-s(p)$. Then duality, cf.~\eqref{eq:duality}, extends \eqref{eq:cosbounded} to all $-r+s(p)< s< r-s(p)$. 

Applying the same reasoning to the function $t\mapsto v(t):=\sin(t\sqrt{L})L^{-1/2}f$, one obtains \eqref{eq:sinbounded} for all $-r+s(p)+1< s< r-s(p)$. The statements for $r\in\N$ follow in the same manner.
\end{proof}

\section{Operators in standard form}\label{sec:standardform}

This section contains our main result for operators in standard form.

The bounded real-valued functions $a_{i,j}:\Rn\to\R$, $1\leq i,j\leq n$, are as before. That is, we assume that there exists a $\kappa_{0}>0$ such that $\sum_{i,j=1}^{n}a_{i,j}(x)\xi_{i}\xi_{j}\geq \kappa_{0}|\xi|^{2}$ for all $x,\xi\in\Rn$, and that $a_{i,j}\in \Cr(\Rn)$ for all $1\leq i,j\leq n$ and some $r\geq2$. Let $A(x,D):=\sum_{i,j=1}^{n}a_{i,j}(x)D_{i}D_{j}$. Our main result for such differential operators in standard form is then as follows.

\begin{theorem}\label{thm:mainwave}
Let $p\in(1,\infty)$ be such that $2s(p)<r-1$, and let $-r+s(p)+2< s< r-s(p)+1$. Then for all $u_{0}\in\Hps$, $u_{1}\in \Hpsm$ and $F\in L^{1}_{\loc}(\R;\Hpsm)$, there exists a unique
\[
u\in C(\R;\Hps)\cap C^{1}(\R;\Hpsm)\cap W^{2,1}_{\loc}(\R;\Hpsmm)
\]
such that $u(0)=u_{0}$, $\partial_{t}u(0)=u_{1}$, and
\begin{equation}\label{eq:eqdistr2}
(D_{t}^{2}-A(x,D))u(t)=F(t)
\end{equation}
in $\Hpsmm$ for almost all $t\in\R$. If $r\in\N$, then this statement holds for $-r+s(p)+2\leq s\leq r-s(p)+1$.
\end{theorem}

 Note that the condition $p$ is identical to that in Theorem \ref{thm:divwave}, but the Sobolev interval for $s$ is shifted by one. As was the case for Theorem \ref{thm:divwave}, for $r>2$ the regularity assumptions on the coefficients can be weakened, cf.~Appendix \ref{app:regularity}. 

As in Corollary \ref{cor:divLp}, the Sobolev embeddings and Theorem \ref{thm:mainwave} combine to yield the following corollary on the $L^{p}$ regularity of wave equations in standard form.

\begin{corollary}\label{cor:mainLp}
Let $p\in(1,\infty)$ be such that $2s(p)<r-1$, and let $-r+s(p)+2\leq s\leq r-s(p)+1$. Then for all $u_{0}\in W^{s+s(p),p}(\Rn)$, $u_{1}\in W^{s-1+s(p),p}(\Rn)$ and $F\in L^{1}_{\loc}(\R;W^{s-1+s(p),p}(\Rn))$, there exists a
\[
u\in C(\R;W^{s-s(p),p}(\Rn))\cap C^{1}(\R;W^{s-1-s(p),p}(\Rn))\cap W^{2,1}_{\loc}(\R;W^{s-2-s(p),p}(\Rn))
\]
such that $u(0)=u_{0}$, $\partial_{t}u(0)=u_{1}$, and
\[
(D_{t}^{2}-A(x,D))u(t)=F(t)
\]
for almost all $t\in \R$. If, additionally, $s> -r+3s(p)+2$, then such a $u$ is unique. If $r\in\N$, then the first statement holds for $-r+s(p)+1\leq s\leq r-s(p)$, and the second statement holds for $s\geq -r+3s(p)+1$.
\end{corollary}

\begin{proof}[Proof of Theorem \ref{thm:mainwave}]
The proof is analogous to that of Theorem \ref{thm:divwave}, and we just summarize some key steps here. For details, we refer to the earlier preprint version \cite{Hassell-Rozendaal20extra} of the present article. 

First apply the symbol smoothing procedure from Lemma \ref{lem:smoothing}, to write $a_{i,j}=a_{i,j}^{\sharp}+a_{i,j}^{\flat}$ for each $1\leq i,j\leq n$, with $a_{i,j}^{\sharp}\in \A^{r}S^{0}_{1,1/2}\subseteq \A^{2}S^{0}_{1,1/2}$ and $a_{i,j}^{\flat}\in \Cr S^{-r/2}_{1,1/2}$. Equivalently, applying this decomposition to the symbol $A$ of $A(x,D)$ yields
\begin{equation}\label{eq:Asharpflat}
A^{\sharp}(x,D)=\sum_{i,j=1}^{n}a^{\sharp}_{i,j}(x,D)D_{i}D_{j}\text{ and }A^{\flat}(x,D)=\sum_{i,j=1}^{n}a^{\flat}_{i,j}(x,D)D_{i}D_{j}.
\end{equation}
Also set $B_{1}:=\sum_{i,j=1}^{n}D_{i}D_{j}a^{\sharp}_{i,j}(x,D)$ and $B_{2}:=\sum_{i,j=1}^{n}D_{i}D_{j}a^{\flat}_{i,j}(x,D)$, so that $A(x,D)^{*}=B_{1}+B_{2}$. Next, let $b\in\A^{2}S^{1}_{1,1/2}$ be as in Proposition \ref{prop:Lproperties} \eqref{it:Lprop1}. Then Proposition \ref{prop:squareroot} yields $e_{1},e_{2}\in S^{1}_{1,1/2}$ such that
\[
A^{\sharp}(x,D)=b(x,D)^{2}+e_{1}(x,D)\text{ and }B_{1}=b(x,D)^{2}+e_{2}(x,D).
\]
Moreover, by Theorem \ref{thm:roughpseudo}, the operators $A^{\flat}(x,D)$ and $B_{2}$ have similar mapping properties as the operator $L_{2}$ in Proposition \ref{prop:Lproperties} \eqref{it:Lprop2}, albeit with Sobolev intervals shifted by $1$ and $-1$, respectively. By duality, similarly shifted versions of the mapping properties in Proposition \ref{prop:Lproperties} \eqref{it:Lprop3} hold for $A(x,D)$ and $A(x,D)^{*}$.

Now, one first proves existence of solutions to \eqref{eq:eqdistr2} for a restricted Sobolev range, the interval $1\leq  s< r-s(p)+1$. This part of the argument is completely analogous to that in Theorem \ref{thm:divwave}, with the only difference being that the Sobolev interval is shifted by $1$ when incorporating the factor $A^{\flat}(x,D)u$ into the inhomogeneous term. Then, with an eye towards deriving uniqueness of solutions by solving an adjoint problem, one proves in the same manner the existence of solutions to
\begin{equation}\label{eq:eqdistr3}
(D_{t}^{2}-A(x,D)^{*})u(t)=F(t)
\end{equation}
for $-1\leq  s< r-s(p)-1$. With this second existence result in hand, one can prove uniqueness of solutions to \eqref{eq:eqdistr2} for $3/2\leq  s< r-s(p)+1$. Working with this smaller interval again allows us to circumvent some subtleties that arise from the roughness of the coefficients. In the same manner, one proves uniqueness of solutions to \eqref{eq:eqdistr3} for $-1/2\leq s< r-s(p)-1$. Finally, having established existence and uniqueness of solutions to both \eqref{eq:eqdistr2} and \eqref{eq:eqdistr3} in restricted Sobolev ranges, one extends the range of existence and uniqueness for \eqref{eq:eqdistr2} to $-r+s(p)+2< s< r-s(p)+1$ using duality and semigroup theory.
\end{proof}

\begin{remark}\label{rem:lowerorder}
In this article we consider pure second order operators, but our techniques also allow for lower order perturbations. Indeed, in the proof of Theorem \ref{thm:divwave}, one can include in the error term $\tilde{e}(x,D)$ from \eqref{eq:etilde} an additional factor of the form $\sum_{i=1}^{n}b_{i}(x)D_{i}+c_{0}(x)$, as long as the mapping property $\tilde{e}(x,D):\Hps\to\Hpsm$ from \eqref{eq:mappinge} still holds. Using results about multiplication operators on $\Hps$ from \cite{Rozendaal22}, one can then determine conditions on $p$ and $s$ that guarantee solvability and uniqueness of the corresponding wave equation in $\Hps$. These conditions will depend on the regularity of the lower order terms.

We leave the details of the various cases to the reader. We do note that, although operators in standard form are first order perturbations of operators in divergence form, and vice versa, one cannot directly deduce Theorem \ref{thm:mainwave} from Theorem \ref{thm:divwave} in this manner, or the other way around. Indeed, doing so would lead to a smaller range of Sobolev indices. Instead, we used the specific structure of the operators in Theorems \ref{thm:divwave} and \ref{thm:mainwave} to obtain well-posedness for these larger Sobolev intervals.
\end{remark}

\section{First order equations}\label{sec:firstorder}

In this section we reduce Theorem \ref{thm:mainsmooth}, our main result for smooth first order equations, to the following theorem about parametrices for such equations. This theorem will be proved in Sections \ref{sec:flow} and \ref{sec:parametrix}.

\begin{theorem}\label{thm:parametrix}
Let $b\in \A^{2}S^{1}_{1,1/2}$ be real-valued, elliptic and asymptotically homogeneous of degree $1$ with real-valued limit $a \in \Crtwo S^{1}_{1,0}$. Then there exists a family $(U_{t})_{t\in\R}$ of operators on $\Sw'(\Rn)$ such that, for all $p\in[1,\infty)$, $s\in\R$ and $t_{0}>0$, there exists an $\con>0$ such that the following properties hold for all $f\in\Hps$:
\begin{enumerate}
\item\label{it:par1} $[t\mapsto U_{t}f]\in C^{k}(\R;\HT^{s-k,p}_{FIO}(\Rn))$ for $k\in\{0,1\}$;
\item\label{it:par2} $\|\partial_{t}^{k}U_{t}f\|_{\HT^{s-k,p}_{FIO}(\Rn)}\leq \con\|f\|_{\Hps}$ for $k\in\{0,1\}$ and $t\in[-t_{0},t_{0}]$;
\item\label{it:par3} $U_{0}f=f$, 
\begin{equation}\label{eq:boundederror}
\|(D_{t}-b(x,D))U_{t}f\|_{\Hps}\leq \con\|f\|_{\Hps}
\end{equation}
for all $t\in[-t_{0},t_{0}]$, and $[t\mapsto(D_{t}-b(x,D))U_{t}f]\in C(\R;\Hps)$.
\end{enumerate}
\end{theorem}

We will now use Theorem \ref{thm:parametrix} to prove Theorem \ref{thm:mainsmooth}, the statement of which we recall here.

\begin{theorem}\label{thm:mainsmooth2}
Let $b\in \A^{2}S^{1}_{1,1/2}$ be real-valued, elliptic and asymptotically homogeneous of degree one with real-valued limit $a\in \Crtwo S^{1}_{1,0}$. Then there exists a unique family $(\ex_{t})_{t\in\R}$ of operators on $\Sw'(\Rn)$ such that, for all $p\in[1,\infty)$, $s\in\R$, $k\in\Z_{+}$ and $t_{0}>0$, there exists an $\con>0$ such that the following properties hold for all $f\in\Hps$:
\begin{enumerate}
\item\label{it:smooth21} $[t\mapsto \ex_{t}f]\in C^{k}(\R;\HT^{s-k,p}_{FIO}(\Rn))$;
\item\label{it:smooth22} $\|\partial_{t}^{k}\ex_{t}f\|_{\HT^{s-k,p}_{FIO}(\Rn)}\leq \con\|f\|_{\Hps}$ for all $t\in[-t_{0},t_{0}]$;
\item\label{it:smooth23} $\ex_{0}f=f$, and $D_{t}\ex_{t}f=b(x,D)\ex_{t}f$ for all $t\in\R$.
\end{enumerate}
\end{theorem}

\begin{remark}\label{rem:parametrix}
We prove uniqueness of the family $(\ex_{t})_{t\in\R}$ because it is of independent interest. However, in Sections \ref{sec:divform} and \ref{sec:standardform} we only used the existence of such a family.
In fact, we will also show (see \eqref{eq:defUt}) that one may let
\[
\ex_{t}f=U_{t}f+\int_{0}^{t}U_{t-\tau}Vf(\tau)\ud \tau,
\]
for each $f\in\Hps$, where $U_{t}$ is as in Theorem \ref{thm:parametrix}, and $V:\Hps\to C([-t_{0},t_{0}];\Hps)$ is bounded for each $t_{0}>0$. Moreover, it will be shown in Section \ref{sec:parametrix} (see \eqref{eq:defUt}) that one may choose $U_{t}=W^{*}\F_{t}W$, for $W$ the wave packet transform from \eqref{eq:deftransform}, and $\F_{t}$ the pull-back from \eqref{flowdefn} with respect to the bicharacteristic flow map $\Phi_{t}$ associated to the symbol $ \chi(\xi) b(x,\xi)$, for a $\chi\in C^{\infty}(\Rn)$ with $\chi(\xi)=0$ for $|\xi|\leq 8$, and $\chi(\xi)=1$ for $|\xi|\geq 16$ (note that $\Phi_{t}$ solves \eqref{eq:ODE} with $b$ replaced by $\chi b$).
\end{remark}

\begin{proof}
We first prove existence of solutions in a manner similar to, but somewhat different from, the first step of the proofs of Theorems \ref{thm:divwave} and \ref{thm:mainwave}. Then we solve an adjoint problem to show uniqueness of solutions. Since $b$ is a smooth symbol, there are no subtleties related to the Sobolev interval on which we prove existence and uniqueness. In particular, one does not need a final step as in the proofs of Theorems \ref{thm:divwave} and \ref{thm:mainwave}.

\subsubsection{Existence} 
The point is, of course, to deal with the error term associated with the parametrix in Theorem \ref{thm:parametrix}. To do so we again use an iterative construction. 

Fix $p\in[1,\infty)$ and $s\in\R$, with the construction of the family $\{\ex_{t}\mid t\in\R\}$ being independent of these parameters. First we explicitly construct the error term. For $f\in\Hps$ and $t\in\R$, set $V_{0}f(t):=-i(D_{t}-b(x,D))U_{t}f$ and, recursively,
\[
V_{j+1}f(t):=-i\int_{0}^{t}(D_{t}-b(x,D))U_{t-\tau}V_{j}f(\tau)\ud \tau
\]
for $j\geq0$. By Theorem \ref{thm:parametrix} \eqref{it:par3}, $V_{0}f\in C(\R;\Hps)$, and for each $t_{0}>0$ there exists an $\con_{0}>0$ such that $\|V_{0}f(t)\|_{\Hps}\leq \con_{0}\|f\|_{\Hps}$ for all $t\in[-t_{0},t_{0}]$. Using Theorem \ref{thm:parametrix} \eqref{it:par3} again, to obtain $\|V_{j+1}f(t)\|_{\Hps}\leq \int_{0}^{t}\con_{0}\|V_{j}f(\tau)\|_{\Hps}\ud \tau$, it now follows by induction from the dominated convergence theorem that $V_{j}f\in C(\R;\Hps)$ for all $j\geq0$, with
\[
\|V_{j}f(t)\|_{\Hps}\leq \frac{\con_{0}^{j+1}t^{j}}{j!}\|f\|_{\Hps}
\]
for $t\in[-t_{0},t_{0}]$. Hence $V:=\sum_{k=0}^{\infty}V_{k}$ defines a bounded operator
\begin{equation}\label{eq:Vbound}
V:\Hps\to C([-t_{0},t_{0}];\Hps)\subseteq L^{1}([-t_{0},t_{0}];\Hps)
\end{equation}
for each $t_{0}>0$. In particular, $Vf\in C(\R;\Hps)\subseteq L^{1}_{\loc}(\R;\Hps)$ for all $f\in\Hps$.

Next, we construct the solution operators. For $f\in\Hps$ and $t\in\R$, set
\begin{equation}\label{eq:defe}
\ex_{t}f:=U_{t}f+\int_{0}^{t}U_{t-\tau}Vf(\tau)\ud \tau.
\end{equation}
It follows from the dominated convergence theorem, using the strong continuity of $t\mapsto U_{t}$ from Theorem \ref{thm:parametrix} \eqref{it:par1} and that $Vf\in L^{1}_{\loc}(\R;\Hps)$, that $[t\mapsto \ex_{t}f]\in C(\R;\Hps)$. Moreover, by Theorem \ref{thm:parametrix} \eqref{it:par2} and \eqref{eq:Vbound}, one has $\sup_{|t|\leq t_{0}}\|\ex_{t}\|_{\La(\Hps)}<\infty$ for each $t_{0}>0$. This proves \eqref{it:smooth21} and \eqref{it:smooth22} for $k=0$. 

For $k=1$, let $h>0$ and write
\begin{align*}
&\frac{1}{h}\Big(\int_{0}^{t+h}U_{t+h-\tau}Vf(\tau)\ud\tau-\int_{0}^{t}U_{t-\tau}Vf(\tau)\ud\tau\Big)\\
&=\int_{0}^{t}\frac{U_{t+h-\tau}-U_{t-\tau}}{h}Vf(\tau)\ud\tau+\frac{1}{h}\int_{t}^{t+h}U_{t+h-\tau}Vf(\tau)\ud\tau.
\end{align*}
By Theorem \ref{thm:parametrix} \eqref{it:par1} and \eqref{it:par2} with $k=1$, the first term on the last line converges in $\Hpsm$ to $\int_{0}^{t}\partial_{t}U_{t-\tau}Vf(\tau)\ud\tau$ as $h\to0$. Since $U_{0}$ is the identity operator $I$, we can write the second term as
\begin{align*}
&\frac{1}{h}\int_{t}^{t+h}(U_{t+h-\tau}-I)Vf(\tau)\ud\tau+\frac{1}{h}\int_{t}^{t+h}Vf(\tau)\ud\tau\\
&=\int_{t}^{t+h}\frac{1}{h}\int_{0}^{t+h-\tau}\partial_{\sigma}U_{\sigma}Vf(\tau)\ud\sigma\ud\tau+\frac{1}{h}\int_{t}^{t+h}(Vf(\tau)-Vf(t))\ud\tau+Vf(t).
\end{align*}
By Theorem \ref{thm:parametrix} \eqref{it:par2} with $k=1$, and because $Vf\in C(\R;\Hps)$, the first two terms on the last line converge to zero in $\Hpsm$ as $h\to0$. We thus conclude that $[t\mapsto \ex_{t}f]\in C^{1}(\R;\Hpsm)$ with
\begin{equation}\label{eq:firstdere}
\partial_{t}\ex_{t}f=\partial_{t}U_{t}f+Vf(t)+\int_{0}^{t}\partial_{t}U_{t-\tau}Vf(\tau)\ud\tau
\end{equation}
for all $t\in\R$. Using again Theorem \ref{thm:parametrix} \eqref{it:par2} with $k=1$, and \eqref{eq:Vbound}, one obtains the necessary norm bounds to conclude the proof of \eqref{it:smooth21} for $k=1$.

 To deal with \eqref{it:smooth21} for $k\geq 2$, it suffices to prove part \eqref{it:smooth22}. Indeed, by Theorem \ref{thm:roughpseudo} one has $b(x,D)^{k}:\Hps\to\HT^{s-k,p}_{FIO}(\Rn)$, from which one would then obtain $[t\mapsto \ex_{t}f]\in C^{k}(\R;\HT^{s-k,p}_{FIO}(\Rn))$ with $\partial_{t}^{k}\ex_{t}f=i^{k}b(x,D)^{k}\ex_{t}f$ for all $t\in\R$. The bounds for $\ex_{t}$ that we already obtained then conclude the proof of \eqref{it:smooth21} for all $k\geq2$.

Hence it remains to prove \eqref{it:smooth22}. First note that, by definition of $\ex_{t}$ in \eqref{eq:defe} and by Theorem \ref{thm:parametrix} \eqref{it:par3}, one has $\ex_{0}f=U_{0}f=f$ for all $f\in\Hps$. Also, by \eqref{eq:firstdere} and by definition of $V$, we obtain for each $t\in\R$ that
\begin{align*}
&(D_{t}-b(x,D))\ex_{t}f=(D_{t}-b(x,D))U_{t}f-iVf(t)+\int_{0}^{t}(D_{t}-b(x,D))U_{t-\tau}Vf(\tau)\ud\tau\\
&=(D_{t}-b(x,D))U_{t}f-i\sum_{j=0}^{\infty}V_{j}f(t)+\sum_{j=0}^{\infty}\int_{0}^{t}(D_{t}-b(x,D))U_{t-\tau}V_{j}f(\tau)\ud\tau\\
&=(D_{t}-b(x,D))U_{t}f-iV_{0}f(t)=0.
\end{align*}
This proves \eqref{it:smooth22} and concludes the proof of existence of the family $\{\ex_{t}\mid t\in\R\}$.

\subsubsection{Uniqueness}
We prove uniqueness of a family with the properties in the statement of the theorem. For $s\in\R$, it suffices to show that if 
\[
u\in C(\R;\Hps)\cap C^{1}(\R;\Hpsm)
\]
is such that $u(0)=0$ and $(D_{t}-b(x,D))u(t)=0$ for all $t\in\R$, then $u\equiv 0$. We again write $u_{+}(t):=\ind_{[0,\infty)}(t)u(t)$ and $u_{-}(t):=u(t)-u_{+}(t)$ for $t\in\R$. Then
\[
u_{+},u_{-}\in C(\R;\Hps)\cap W^{1,1}_{\loc}(\R;\Hpsm)
\]
with $(D_{t}-b(x,D))u_{+}(t)=(D_{t}-b(x,D))u_{-}(t)=0$ for almost all $t\in\R$. The latter identity shows that in fact $u_{+},u_{-}\in C^{1}(\R;\Hpsm)$. It suffices to show that
\[
\int_{\R}\lb u_{+}(t),F(t)\rb_{\Rn}\ud t=0
\]
for all $F\in C^{\infty}_{c}(\R;C^{\infty}_{c}(\Rn))$. Let $b_{1}\in\A^{2}S^{1}_{1,1/2}$ be as in Lemma \ref{lem:Aduality}, with $b(x,D)^{*}=b_{1}(x,D)$, and let $\{\tilde{\ex}_{t}\mid t\in\R\}$ be the family obtained in the previous part of the proof with $b$ replaced by $-b_{1}$. Let $t_{0}>0$ be such that $F(t)=0$ for $t\geq t_{0}$. Now set $w(t):=i\int_{t}^{t_{0}}\tilde{\ex}_{t-\tau}F(\tau)\ud\tau$ for $t\in\R$. Then $w\in C^{k}(\R;\HT^{\sigma,p'}_{FIO}(\Rn))$ for all $k\geq0$ and $\sigma\in\R$, with $w(t)=0$ for $t\geq t_{0}$ and $(D_{t}+b_{1}(x,D))w(t)=F(t)$ for all $t\in\R$. Hence
\begin{align*}
\int_{\R}\lb u_{+}(t),F(t)\rb_{\Rn}\ud t&=\int_{\R}\lb u_{+}(t),(D_{t}+b_{1}(x,D))w(t)\rb_{\Rn}\ud t\\
&=-\int_{\R}\lb (D_{t}-b_{1}(x,D))u_{+}(t),w(t)\rb_{\Rn}\ud t=0,
\end{align*}
where we used the regularity and support conditions of $u_{+}$ and $w$ to see that all the integrals converge. This concludes the proof.
\end{proof}

\section{The flow on phase space}\label{sec:flow}

In this section we prove some properties of flow maps on phase space which will be needed for the proof of Theorem \ref{thm:mainsmooth} in the next section.

Let $b\in\A^{2}S^{1}_{1,1/2}$ be real-valued, elliptic and asymptotically homogeneous of degree $1$ with real-valued limit $\ahom \in \Crtwo S^{1}_{1,0}$ in the sense of Definition \ref{def:asymphom}. Throughout this section, we suppose additionally that $b(x,\xi)=0$ for all $(x,\xi)\in\Tp$ with $|\xi|\leq 8$. Such a $b$ will arise in the next section as $\chi b$ for a $b$ as in Theorem \ref{thm:mainsmooth} and a $\chi\in C^{\infty}(\Rn)$ with $\chi(\xi)=0$ for $|\xi|\leq 8$, and $\chi(\xi)=1$ for $|\xi|\geq 16$, but for simplicity of notation we will simply denote such a reduced symbol by $b$ in the present section, and merely assume that $b(x,\xi)=0$ for $|\xi|\leq 8$.

We recall that $\ahom$ is, by definition, homogeneous of degree one for $|\xi| \geq 1$. We shall also compute with the completely homogeneous function determined by $\ahom$, which we denote $\bhom$. Note that $\bhom$ is not in $\Crtwo S^1_{1,0}$, since second $\xi$-derivatives of $\bhom$ blow up at $\xi = 0$; however, we certainly have this regularity away from $\xi = 0$. We also record that 
\begin{equation}\label{b-bhom}
b - \bhom = (b - \ahom) + (\ahom - \bhom) \in C^1_- S^0_{1,1/2},
\end{equation}
which follows since the first derivatives of $\bhom$ are uniformly bounded near $\xi = 0$.  

Our parametrix for the first order equation $(D_t - b(x, D))u = 0$  involves the flow of the Hamilton vector field $V$ of $b$ on phase space $T^* \R^n$, which we denote by $(\Phi_t)_{t\in\R}$. Recall that for all $(x,\xi)\in\Tp$ one has
\[
V(x,\xi) := \partial_{\xi} b(x,\xi)\cdot \partial_{x}-\partial_{x}b(x,\xi)\cdot\partial_{\xi},
\]
and that $\Phi_{t}(x,\xi)\in\Tp$ satisfies $\partial_{t}\Phi_{t}(x,\xi)=V(\Phi_{t}(x,\xi))$, with $\Phi_{0}(x,\xi)=(x,\xi)$. That is, the $\Phi_{t}$ are the solution operators to the ODE
\begin{equation}\label{eq:ODE}
\begin{pmatrix}
\partial_{t}x(t)\\
\partial_{t}\xi(t)
\end{pmatrix}
=\begin{pmatrix}
\partial_{\xi}b(x(t),\xi(t))\\
-\partial_{x}b(x(t),\xi(t))
\end{pmatrix}\text{ and }\begin{pmatrix}
x(0)\\
\xi(0)
\end{pmatrix}=\begin{pmatrix}
x\\
\xi
\end{pmatrix}.
\end{equation}
We recall for later use that each $\Phi_t$ is a symplectomorphism, that is, it preserves the symplectic form $\sum_{j=1}^{n} d\xi_j \wedge dx_j$. Note also that the low-frequency part of the flow is trivial. That is, since $b(x,\xi)=0$ for $|\xi|\leq 8$, one has $V(x,\xi)=0$ and $\Phi_{t}(x,\xi)=(x,\xi)$ for all $|\xi|<8$ and $t\in\R$. 

To prepare for the parametrix construction, in this section we study the properties of this flow, and the induced operators $\F_t$ on the weighted tent space $\Tentps$ from Definition \ref{def:tentspaces}, given by pullback:
\begin{equation}
\F_t G = G \circ \Phi_{t}
\label{flowdefn}\end{equation}
for $p\in[1,\infty)$, $s\in\R$ and $G\in \Tentps$.  
The parametrix itself will be defined in terms of $\F_t$ by 
\[
U_t:= W^* \F_t W,
\]
where $W$ is the wave packet transform from \eqref{eq:deftransform}.

Throughout, fix $p\in[1,\infty)$ and $s\in\R$. The main result of this section is as follows.

\begin{theorem}\label{thm:flowprop} 
For each $t_{0}>0$ there exists an $M\geq0$ such that the following properties hold.
\begin{enumerate}
\item\label{it:Lipflow1} For all $t\in\R$ one has $\F_t\in\La(\Tentps)$, and $\|\F_{t}\|_{\La(\Tentps)}\leq M$ for all $t\in[-t_{0},t_{0}]$.
\item\label{it:Lipflow2} The map $t\mapsto \F_{t}$ is continuous in the strong operator topology from $\R$ into $\Tentps$. 
\item\label{it:Lipflow3} For all $f \in \Hps$, the map $t \mapsto \F_t W f$ is continuously differentiable from $\R$ into $\Tentpsm$, and 
\[
\|\partial_{t}\F_{t}Wf\|_{\Tentpsm}\leq M\|f\|_{\Hps}
\] 
for all $t\in[-t_{0},t_{0}]$.
\end{enumerate}
\end{theorem}

The proof of this theorem essentially follows from the results in~\cite[Section 6]{Geba-Tataru07}, which considers operators on $L^{2}(\Tp)$. As our setup, involving flows on tent spaces, is a bit different, we provide full details. It is also very closely related to~\cite[Section 3]{Smith98b}. 

We prepare for the proof of Theorem \ref{thm:flowprop} by  deriving some properties of the Hamilton vector field $V$. To do so, we use coordinates $(x, \w, \sigma)\in\Sp\times(0,\infty)$, where $\w=\xihat$ and $\sigma = |\xi|^{-1}$ for $\xi\neq0$, and we write
\begin{equation}\begin{gathered}
V = V_x \cdot \partial_{x} - V_{\w} \cdot \partial_{\w} - V_\sigma \sigma \partial_{\sigma}, \\
V_x := \partial_{\xi}b, \quad V_{\w} := \sigma(\partial_{x}b)^\perp, \quad V_\sigma := - \sigma\partial_{x}b \cdot \w,
\end{gathered}\label{Vcomp}\end{equation}
where $\perp$ denotes the orthogonal component of a vector relative to $\w$. There is a slight abuse of notation in writing $V_{\w} \cdot \partial_{\w}$, which we ignore. Keep in mind that $V_\sigma$ is the component of $\sigma \partial_\sigma$, not of $\partial_\sigma$. Note that the components of $V_{x}$, $V_{\w}$ and $V_{\sigma}$ are all $\A^{1}S^{0}_{1,1/2}$ symbols, since $b\in \A^{2}S^{1}_{1,1/2}$ and $V(x,\w,\sigma)=0$ for $\sigma\geq 1/8$.

 We will also work with the Hamilton vector field of the homogeneous symbol $\bhom$, which we denote by $\Vhom$ and decompose in a similar way as $V$, in terms of $\Vhom_{x}$, $\Vhom_{\w}$ and $\Vhom_{\sigma}$. Note that $\Vhom$ is homogeneous of order zero.

The following lemma will be used in the proof of part \eqref{it:Lipflow1} of Theorem \ref{thm:flowprop}. 

\begin{lemma}\label{lem:Lipbounds} 
Set $V_{\bullet}:=(V_{x},V_{\hat{\xi}},V_{\sigma})$. Then there exists an $M\geq0$ such that
\[
\big| \partial_{x} V_\bullet (x,\w,\sigma)\big| + \big| \partial_{\w} V_\bullet (x,\w,\sigma)\big| + \big| \partial_{\sigma}V_\bullet (x,\w,\sigma)\big| \leq \con
\]
for all $(x,\w,\sigma)\in\Sp\times(0,\infty)$.
\end{lemma}

\begin{proof} 
As already noted, the components of $V_{\bullet}$ are all $\A^{1}S^{0}_{1,1/2}$ symbols. Hence we can apply the vector fields $\partial_x$, $\partial_w$ and $\sigma \partial_\sigma$, and the resulting symbols will be uniformly bounded. However, we are claiming more, namely that the same holds if we apply $\partial_\sigma$, not just $\sigma \partial_\sigma$. The reason for this is the asymptotic homogeneity of $b$. For example, we apply $\partial_\sigma = - |\xi| \xi \cdot \partial_\xi$ to $V_x$. We obtain 
\begin{align*}
\partial_\sigma V_x(x,\w,\sigma) &= \partial_{\sigma}\partial_{\xi}b(x, \xi) = - |\xi| ( \xi \cdot \partial_{\xi}) \partial_{\xi}b(x,\xi) \\
&= - |\xi| \partial_{\xi}( \xi \cdot \partial_{ \xi} - 1 ) b(x,\xi) 
\end{align*}
for all $(x,\w,\sigma)\in \Sp\times(0,\infty)$, where $\xi=\sigma^{-1}\w$. 
Using the first asymptotic homogeneity property \eqref{it:asymphom1} of $b$, we see that 
$$
|\xi| \partial_{\xi} ( \xi \cdot \partial_{\xi} - 1 ) b(x, \xi) 
$$
is uniformly bounded. The argument is similar for the other components of $V$. 
\end{proof}

We will also need the following lemma, in the proof of part \eqref{it:Lipflow3} of Theorem \ref{thm:flowprop}.

\begin{lemma}\label{lem:beta} The map $VW:\Hps\to\Tentpsm$ is bounded.
\end{lemma}

\begin{proof}
Let $f\in\Hps$. Note that $VWf(x,\xi)=0$ for $|\xi|<8$, since we already observed that $V(x,\xi)=0$ for such $(x,\xi)\in\Tp$. On the other hand, $Wf(x,\xi)$ is smooth for $|\xi|\geq8$, by Lemma \ref{lem:packets} and the definition of $W$ in \eqref{eq:deftransform}. Hence
\begin{equation}\label{eq:PtWf}
VWf(x,\xi)= 
(\partial_{\xi} b\cdot \partial_{x}Wf)(x,\xi)- (\partial_{x}b\cdot \partial_{\xi}Wf)(x,\xi)
\end{equation}
makes sense classically for such $(x,\xi)$. Now fix $1\leq j\leq n$ and first note that
\begin{equation}\label{eq:boundb}
\partial_{\xi_{j}}b:\Tentpsm\to\Tentpsm\text{ and }\partial_{x_{j}}b:\Tentps\to\Tentpsm,
\end{equation}
as follows from pointwise estimates, using that $\partial_{\xi}b\in S^{0}_{1,1/2}$ and $\partial_{x}b\in S^{1}_{1,1/2}$.

Next, consider 
\begin{equation}\label{eq:GW}
\begin{aligned}
\partial_{x_{j}} Wf(x,\xi) &=\partial_{x_{j}} \psi_{\xi}(D)f(x)=  (2\pi)^{-n} \partial_{x_{j}} \int_{\R^{2n}} e^{ix \cdot \eta} \psi_\xi(\eta) \wh{f}(\eta) \ud\eta\\
&=i(2\pi)^{-n}\int_{\R^{2n}} e^{ix \cdot \eta} \eta_{j}\psi_\xi(\eta) \wh{f}(\eta) \ud\eta
\end{aligned}
\end{equation}
and note that the wave packets $\psi_{\xi,1}(\eta):=|\xi|^{-1}\eta_{j}\psi_{\xi}(\eta)$ have similar support and boundedness properties as the $\psi_{\xi}$. Hence, by Remark \ref{rem:transformbounded}, one has $W_{1}W^{*}:\Tentps\to \Tentpsm$ for the new wave packet transform 
\[
W_{1}g(x,\xi):=\begin{cases}
\psi_{\xi,1}(D)f(x)&\text{if }|\xi|>1,\\
0&\text{if }|\xi|\leq 1.\end{cases}
\]
Combined with \eqref{eq:boundb} and the fact that $W$ is an isometry, this shows that
\begin{align*}
\|\ind_{[8,\infty)}(\xi)\partial_{\xi_{j}}b\partial_{x_{j}}Wf\|_{\Tentpsm}&\lesssim \|\ind_{[8,\infty)}(\xi)\partial_{x_{j}}Wf\|_{\Tentpsm}\\
&=\|\ind_{[8,\infty)}(\xi)W_{1}f\|_{\Tentps}\\
&=\|\ind_{[8,\infty)}(\xi)W_{1}W^{*}Wf\|_{\Tentps}\\
&\lesssim \|Wf\|_{\Tentps}=\|f\|_{\Hps}.
\end{align*}
This suffices to deal with the first term on the right-hand side of \eqref{eq:boundb}.

The argument for the second term on the right-hand side of \eqref{eq:boundb} is similar. One differentiates with respect to $\xi_{j}$ instead of $x_{j}$ in \eqref{eq:GW}, and applies Lemma \ref{lem:packets} to see that the wave packets $\psi_{\xi,2}(\eta):=\partial_{x_{j}}\psi_{\xi}(\eta)$ satisfy similar (and in fact slightly better) support and boundedness properties as the $\psi_{\xi}$. Then the natural associated wave packet transform $W_{2}$ satisfies $W_{2}W^{*}:\Tentps\to \Tentps$, and a similar argument as before concludes the proof.
\end{proof}

We are now ready to prove Theorem \ref{thm:flowprop}.

\begin{proof}[Proof of Theorem \ref{thm:flowprop}]
Most of the work will go into proving part \eqref{it:Lipflow1}.

\subsubsection{Proof of \eqref{it:Lipflow1}}

First note that it suffices to consider only the high-frequency part of the flow. More precisely, one can write any $G\in T^{p}_{s}(\Sp)$ as $G=\ind_{|\xi|\leq 1}G+\ind_{|\xi|>1}G$. As already remarked, one has $\Phi_{t}(x,\xi)=(x,\xi)$ for all $|\xi|<8$. Hence the required bounds are trivial for $\ind_{|\xi|\leq 1}G$, and we may consider in the remainder only $|\xi|>1$.

\subsubsection{The homogeneous flow}

 The function $\bhom$ is homogeneous of degree one and is $\Crtwo S^{1}_{1,0}$ away from $\xi = 0$, thus it generates a flow $(\Phihom_t)_{t\in\R}$ which is homogeneous of degree one and is Lipschitz away from $\xi = 0$. As shown in \cite[Proposition 2.4]{HaPoRo20}, by projection such a flow induces a flow on the cosphere bundle, denoted $(\chi_{t})_{t\in\R}$, that is not only Lipschitz with respect to the usual metric on the cosphere bundle, but Lipschitz\footnote{Technically speaking, in \cite[Proposition 2.4]{HaPoRo20} the maps are assumed to be (infinitely) smooth. However, the relevant bounds only involve the supremum of the Jacobian of $\chi_{t}$, and the same proof works for bi-Lipschitz maps (see also footnote \ref{foot:curves}).\label{foot:Lipschitz}} with respect to the contact metric $d$. Moreover, the $d$-Lipschitz bounds for $\chi_t$ are locally uniform in $t$, since the Lipschitz bounds of the $\Phihom_{t}$ are.
 
For $(x_{0},\w_{0},\sigma_{0})\in \Sp\times(0,1)$, write $\Phihom_{t}(x_{0},\w_{0},\sigma_{0})=(y(t),\nu(t),\tau(t))$ for $t\in\R$. Since $\Vhom_\sigma$ is bounded, for each $t_{0}>0$ there exists an $M=M_{t_{0}}\geq0$, independent of $(x_{0},\w_{0},\sigma_{0})$, such that
$$
\Big|\frac{d}{d t}\log \tau(t) \Big| =\Big|\frac{\dot{\tau}(t)}{\tau(t)}\Big| \leq \con
$$
for all $t\in[-t_{0},t_{0}]$. This in turn implies that $\tau(t)$ does not vary much on $[-t_{0},t_{0}]$:
\begin{equation}
e^{-\con t_{0}} \leq \frac{\tau(t)}{\sigma_{0}}\leq e^{\con t_{0}}.
\label{sigmacomp}\end{equation}
Now the combination of the $d$-Lipschitz property of $\chi_t$ and \eqref{sigmacomp} shows that $\Phihom_t$ maps the parabolic region $\Gamma_1(x_{0}, \omega_{0})$, defined by \eqref{Gammadefn}, into $\Gamma_{\alpha(t)}(\chi_t(x_{0}, \omega_{0}))$ for some $\alpha(t)>0$ which is locally uniformly bounded in $t$, from above and below. Lemma \ref{lem:aperture}, on comparability of the tent space norm under change of aperture, then implies locally uniform boundedness of the homogeneous flow $(\Phihom_t)_{t\in\R}$ on $\Tentps$. 

\subsubsection{The smoothed flow}

We are, however, interested in the smoothed flow $(\Phi_t)_{t\in\R}$. To analyze it we estimate the difference between $\Phi_t$ and $\Phihom_{t}$,
using the metric $d$. Recall that 
\begin{equation}
\big(d((x, \w), (y, \nu))\big)^{2} \eqsim |x-y|^2 + |\w - \nu|^2 + |\nu \cdot (y-x)|
\label{d2equiv}\end{equation}
for all $(x,\w),(y,\nu)\in\Sp$, by \eqref{eq:metric}. Note that the low-frequency parts of these flows are very different; however, our interest is in the high-frequency part.

For a given $(x_{0},\w_{0},\sigma_{0})\in\Sp\times(0,1)$, write $\Phi_{t}(x_{0},\w_{0},\sigma_{0})=(x(t), \w(t), \sigma(t))=(x(t),\xi(t))$ and $\Phihom_{t}(x_{0}, \w_{0}, \sigma_{0}) = (y(t), \nu(t), \tau(t))=(y(t),\eta(t))$ as before. We fix $t_{0}>0$ and consider $t\in[-t_{0},t_{0}]$. The bounds which we will obtain are uniform in such $t$ and in the initial data. For simplicity of notation, we will typically omit the variable $t$. 

We begin by observing that $V_\sigma$ is bounded, since $V_\sigma$ is in $\A^{1}S^{0}_{1,1/2}$, and therefore, similar to \eqref{sigmacomp}, we have 
\begin{equation}
e^{-\con t_{0}} \leq \frac{\sigma(t)}{\sigma_{0}}\leq e^{\con t_{0}}.
\label{sigmacomp2}\end{equation}

Next, we claim that \emph{it suffices to show that the right-hand side of \eqref{d2equiv}, that is, 
$$
|x(t)-y(t)|^2 + |\w(t) - \nu(t)|^2 + |\nu(t) \cdot (y(t)-x(t))|,
$$
 is bounded by a constant multiple of $\sigma_0$}, as a function of $t$. Indeed, suppose that we have proved this. 
By comparing with the homogeneous flow $\Phihom_{t}$, and using again that $\chi_{t}$ is $d$-Lipschitz, it then follows that the ball $B_{\sqrt{\sigma_{0}}}(x_{0},\w_{0})\subseteq\Sp$ gets mapped into $B_{\beta(t)\sqrt{\sigma_{0}}}(x(t),\w(t))$ for some locally uniformly bounded $\beta(t)>0$. This shows that $\Phi_{t}$ maps the high frequency part of the paraboloid $\Gamma_{1}(x_{0},\w_{0})$ into $\Gamma_{\beta(t)}(x(t),\w(t))$. 
Since $\Phi_{t}$ is also Lipschitz with respect to the standard metric, the Jacobian factor in the change of variables between $\Gamma_{1}(x_{0},\w_{0})$ and $\Gamma_{\beta(t)}(x(t),\w(t))$ is bounded. 
Lemma \ref{lem:aperture} then completes the proof.

In fact, we will prove the stronger statement that
\begin{equation}
\mcD(t) := |x-y|^2 + |\w - \nu|^2 + \Big( 1- \frac{\sigma}{\tau} \Big)^2 +  |\nu \cdot (y-x)| \lesssim \sigma_0.
\label{D2} \end{equation}
By Gronwall's lemma, it suffices to prove an inequality of the form 
\begin{equation}
\frac{d}{dt} \mcD(t) \leq C(\mcD(t) + \sigma_{0})
\label{enoughD2}\end{equation}
for some $C>0$, and then multiply $\mcD(t)$ by $e^{-Ct}$ and integrate from $0$ to $t$. In turn, to show \eqref{enoughD2} we take the time derivative of each term in \eqref{D2}, and show that it is bounded by the right-hand side of \eqref{enoughD2}. 

\subsubsection{The first three terms}
We start with the time derivative of the first term in \eqref{D2},  $|x-y|^2$. It is equal to 
$$
2(x-y) \cdot \big(V_x(x, \w, \sigma) - \Vhom_x(y, \nu, \tau) \big)
$$
which we estimate by expressing it as 
$$
2(x-y) \cdot \big(V_x(x, \w, \sigma) - V_x(y, \nu, \tau) + V_x(y, \nu, \tau) - \Vhom_x(y, \nu, \tau) \big).
$$
By Lemma \ref{lem:Lipbounds}, $V_x$ is Lipschitz in $x$, $\w$ and $\sigma$. Hence one can use \eqref{sigmacomp}, \eqref{sigmacomp2} for $\sigma$ and $\tau$ to bound the difference between the first two terms in brackets by a multiple of $|x-y| + |\w - \nu| + \sigma_{0}$. To estimate the second term, we use the fact that the difference between the vector fields $V - \Vhom$ is the Hamilton vector field of $b - \bhom$, which satisfies \eqref{b-bhom}. 
The $x$-component is the $\xi$-derivative of this vector field, which is therefore $O(|\xi|^{-1}) = O(\sigma_{0})$. So we get an estimate for this time derivative by a multiple of
$$
|x-y| \big( |x-y| + |\w - \nu| + \sigma_{0}),
$$
which is bounded by the right-hand side of \eqref{enoughD2} since $\sigma_{0}<1$. 

The argument to bound the time derivative of the second term on the right-hand side of \eqref{D2} is similar to the argument for the first term, so we omit the details. 

Next, consider the third term of \eqref{D2}. The time derivative is 
\begin{equation*}\begin{gathered}
2 \Big( 1- \frac{\sigma}{\tau} \Big)\Big( - \frac{\sigma V_\sigma(x, \w, \sigma)}{\tau} + \frac{\sigma \Vhom_\sigma(y, \nu, \tau)}{\tau} \Big)  \\
= 2 \Big( 1- \frac{\sigma}{\tau} \Big) \frac{\sigma}{\tau} \big( \Vhom_\sigma(y, \nu, \tau) - \Vhom_\sigma(x, \w, \sigma) + \Vhom_\sigma(x, \w, \sigma) -V_\sigma(x, \w, \sigma) \big) .
\end{gathered}\end{equation*}
Using the Lipschitz property of $\Vhom$ with respect to the standard metric, and the fact that $\Vhom$ is independent of $\tau$, we can estimate 
$$
\big|  \Vhom_\sigma(y, \nu, \tau) - \Vhom_\sigma(x, \w, \sigma) \big| \lesssim |x-y| + |\w - \nu|.
$$
For the final two terms, we use the fact that $V - \Vhom$ is the Hamilton vector field of \eqref{b-bhom}, and therefore we can bound 
$$
\big| \Vhom_\sigma(x, \w, \sigma) - V_\sigma(x, \w, \sigma) \big| \lesssim \sqrt{\sigma}. 
$$
Hence, using \eqref{sigmacomp2}, the time derivative of the third term in \eqref{D2} is bounded by a multiple of
$$
\Big( 1- \frac{\sigma}{\tau} \Big) \big( |x-y| + |\w - \nu| + \sqrt{\sigma_{0}} \big)
$$
which is bounded by the right-hand side of \eqref{enoughD2}. 

\subsubsection{The fourth term}

This term is more interesting. We follow a similar calculation by Smith in \cite{Smith98b}, although our setting is slightly different. Smith treats a sequence of approximate Hamiltonians ${H}_k$, each of which is homogeneous, while we have in effect `pasted together' all these into a single Hamiltonian $b$ which is only asymptotically homogeneous. So we need to take care with the dependence on $\sigma$. We compute the absolute value of the time derivative as 
\begin{equation*}\begin{gathered}
|-\Vhom_\sigma(y, \nu, \tau) \big(\nu \cdot (y-x)\big) + \tau \frac{d}{dt} \big(\tau^{-1}\nu \cdot (y-x)\big)| \\
\lesssim |\nu \cdot (y-x)| + \tau\big| \partial_{y} \bhom(y,\nu,\tau)\cdot (x-y) + \tau^{-1}\nu \cdot \big(  \partial_{\eta} \bhom(y, \nu,\tau) - \partial_{\xi} b(x, \w,\sigma) \big)\big|
\end{gathered}\end{equation*}
and note that the first term is bounded by $\mcD$. The other terms can be rewritten as 
\begin{align*}
&\tau \big( \partial_{y} \bhom(y,\eta)\cdot (x-y) + \eta \cdot \big(\partial_{\eta} \bhom(y, \eta) - \partial_{\xi} \bhom(x, \xi)  - \partial_{\xi} (b-\bhom)(x, \xi) \big) \big) \\
&= \tau \big( \bhom(y, \eta) - \bhom(x, \xi)  + (x-y) \cdot \partial_{y} \bhom(y, \eta) + (\xi - \eta) \cdot \partial_{\xi} \bhom(x, \xi)\\
&\ \ \ -\eta \cdot  \partial_{\xi} (b-\bhom)(x, \xi) \big)  
\\
&=  \tau \big( \bhom(y, \eta) - \bhom(x, \xi)  + (x-y) \cdot \partial_{y} \bhom(x, \xi) + (\xi - \eta) \cdot \partial_{\xi} \bhom(x, \xi) \big) \\
&\ \ \ + \tau  (x-y) \cdot \big( \partial_{y} \bhom(y, \eta) - \partial_{y} \bhom(x, \xi) \big) - \hat{\eta} \cdot \partial_{\xi} (b-\bhom)(x, \xi) .
\end{align*}
In the last line, the term $\hat{\eta}\cdot \partial_\xi (b-\bhom)$ is $O(\sigma)$, using \eqref{b-bhom}, so this term is acceptable. To treat the two terms in large parentheses, 
we first consider the case where $|\hat \eta - \hat \xi|=|\nu-\w|$ is large, say $|\hat \eta - \hat \xi|\geq 1/4$. In that case, $\mcD \geq 1/16$, so we need only show that these two terms are $O(1)$. 
This is easy to check, noting that $|x-y| = O(1)$ since both $\dot x$ and $\dot y$ are $O(1)$. 

The remaining case is $|\hat \eta - \hat \xi| \leq 1/4$. Here we treat the first of the
terms in large parentheses as follows (the other one is treated analogously; we omit the details). We use Taylor's formula to express this term as 
\begin{align}
-\tau &\bigg( \sum_{jk} (y-x)_j (y-x)_k \int_0^1 (1-r) \partial_{x_{j}x_{k}}^2\bhom(x + r(y-x), \xi + r(\eta - \xi)) \ud r\nonumber\\ 
\label{Taylorhom}&+ 2 \sum_{jk} (y-x)_j (\eta - \xi)_k \int_0^1  (1-r) \partial_{x_{j}\xi_{k}}^2  \bhom(x+r(y-x),\xi + r(\eta - \xi) ) \ud r \\
&+ \sum_{jk} (\eta - \xi)_j (\eta - \xi)_k \int_0^1 (1-r) \partial_{\xi_{j}\xi_{k}}^2  \bhom(x + r(y-x), \xi + r(\eta - \xi)) \ud r \bigg)\nonumber. 
\end{align}
We note that the condition $|\hat \eta - \hat \xi| \leq 1/4$ implies that $|\xi + r(\eta - \xi)|^{-1}$ is comparable to $\sigma$ and $\tau$ (and hence, to $\sigma_0$ thanks to \eqref{sigmacomp}, \eqref{sigmacomp2}) for $0 \leq r \leq 1$. 
Using the fact that $\bhom \in \Crtwo S^1_{1,0}$ away from $\xi = 0$, and the homogeneity of degree $1$ in $\xi$, as well as \eqref{sigmacomp}, \eqref{sigmacomp2} and the inequality
$$
|\eta - \xi|^2 \leq 2\sigma^{-2} \Big( |\hat\eta - \xihat|^2 + \Big( 1 - \frac{\sigma}{\tau} \Big)^2 \Big),
$$
  we see that this is bounded by a multiple of
$$
|y-x|^2 + \sigma_{0} |y-x| |\eta - \xi| + \sigma_{0}^2 |\eta - \xi|^2 \lesssim  |y-x|^2  + |\hat\eta - \xihat|^2 +  \Big( 1- \frac{\sigma}{\tau} \Big)^2,  
$$
which is an acceptable error term. This completes the proof of \eqref{enoughD2}, and thus also of \eqref{it:Lipflow1}. 

\subsubsection{Proof of \eqref{it:Lipflow2}}  In view of part \eqref{it:Lipflow1}, we only need to show the continuity of $\F_t G$ for $G$ in a dense subset $S$ of $\Tentps$. We choose for our dense subset $S$ the continuous functions of compact support. Then, for all $G\in S$ and $t_{0}>0$, there is a fixed compact set containing the support of $\F_t G$ for all $t\in[-t_{0},t_{0}]$. 
We can thus choose an integrable function that dominates all the $\F_t G$. It follows, using the dominated convergence theorem,  that for all $(x,\w)\in\Sp$ the integrals 
$$
| \mathcal{A}_{s}(\F_t G - \F_{t'} G) (x, \w) |^2 =  \int_{\Gamma(x, \w)} \ang{\eta}^{2s}  \mu(|\eta|) |\F_t G - \F_{t'} G|^2(y,\eta) \ud y \ud\eta
$$
converge to zero as $t' \to t$, where $\mu$ is as in \eqref{eq:defmu}. Moreover, one has $\F_{t}G(y,\eta)-\F_{t'}G(y,\eta)=G(y,\eta)-G(y,\eta)=0$ for $|\eta|<8$. Hence there is also a fixed compact subset of $\Sp$ containing the support of all the functions $\A_{s} (\F_t G-F_{t'}G)$, for $t,t'\in[-t_{0},t_{0}]$. So, using the dominated convergence theorem again, the $L^p(\Sp)$ norm of $\mathcal{A}_{s}(\F_t G - \F_{t'} G)$ converges to zero as $t' \to t$.
It now follows from Definition~\ref{def:tentspaces} that the $\Tentps$ norm of $\F_t G - \F_{t'} G$ converges to zero as $t' \to t$.
This establishes the strong continuity of $\F_t$ on $\Tentps$.

\subsubsection{Proof of \eqref{it:Lipflow3}}  Let $f\in\Hps$. Recall that $Wf$ is smooth in $(x,\xi)$ for $|\xi|\geq 8$, and $\F_{t}Wf(x,\xi)=Wf(x,\xi)$ for $|\xi|<8$. Hence $\F_t Wf$ is pointwise continuously differentiable in time, with derivative $V \F_t Wf $. This is equal to $\F_t V Wf$, since $V$ commutes with the flow that it generates. Now Lemma \ref{lem:beta}, combined with parts \eqref{it:Lipflow1} and \eqref{it:Lipflow2}, concludes the proof of \eqref{it:Lipflow3}, and thereby of Theorem~\ref{thm:flowprop}. 
\end{proof}

\begin{remark}\label{rem:time-dep} 
Note that, if the symbol $b$ is time-dependent, then $V$ no longer commutes with the flow, and a more elaborate argument is needed to show \eqref{it:Lipflow3}. This is one of the additional difficulties in dealing with equations with time-dependent coefficients.
\end{remark}

\begin{remark}\label{rem:C11needed}
We used the assumption that $\bhom\in \Crtwo S^{1}_{1,0}$, and in particular the uniform bounds for the second derivatives of $\bhom$ for $|\xi| \gtrsim 1$, in several key places in the proof of Theorem \ref{thm:flowprop}. For example, these bounds imply that the homogeneous flow $(\Phihom_{t})_{t\in\R}$ is Lipschitz, and they were also crucial in \eqref{Taylorhom}.
\end{remark}

\section{Parametrix for the first order equation}\label{sec:parametrix}

Let $b$ be as in Theorem \ref{thm:parametrix}. That is, $b\in\A^{2}S^{1}_{1,1/2}$ is real-valued, elliptic and asymptotically homogeneous of degree one with real-valued limit $a\in \Crtwo S^{1}_{1,0}$. Fix $p\in[1,\infty)$ and $s\in\R$. In this section we prove Theorem~\ref{thm:parametrix} for the parametrix  $(U_t)_{t\in\R}$ defined on $\Hps$ by 
\begin{equation}\label{eq:defUt}
U_t := W^* \F_t W
\end{equation}
for $t\in\R$. Here $W$ is the wave packet transform from \eqref{eq:deftransform}, and $\F_t$ is the flow map (see \eqref{flowdefn}) induced by the Hamilton vector field $V$ of $\chi b$, for a $\chi\in C^{\infty}(\Rn)$ such that $\chi(\xi)=0$ for $|\xi|\leq 8$ and $\chi(\xi)=1$ for $|\xi|\geq 16$. 

\begin{proof}[Proof of Theorem~\ref{thm:parametrix}]
Almost all of the work will go into proving property \eqref{it:par3}.

\subsubsection{Reductions}

Let $f \in \Hps$. Properties \eqref{it:par1} and \eqref{it:par2} in the statement of Theorem~\ref{thm:parametrix} are an almost immediate consequence of Theorem~\ref{thm:flowprop}. Indeed, recall that, by Proposition \ref{prop:transformsandweights},
the $\Hps$ norm of $f$ is equivalent to the $\Tentps$ norm of $Wf$. Thus, the strong continuity of $t\mapsto W^* \F_t W$ on $\Hps$ is implied by the strong continuity of $t\mapsto W W^* \F_t$
on $\Tentps$, and a similar statement holds for the time derivatives of these maps, and for the bounds in \eqref{it:par2}. Theorem~\ref{thm:flowprop} demonstrates the boundedness and continuity of $\F_t$, while Corollary \ref{cor:transformbounded} states that $W W^*$ is bounded.

The first statement of property \eqref{it:par3}, namely that $U_0 f = f$, is trivially satisfied. Thus, it remains to show that
for $f \in \Hps$, the error term 
\begin{equation}\label{eq:errortoprove}
(D_t - b(x, D)) U_t f 
\end{equation}
is bounded, locally uniformly in $t$, and continuous as a map from $\R$ into $\Hps$.  
The remainder of this section will be devoted to  proving this. 

We begin with a convenient reduction: it suffices to prove instead that
\[
(D_t - (\chi b)(x, D)) U_t f=(D_t - b(x, D)\chi(D)) U_t f
\]
is bounded, locally uniformly in $t$, and continuous as a map from $\R$ into $\Hps$. Here, as before, $\chi\in C^{\infty}(\Rn)$ is such that $\chi(\xi)=0$ for $|\xi|\leq 8$ and $\chi(\xi)=1$ for $|\xi|\geq 16$. Indeed, suppose that we can prove the theorem for $\tilde b$, and write
$$
(D_t - b(x, D)) U_t = (D_t - (\chi b)(x, D)) U_t + ((\chi b)(x, D) - b(x, D)) U_t.
$$
By assumption, the first term on the right-hand side, $(D_t - (\chi b)(x, D)) U_t$, satisfies the conditions of the theorem. On the other hand, the term $((\chi b)(x, D) - b(x, D)) U_t$ does so as well, since $(\chi b)(x, D) - b(x, D)=((\chi-1)b)(x,D)$ is a pseudodifferential operator of order zero (in fact, it has order $-\infty$), and it is therefore bounded on $\Hps$, as follows from Theorem \ref{thm:roughpseudo}. This shows that it suffices to prove \eqref{eq:errortoprove} with $b$ replaced by $\chi b$. However, for simplicity of notation, in the remainder of this section we will simply denote this reduced symbol by $b$, and merely assume in addition that $b(\chi)=0$ for $|\xi|\leq 8$, as we also did in Section \ref{sec:flow}.

Having made this reduction, we claim that it suffices to consider \eqref{eq:errortoprove} under the assumption that the function $f$ has Fourier transform supported in the region $\{ |\xi| \geq 2 \}$. In fact, we can write $f = \tilde \chi(D) f + (1 - \tilde \chi)(D) f =: f_1 + f_2$ for a $\tilde \chi \in C^\infty(\R^n)$ which is equal to $0$ for $|\xi| \leq 2$ and equal to $1$ for $|\xi| \geq 4$. Then $Wf_2$ is supported where $|\xi| \leq 8$ due to the support properties of the $\psi_\xi$. Hence $\F_t W f_2 = W f_2$, since $b$ vanishes where $|\xi| \leq 8$. Thus $U(t) f_2 = f_2$ for all $t$, and $b(x,D)f_{2}=0$. Hence the conclusions of the theorem are trivial for $f_2$, and we only need to check them for $f_1$, which proves the claim.

The virtue of these reductions is that we can now disregard the component of $W$ on the second line of \eqref{eq:deftransform} entirely. To see this, recall first that the function $q$ defined in \eqref{r-defn} is supported where $|\zeta| \leq 2$, cf.~\eqref{r-support}. Hence, if $f$ has Fourier transform supported where $\{|\xi| \geq 2 \}$, then $Wf(\xi)=0$ for $|\xi|\leq1$. The same holds true for $\F_t Wf$ for all $t\in\R$, as $\F_{t}$ is the identity where $|\xi| \leq 1$. Finally, the low-frequency component of $W^*$ applied to $\F_t Wf$ then vanishes as well. 

For the rest of the proof we may thus assume that the quantity on the second line of \eqref{eq:deftransform} is in fact zero, and similarly for $W^{*}$. We may also assume that $f\in\Sw(\Rn)$ and obtain appropriate $\Hps$-bounds for $f$, since the Schwartz functions lie dense. Then $Wf(x,\xi)$ is smooth for $|\xi|>1$, cf.~Lemma \ref{lem:packets}.

\subsubsection{Kernel of the error term}

Next, we calculate the kernel of the error term. 

Recall that $\partial_t \F_t = V \F_t$ for all $t$. Hence the error term can be written as
\[
( -i W^* V - b(x, D) W^* ) \F_t Wf.
\]
We therefore begin our analysis by considering the operator $-i W^* V - b(x, D) W^*$ acting on $Wf$. First note that $-iW^* VWf(x)$ equals 
\[
(2\pi)^{-n} \int_{\R^{3n}}  (-i) e^{i(x-y)\cdot \zeta} \psi_\eta(\zeta) (\partial_{\eta} b\cdot \partial_{y} - \partial_{y} b\cdot \partial_{\eta})Wf(y, \eta) \ud y\ud\eta\ud\zeta,
\]
where the integrals are absolutely convergent because $F$ is an element of the class of test functions $\Da(\Tp)$ from Section \ref{sec:tentspaces}, as was remarked above \eqref{eq:defadjoint}. It follows that the Schwartz kernel of this operator is 
\[
 (2\pi)^{-n}  \int_{\Rn}  e^{i(x-y)\cdot \zeta}  \big( \zeta \cdot \partial_{\eta} b(y,\eta) \psi_\eta(\zeta) -i \partial_{y} b(y,\eta)\cdot \partial_{\eta} \psi_\eta(\zeta) \big) \ud\zeta
\]
for $|\eta|>1$. On the other hand, the kernel of the operator $b(x, D) W^*$ takes the form of an oscillatory integral 
\[
 (2\pi)^{-2n} \int_{\R^{3n}} e^{i(x-z) \cdot \xi} b(x, \xi) e^{i(z-y) \cdot \zeta} \psi_\eta(\zeta) \ud\xi\ud z \ud\zeta,
\]
 which simplifies to
\[
 (2\pi)^{-n} \int_{\Rn} e^{i(x-y) \cdot \zeta} b(x, \zeta) \psi_\eta(\zeta)  \ud \zeta.
\]
 Combining these we find the kernel of $(2\pi)^{n}(-i W^* V - b(x, D) W^*)$:
 \begin{equation}
\int_{\Rn}  e^{i(x-y)\cdot \zeta}  \big( \zeta \cdot \partial_{\eta} b(y,\eta) \psi_\eta(\zeta) -i \partial_{y} b(y,\eta) \cdot \partial_{\eta} \psi_\eta(\zeta) - b(x, \zeta) \psi_\eta(\zeta) \big) \ud\zeta
\label{initial}\end{equation}
for $|\eta|>1$. For the rest of the proof we will only deal with such $\eta$, as is allowed by our assumption on $W$. Moreover, from here on the function $f$ will no longer play any role in the proof, and we will only work with the kernel of the error term.

In this expression, we will use the identity 
$$
\partial_\eta \psi_\eta(\zeta) = -|\eta|^{-1} |\zeta|  \partial_\zeta \psi_\eta(\zeta)  + \Omega_{\eta}( \zeta),
$$
where $\Omega_{\eta}$ is defined in \eqref{Omega-defn} and satisfies the bounds in \eqref{eq:packetbounds3}. In particular, by the support properties of $\psi_{\eta}$ in Lemma \ref{lem:packets}, $\Omega_{\eta}(\zeta)$ decays a full power of  $\lb\zeta\rb$ faster than $\psi_{\eta}(\zeta)$. We  substitute this into \eqref{initial}, and integrate by parts in $\zeta$. At the same time, we write $\zeta \cdot \partial_\eta b = (\zeta - \eta) \cdot \partial_\eta b + \eta \cdot \partial_\eta b$. 
This gives us
 \begin{equation}\label{initial*}
 \begin{aligned}
\int_{\Rn} e^{i(x-y)\cdot \zeta}  \Big[&\Big( (\zeta - \eta) \cdot \partial_\eta b(y, \eta) + \eta \cdot \partial_\eta b(y, \eta) +(x-y) \cdot \partial_y b(y, \eta)
\\
&-\frac{|\eta| - |\zeta|}{|\eta|} (x-y) \cdot \partial_{y} b(y, \eta)  - i \frac{\hat \zeta}{|\eta|}  \cdot \partial_{y} b(y, \eta) - b(x, \zeta) \Big) \psi_\eta(\zeta)\\
&-i \partial_y b(y, \eta) \cdot  \Omega_\eta(\zeta) \Big] \ud \zeta. 
\end{aligned}\end{equation}

\subsubsection{The main part of the proof}

We now get to the heart of the proof, which revolves around the expression in \eqref{initial*}.

To deal with this expression, we first make a definition. We say that $s\in C^{\infty}(\R^{4n})$ is an \emph{acceptable} symbol of order $\delta\in\R$ 
if the following two conditions hold:
\begin{itemize}
\item For all $\alpha,\beta\in\Z_{+}^{n}$ and $\gamma\in \Z_{+}$ there exists an $M_{\alpha,\beta,\gamma}\geq0$ such that
\[
\big| \partial_{x}^\alpha \partial_{\zeta}^\beta (\hat{\zeta}\cdot\partial_{\zeta})^\gamma s(x, y, \eta, \zeta) \big| \leq \con_{\alpha, \beta, \gamma} \langle \zeta \rangle^{-\delta + |\alpha|/2 - |\beta|/2 - \gamma}
\]
for all $x,y,\eta,\zeta\in\Rn$;
\item For all $x,y,\eta,\zeta\in\Rn$ one has 
\[
s(x,y,\eta,\zeta)=0\text{ if }\zeta\notin \supp(\psi_{\eta}).
\]
\end{itemize}
The rest of the proof will mainly concern the following proposition.

\begin{proposition}\label{prop:s-symbol}
The expression \eqref{initial*} is of the form 
 \begin{equation}
\int_{\Rn}  e^{i(x-y)\cdot \zeta}  s(x, y, \eta, \zeta)\ud\zeta
\label{initial3}\end{equation}
for an acceptable symbol $s$ of order $-(n+1)/4$.
\end{proposition}

\begin{proof}
Proving the proposition requires finding cancellations between the individual terms, as most of these have order $-(n+1)/4 + 1$. 

\paragraph{\textbf{First steps}}

Note that, by Lemma \ref{lem:packets}, $\psi_{\eta}$ is an acceptable symbol of order $-(n+1)/4$, and $\Omega_{\eta}$ is an acceptable symbol of order $-(n+1)/4-1$. 
We start by using the asymptotic homogeneity \eqref{it:asymphom1}  of $b$, which we write as
\begin{equation}
c(y, \eta):=(\eta \cdot \partial_\eta - 1) b(y, \eta) \in S^0_{1,1/2}.
\label{Bsharp-quasihom}\end{equation}
We use \eqref{Bsharp-quasihom}, twice,  to write \eqref{initial*} in the form
 \begin{align*}
\int_{\Rn}e^{i(x-y)\cdot \zeta}\bigg[&   \Big( (\zeta - \eta) \cdot \partial_\eta b(y, \eta) + b(y, \eta) + c(y, \eta) +(x-y) \cdot \partial_y b(y, \eta)\\
&-    \frac{|\eta| - |\zeta|}{|\eta|} \sum_{jk}(x-y)_j \eta_k \cdot \partial^2_{y_j \eta_k} b(y, \eta) + \frac{|\eta| - |\zeta|}{|\eta|} (x-y) \cdot \partial_{y} c(y, \eta)\\ 
 &- b(x, \zeta) \Big) \psi_\eta(\zeta)- i \frac{\hat \zeta}{|\eta|}  \cdot \partial_{y} b(y, \eta)\psi_{\eta}(\zeta) -i \partial_y b(y, \eta) \cdot  \Omega_{\eta}( \zeta) \bigg]
 \ud\zeta. 
\end{align*}
Since $c\in S^{0}_{1,1/2}$, we can check that the terms involving $c$, as well as the final two terms on the last line, are individually acceptable. For the term involving $\partial_{y}c$, this requires integrating by parts with respect to $\zeta$ to deal with the factor $x-y$. 

Thus we can write \eqref{initial*} in the form 
 \begin{equation}\label{initial2}
 \begin{aligned}
 \int_{\Rn} e^{i(x-y)\cdot \zeta}\bigg[  &\Big( (\zeta - \eta) \cdot \partial_\eta b(y, \eta) + b(y, \eta)+(x-y) \cdot \partial_y b(y, \eta)- b(x, \zeta)\\
 &-  \frac{|\eta| - |\zeta|}{|\eta|} \sum_{jk}(x-y)_j \eta_k \cdot \partial^2_{y_j \eta_k} b(y, \eta)   
  \Big) \psi_\eta(\zeta) \\
  & + s_{0}(x,y,\eta,\zeta)
 \bigg]
 \ud\zeta
\end{aligned}
\end{equation}
for some acceptable $s_{0}$ of order $-(n+1)/4$.
We now apply Taylor's formula to the first line of \eqref{initial2}, obtaining a similar expression to that in \eqref{Taylorhom}:
\begin{equation}\label{Taylor}
\begin{aligned}
&b(y, \eta) +  (\zeta - \eta) \cdot \partial_\eta b(y, \eta) +  (x-y) \cdot \partial_y b(y, \eta) -  b(x, \zeta) \\
&= -\sum_{jk} \int_0^1 (1-r)(x-y)_j (x-y)_k \partial^2_{y_j y_k} b(y + r(x-y), \eta + r(\zeta - \eta)) \ud r \\
&\ \ \ - 2\sum_{jk} \int_0^1 (1-r)(x-y)_j (\zeta - \eta)_k \partial^2_{y_j \eta_k} b(y + r(x-y), \eta + r(\zeta - \eta))\ud r\\
&\ \ \ - \sum_{jk} \int_0^1 (1-r)(\zeta - \eta)_j (\zeta - \eta)_k \partial^2_{\eta_j \eta_k} b(y + r(x-y), \eta + r(\zeta - \eta))\ud r.
\end{aligned}
\end{equation}
We examine the terms in \eqref{initial2}, using \eqref{Taylor} to substitute for the first line.

\paragraph{\textbf{The second line of \eqref{Taylor}}} As it stands, it appears to contribute a symbol of order $-(n+1)/4+1$. However, if we substitute this expression into the integral in \eqref{initial2}, we see that the factors of $(x-y)_j$ can be expressed as $D_{\zeta_j}$ hitting the exponential factor. Since we only need to consider $\zeta\in\supp(\psi_{\eta})$, integrating by parts in $\zeta$ yields a decay factor of $\langle \zeta \rangle^{-1}$. Indeed, each $\zeta$-derivative gains us $\langle \zeta \rangle^{-1}$ when it hits the $b$ factor, and $\lb \eta+r(\zeta-\eta)\rb^{-1/2}\eqsim\lb \zeta \rb^{-1/2}$ when it hits the $\psi_\eta$ factor, due to the symbol estimates in \eqref{eq:packetbounds1}. Using these bounds and  the fact that $\partial^2_{y_j y_k} b$ is a symbol of order $S^1_{1, 1/2}$, since $b\in\A^{2}S^{1}_{1,1/2}$, we conclude that we obtain an acceptable symbol of order $-(n+1)/4$.

\paragraph{\textbf{The fourth line of \eqref{Taylor}}}

Without loss of generality, we may suppose that $\hat{\eta}=(1,0,\ldots,0)$ is the first basis vector of $\Rn$. Fix $r\in[0,1]$, let $\zeta\in\supp(\psi_{\eta})$, and set $\xi:=\eta+r(\zeta-\eta)$. Then the symbol we have to consider is 
\begin{equation}\label{eq:termfour}
(\zeta - \eta)_j (\zeta - \eta)_k \partial^2_{\xi_j \xi_k} b(y + r(x-y), \xi)\psi_{\eta}(\zeta)
\end{equation}
for $1\leq j,k\leq n$. Note that $|\eta|\eqsim |\zeta|\eqsim |\xi|$ and $|\hat{\eta}-\hat{\xi}|\lesssim |\eta|^{-1/2}\eqsim|\zeta|^{-1/2}$. 

Now, if both $j$ and $k$ are at least $2$, then 
\[
|(\zeta-\eta)_{j}|\lesssim |\eta|^{1/2}\eqsim|\zeta|^{1/2},
\]
and similarly for $k$. Since $\partial^2_{\xi_j \xi_k} b\in S^{-1}_{1,1/2}$ and $|\xi|^{-1}\eqsim |\zeta|^{-1}$, this shows that \eqref{eq:termfour} is an acceptable symbol of order $-(n+1)/4$. In fact, there is a slight subtlety here, when differentiating the factors $(\zeta-\eta)_{j}$ and $(\zeta-\eta)_{k}$. Any nonradial derivative of these factors leads to additional decay of order $|\zeta|^{-1/2}$, and at first sight the same seems to hold for the radial derivative. However, one can write
\[
\hat{\zeta}\cdot\partial_{\zeta}=(\hat{\zeta}-\hat{\eta})\cdot\partial_{\zeta}+\hat{\eta}\cdot\partial_{\zeta}=(\hat{\zeta}-\hat{\eta})\cdot\partial_{\zeta}+\partial_{\zeta_{1}}.
\]
The first term applied to either $(\zeta-\eta)_{j}$ or $(\zeta-\eta)_{k}$ does yield additional decay of order $|\zeta|^{-1}$, whereas the second term annihilates these factors.

Next, suppose that $j=1$ and $k\geq2$. Note that
\begin{equation}\label{eq:termfour2}
\begin{aligned}
\hat \xi \cdot \partial_\xi \big( \partial_{\xi_{k}} b\big)&= \frac1{|\xi|} \xi \cdot \partial_\xi \big( \partial_{\xi_{k}} b\big)= \frac1{|\xi|} \partial_{\xi_k} \big( (\xi \cdot \partial_\xi - 1) b \big)\\
&= \frac1{|\xi|} \partial_{\xi_{k}}c \in S^{-2}_{1, 1/2}.
\end{aligned}
\end{equation}
Now write
\begin{equation}\label{eq:termfour22}
\partial^{2}_{\xi_{j}\xi_{k}}b=(\hat{\eta}-\hat{\xi})\cdot\partial_{\xi}(\partial_{\xi_{k}}b)+\hat{\xi}\cdot\partial_{\xi}(\partial_{\xi_{k}}b).
\end{equation}
The first term on the right-hand side leads to a symbol of order $-(n+1)/4$ in \eqref{eq:termfour}, since $|\hat{\eta}-\hat{\xi}|\lesssim |\zeta|^{-1/2}$ and $|(\zeta-\eta)_{j}(\zeta-\eta)_{k}|\lesssim |\zeta|^{3/2}$.  Again, there is a subtlety when differentiating the term $\hat{\eta}-\hat{\xi}$ in the radial direction, since $\xi$ depends on $\zeta$. To get the required decay, this time one writes 
\[
\hat{\zeta}\cdot\partial_{\zeta}=(\hat{\zeta}-\hat{\xi})\cdot\partial_{\zeta}+\hat{\xi}\cdot\partial_{\zeta},
\]
and notes that the first term yields a full gain of order $|\zeta|^{-1}$ when applied to $\hat{\eta}-\hat{\xi}$. On the other hand, $\hat{\xi}\cdot\partial_{\zeta}(\hat{\xi})=r\hat{\xi}\cdot\partial_{\xi}(\hat{\xi})=0$ since $\xi=r\zeta+(1-r)\eta$. Note that, when differentiating $(\zeta-\eta)_{1}$ in the radial direction, one automatically gains a full factor of $|\zeta|^{-1}$, since $|(\zeta-\eta)_{1}|\eqsim |\zeta|$.

In a similar manner, the second term in \eqref{eq:termfour22} leads to a symbol of order $-(n+1)/4-1/2$, by \eqref{eq:termfour2}. For $k=1$ and $j\geq2$, exchange the order of differentiation first.

Finally, suppose that $j=k=1$, and write
\[
\partial_{\xi_{1}}^{2}b=(\hat{\eta}-\hat{\xi}+\hat{\xi})\cdot\partial_{\xi}\big((\hat{\eta}-\hat{\xi}+\hat{\xi})\cdot \partial_{\xi}b\big).
\]
After expanding and cleaning up, we are left with
\begin{equation}\label{eq:termfour3}
(\hat{\eta}-\hat{\xi})\cdot\big(((\hat{\eta}-\hat{\xi})\cdot\partial_{\xi})\partial_{\xi}b\big)+2(\hat{\eta}-\hat{\xi})\cdot ((\hat{\xi}\cdot\partial_{\xi})\partial_{\xi}b)+(\hat{\xi}\cdot\partial_{\xi})^{2}b.
\end{equation}
The first term leads to an acceptable symbol of order $-(n+1)/4$, since $|\hat{\eta}-\hat{\xi}|^{2}\lesssim |\zeta|^{-1}$. For the second term, it follows from \eqref{eq:termfour2} that
\[
(\hat{\xi}\cdot\partial_{\xi})\partial_{\xi}b=\frac{1}{|\xi|}\partial_{\xi}c,
\]
the components of which are all $S^{-2}_{1,1/2}$ symbols. Hence this term leads to an acceptable symbol of order $-(n+1)/4-1/2$. The same argument applies to the double radial derivative in \eqref{eq:termfour3}, since $\hat{\xi}\cdot\partial_{\xi}(\hat{\xi})=0$ and
\[
(\hat{\xi}\cdot\partial_{\xi})^{2}b=\hat{\xi}\cdot((\hat{\xi}\cdot\partial_{\xi})\partial_{\xi}b).
\]

\paragraph{\textbf{The third line of \eqref{Taylor}}} 
We write this as
\begin{align*}
&-2\int_{0}^{1}(1-r)\sum_{jk} (x-y)_j (\zeta - \eta)_k \partial^2_{y_j \eta_k} b(y + r(x-y), \eta + r(\zeta - \eta))\ud r  \\
&= -2\int_{0}^{1}(1-r)\sum_{jk} (x-y)_j (\zeta - \eta)_k \partial^2_{y_j \eta_k} b(y, \eta) \ud r\\ 
&-2\int_{0}^{r}(1-r)
\sum_{jk}  (x-y)_j (\zeta - \eta)_k \partial^2_{y_j \eta_k} \big( b(y + r(x-y), \eta + r(\zeta - \eta)) - b(y, \eta) \big)\ud r.
\end{align*}
Note that the second line is equal to
\begin{equation}\label{eq:final}
-\sum_{jk} (x-y)_j (\zeta - \eta)_k \partial^2_{y_j \eta_k} b(y, \eta)
\end{equation}
On the other hand, the third line is treated by using Taylor's formula to write 
$$
\partial^2_{y_j \eta_k} \big( b(y + r(x-y), \eta + r(\zeta - \eta)) - b(y, \eta) \big)
$$
as an integral involving third partial derivatives of $b$. Notice that these third partial derivatives involve either two $y$-partial derivatives and one $\eta$-partial derivative, or one $y$-partial derivatives and two $\eta$-partial derivatives. In the former case we apply the same argument as for the second line of \eqref{Taylor}, while in the latter case we apply the same argument as for the fourth line of \eqref{Taylor}. 

\paragraph{\textbf{The third line of \eqref{initial2}, and \eqref{eq:final}}}
These are the final terms to be dealt with. The combination of these terms, after reinserting the factor $\psi_{\eta}(\zeta)$ in \eqref{eq:final}, is
\[
\begin{gathered}
- \psi_\eta(\zeta) \sum_{jk}(x-y)_j \Big(  \frac{|\eta| - |\zeta|}{|\eta|}\eta_k  + (\zeta - \eta)_k \Big) \partial^2_{y_j \eta_k} b(y, \eta) \\
 =  -\psi_\eta(\zeta) \sum_{jk}(x-y)_j \Big(  |\zeta| (\hat \eta_k - \hat \zeta_k) \Big) \partial^2_{y_j \eta_k} b(y, \eta) .
 \end{gathered}
\]
We note that $(\hat \zeta - \hat \eta ) \psi_\eta(\zeta)$ is an acceptable term of order $-(n+3)/4$, so $\psi_\eta(\zeta)|\zeta| (\hat \eta_k - \hat \zeta_k) \partial^2_{y_j \eta_k} b(y, \eta)$ has order $-(n-1)/4$. We  then integrate by parts to trade $(x-y)_j$ for $D_{\zeta_j}$, which decreases the order of this expression by a further $1/2$, using \eqref{eq:packetbounds1}. This shows that the overall order of this term is $-(n+1)/4$.

 This completes the proof of the proposition. 
\end{proof}

\subsubsection{Comletion of the proof of Theorem~\ref{thm:parametrix}} 
We have seen that the error term is
$$
( -i W^* V - b(x, D) W^* ) \F_tW,
$$
and we also know that $\F_tW$ is bounded from $\Hps$ to $\Tentps$, locally uniformly in $t$, by Theorem \ref{thm:flowprop}. So it suffices to show that the factor $-i W^* V - b(x, D) W^*$ is a bounded map from $\Tentps$ to $\Hps$. Since $W$ is in fact an isometry from $\Hps$ to $\Tentps$, it is enough to show that $W ( -i W^* V - b(x, D) W^* )$ is a bounded map on $\Tentps$. 

Recall that  the kernel of $(2\pi)^{n}( -i W^* V - b(x, D) W^* )$ can be written in the form \eqref{initial}, which was manipulated to \eqref{initial*}. Proposition~\ref{prop:s-symbol} then shows it can be written in the form \eqref{initial3}, with amplitude $s$ an acceptable symbol of order $-(n+1)/4$. Composing on the left with $W$ gives an operator with kernel of the form 
$$
K(x, \xi,y, \eta) :=\int_{\R^{3n}} e^{i(x-z) \cdot \theta} \psi_\xi(\theta) e^{i(z-y) \cdot \zeta} s(z, y, \eta, \zeta) \ud\theta\ud z\ud\zeta.
$$
This oscillatory integral is of the form treated in Proposition~\ref{prop:oscilbdd}, and this proposition then implies that $-i W^* V - b(x, D) W^*$ is bounded on $\Tentps$, which in turn concludes the proof of the bounds in \eqref{eq:boundederror}. 

Finally, the strong continuity of the error term follows immediately from the boundedness of $W ( -i W^* V - b(x, D) W^* )$ and the strong continuity of $t\mapsto \F_t$, proved in Theorem~\ref{thm:flowprop}. 
\end{proof}

\begin{remark}\label{rem:sharpreg}
The fact that $b\in\A^{2}S^{1}_{1,1/2}$, and specifically that one has uniform bounds for the second derivatives of $b$, was crucial in dealing with the second and third lines of \eqref{Taylor}.
\end{remark}

\addtocontents{toc}{\protect\setcounter{tocdepth}{0}}

\section*{Acknowledgements}

The authors would like to thank Pierre Portal for valuable advice about this manuscript and the encompassing project, and Hart Smith for useful conversations about the topic of the article. The authors are also grateful to the anonymous referee, for carefully reading the manuscript and for their suggestions.

\addtocontents{toc}{\protect\setcounter{tocdepth}{2}}

\appendix

\section{Kernel bounds}\label{app:kernel}

In this appendix we collect a result about kernel bounds for oscillatory integrals which has been used at various points in this article. In very similar settings, these kernel bounds can be found in~\cite{Smith98a} and~\cite{Geba-Tataru07}. They were then used in connection with the theory of tent spaces in~\cite{HaPoRo20}. Recall that $\Upsilon$ was defined in \eqref{eq:Upsilon}.

\begin{proposition}\label{prop:oscilbdd}
Let $a:\R^{6n}\to\C$ be measurable and such that the following properties hold:
\begin{enumerate}
\item For all $\xi,y,\eta\in\Rn$ one has $a(\xi,\cdot,\cdot,\cdot,y,\eta)\in C^{\infty}(\R^{3n})$. Moreover, for all $\alpha,\beta,\delta\in\Z_{+}^{n}$ and $\gamma,\veps\in \Z_{+}$ there exists an $\con\geq0$ such that
\[
|\partial_{z}^{\alpha}\partial_{\theta}^{\beta}( \hat{\theta}\cdot\partial_{\theta})^{\gamma}\partial_{\zeta}^{\delta}( \hat{\zeta}\cdot\partial_{\zeta})^{\veps}a(\xi,\zeta,z,\theta,y,\eta)|\leq \con\lb \theta\rb^{-\frac{n+1}{4}+\frac{|\alpha|}{2}-\frac{|\beta|}{2}-\gamma}\lb\zeta\rb^{-\frac{n+1}{4}-\frac{|\delta|}{2}-\veps}
\]
for all $\xi,\zeta,z,\theta,y,\eta\in\Rn$.
\item For all $(\xi,\zeta,z,\theta,y,\eta)\in \supp(a)$ with $\xi\neq0\neq \eta$ one has 
\[
\frac{1}{2}|\xi|\leq |\zeta|\leq 2|\xi|\text{ and }\frac{1}{2}|\eta|\leq |\theta|\leq 2|\theta|,
\]
as well as 
\[
|\hat{\zeta}-\hat{\xi}|\leq |\xi|^{-1/2}\text{ and }|\hat{\theta}-\hat{\eta}|\leq |\eta|^{-1/2}.
\]
\end{enumerate}
For $(x,\xi),(y,\eta)\in\Tp$ set
\begin{equation}
K(x,\xi,y,\eta):=\int_{\R^{3n}}e^{i((x-z)\cdot\zeta+(z-y)\cdot\theta)}a(\xi,\zeta,z,\theta,y,\eta)\ud \theta\ud z\ud\zeta.
\label{K-kernel}\end{equation}
Then for each $N\geq0$ there exists an $\con_{N}\geq0$ such that
\begin{equation}\label{eq:offsing}
|K(x,\xi,y,\eta)|\leq \con_{N}\Upsilon\big(\tfrac{|\xi|}{|\eta|}\big)^{N}\big(1+\rho^{-1}d((x,\hat{\xi}),(y,\hat{\eta}))^{2}\big)^{-N}
\end{equation}
for all $(x,\xi),(y,\eta)\in\Tp\setminus o$, where $\rho:=\min(|\xi|^{-1},|\eta|^{-1})$. In particular, the integral operator with kernel $K$ is bounded on $T^{p}_{s}(\Sp)$ for all $p\in[1,\infty]$ and $s\in\R$.
\end{proposition}
One could also allow for the appropriate dependence of $a$ on $x$, but we will not need such generality. 
\begin{proof}
First note that $K:\Tp\times\Tp\to \C$ is well defined and measurable. Indeed, the integrals in $\theta$ and $\zeta$ are absolutely convergent since $a$ is compactly supported in these variables, and to see that the integral in $z$ converges one can integrate by parts with respect to $\theta$ to gain factors of $(1+|z|^{2})^{-1}$.

By Proposition \ref{prop:offsingbound}, it suffices to prove the first statement. To this end we may in turn suppose that $a$ is compactly supported in the $z$ variable, as follows by multiplying by appropriate cut-offs and then applying the dominated convergence theorem, using an expression for $K$ where one has integrated by parts in $\theta$ sufficiently many times to make the integral absolutely convergent. 

Now the proof is contained in~\cite[Section 5]{HaPoRo20}. Indeed, the proof of~\cite[Theorem 5.1]{HaPoRo20} relies on showing that a kernel as in the proposition (see~\cite[Equation (5.6)]{HaPoRo20}), compactly supported in $z$, satisfies the assumptions of Proposition \ref{prop:offsingbound}. The only differences between our setting and~\cite{HaPoRo20} are that we have expressed $\xi$ and $\eta$ in Cartesian coordinates as opposed to spherical coordinates, that we consider a special case of~\cite[Equation (5.6)]{HaPoRo20} where $\Phi(z,\theta)=z\cdot\theta$ (with the roles of $\theta$ and $\zeta$ reversed), and that we allow $a$ to depend on $y$. The first difference leads to kernel bounds that look different at first sight but are in fact equivalent (see the proof of Proposition \ref{prop:offsingbound}), the second difference slightly simplifies the proof, and the dependence on $y$ plays no role in the proof.
\end{proof}

Proposition \ref{prop:oscilbdd} applies in particular to operators of the form $WTW^{*}$, for suitable operators $T$ on function spaces over $\Rn$. The simplest nontrivial example of this is the case where $T$ is the identity operator, and the boundedness of the resulting operator $WW^{*}$ on tent spaces, guaranteed by the following corollary, is used at various points throughout this article.

\begin{corollary}\label{cor:transformbounded}
For all $p\in[1,\infty]$ and $s\in\R$ one has $WW^{*}\in\La(T^{p}_{s}(\Sp))$.
\end{corollary}
\begin{proof}
By definition of $W$ and by the expression for $W^{*}$ in \eqref{eq:defadjoint}, $WW^{*}$ is an integral operator with kernel $K:\Tp\times\Tp\to\R$ given by 
\begin{equation}\label{eq:kernelWWstar}
K(x,\xi,y,\eta)=
(2\pi)^{-2n}\int_{\R^{3n}}e^{i((x-z)\cdot\zeta+(z-y)\cdot\theta)}\psi_{\xi}(\zeta)\psi_{\eta}(\theta)\ud\theta\ud z\ud\zeta
\end{equation}
for $(x,\xi),(y,\eta)\in\Tp$ with $|\xi|,|\eta|\geq1$, and similarly for other values of $\xi$ and $\eta$. Moreover, by writing $ \hat{\zeta}\cdot\partial_{\zeta}=(\hat{\zeta}-\hat{\xi})\cdot\partial_{\zeta}+ \hat{\xi}\cdot\partial_{\zeta}$, it follows from Lemma \ref{lem:packets} that for all $\alpha\in\Z_{+}^{n}$, $\beta\in\Z_{+}$ and $\xi,\zeta\in\Rn\setminus\{0\}$ one has 
\[
|(\hat{\zeta}\cdot\partial_{\zeta})^{\beta}\partial_{\zeta}^{\alpha}\psi_{\xi}(\zeta)|\lesssim |\zeta|^{-\frac{n+1}{4}-\frac{|\alpha|}{2}-\beta}.
\]
Now the proof is concluded by combining Proposition \ref{prop:oscilbdd} and Lemma \ref{lem:packets}.
\end{proof}

Of course, the expression for the kernel of $WW^{*}$ in \eqref{eq:kernelWWstar} can be simplified, and in fact for this specific kernel the proof of the bounds in \eqref{eq:offsing} can be simplified considerably, as shown in \cite[Proposition 3.6]{Rozendaal21}.

\begin{remark}\label{rem:transformbounded}
Proposition \ref{prop:oscilbdd} applies more generally to operators of the form $\wt{W}W^{*}$, where $\wt{W}$ is a modified wave packet transform that involves wave packets $\wt{\psi}_{\xi}$ and $\tilde{r}$ with similar support and boundedness properties as $\psi_{\xi}$ and $r$. For example, for $t\in\R$, $f\in\Sw'(\Rn)$ and $(x,\xi)\in\Tp$ one may set 
\[
\wt{W}f(x,\xi):=\begin{cases}
\wt{\psi}_{\xi}(D)f(x)&\text{if }|\xi|>1,\\
\ind_{[1/2,1]}(|\xi|)|\xi|^{-t}\tilde{r}(D)f(x)&\text{if }|\xi|
\leq 1,\end{cases}
\]
where $\wt{\psi}_{\xi}(\zeta):=|\xi|^{-t}\lb\zeta\rb^{t}\psi_{\xi}(\zeta)$ and $\tilde{r}(\zeta):=\lb \zeta\rb^{t}r(\zeta)$ for $\zeta\in\Rn$. The proof of Corollary \ref{cor:transformbounded} then shows that $\wt{W}W^{*}$ is bounded on $\La(T^{p}_{s}(\Sp))$ for all $p\in[1,\infty]$ and $s\in\R$. This fact is used in Proposition \ref{prop:transformsandweights}. 
\end{remark}

\section{Other regularity assumptions}\label{app:regularity}

In this appendix we indicate how our main results can be improved slightly by making more refined regularity assumptions on the differential operators involved.

Let $(\psi_{j})_{j=0}^{\infty}\subseteq C^{\infty}_{c}(\Rn)$ be the Littlewood--Paley decomposition from \eqref{eq:LitPaldecomp}. For $r>0$ we let the \emph{Zygmund space} $C^{r}_{*}(\Rn)$ consist of all $f\in\Sw'(\Rn)$ such that 
\[
\|f\|_{C^{r}_{*}(\Rn)}:=\sup_{j\in\Z_{+}}2^{jr}\|\psi_{j}(D)f\|_{L^{\infty}(\Rn)}<\infty.
\]
Then (see e.g.~\cite{Triebel10})
\begin{equation}\label{eq:Zygmund1}
\HT^{r,\infty}(\Rn)\subsetneq C^{r}_{*}(\Rn)=\Cr(\Rn)
\end{equation}
if $r\notin\N$, and
\begin{equation}\label{eq:Zygmund2}
\Cr(\Rn)\subsetneq \HT^{r,\infty}(\Rn)\subsetneq C^{r}_{*}(\Rn)
\end{equation}
if $r\in\N$. Here $\HT^{r,\infty}(\Rn)$ is as in \eqref{eq:classical}.

There are associated classes of rough symbols.

\begin{definition}\label{def:rough2}
Let $r>0$, $m\in\R$ and $\delta\in[0,1]$. Then $C^{r}_{*} S^{m}_{1,\delta}$ is the collection of $a:\R^{2n}\to\C$ such that for each $\alpha\in\Z_{+}^{n}$ there exists an $\con_{\alpha}\geq 0$ with the following properties:
\begin{enumerate}
\item\label{it:symbol21} For all $x,\xi\in\Rn$ one has $a(x,\cdot)\in C^{\infty}(\Rn)$ and $|\partial_{\eta}^{\alpha}a(x,\xi)|\leq \con_{\alpha}\lb\xi\rb^{m-|\alpha|}$. 
\item\label{it:symbol22} For all $\xi\in\Rn$ one has $\partial_{\xi}^{\alpha}a(\cdot,\xi)\in C^{r}_{*}(\Rn)$ and
\[
\|\partial_{\xi}^{\alpha}a(\cdot,\xi)\|_{C^{r}_{*}(\Rn)}\leq \con_{\alpha}\lb\xi\rb^{m-|\alpha|+r\delta}.
\]
\end{enumerate} 
Also, $\HT^{r,\infty}S^{m}_{1,\delta}$ is the class of $a\in C^{r}_{*}S^{m}_{1,\delta}$ such that, in \eqref{it:symbol22}, one has $\partial_{\xi}^{\alpha}a(\cdot,\xi)\in \HT^{r,\infty}(\Rn)$ and $\|\partial_{\xi}^{\alpha}a(\cdot,\xi)\|_{\HT^{r,\infty}(\Rn)}\leq \con_{\alpha}\lb\xi\rb^{m-|\alpha|+r\delta}$ for all $\xi\in\Rn$.
\end{definition}

These symbol classes are related to the results in this article through Theorem \ref{thm:roughpseudo}. Indeed,~\cite[Theorem 5.1]{Rozendaal22} shows that, under the assumptions of Theorem \ref{thm:roughpseudo}, 
\begin{equation}\label{eq:roughpseudo2}
a(x,D):\HT^{s+m+\rho,p}_{FIO}(\Rn)\to \Hps
\end{equation}
for all $-r/2+s(p)-\rho<s<r-s(p)$ if one merely assumes that $a\in C^{r}_{*}S^{m}_{1,1/2}$. Moreover, \eqref{eq:roughpseudo2} also holds for $s=r-s(p)$ if $a\in \HT^{r,\infty}S^{m}_{1,1/2}$. By \eqref{eq:Zygmund1} and \eqref{eq:Zygmund2}, also keeping in mind the additional bounds that are required of $a$ in Definition \ref{def:rough} \eqref{it:symbol2}, this generalizes Theorem \ref{thm:roughpseudo} for all $r>2$.

Using the same arguments as in Section \ref{sec:divform}, one can now show the following. For $r>2$, the existence and uniqueness statement in Theorem \ref{thm:divwave} holds more generally for $-r+s(p)+1<s<r-s(p)$ under the assumption that the coefficients $a_{i,j}$, $1\leq i,j\leq n$, of $L$ satisfy $a_{i,j}\in C^{r}_{*}(\Rn)$. Moreover, if $a_{i,j}\in \HT^{r,\infty}(\Rn)$ then the same statement holds for $s=-r+s(p)+1$ and $s=r-s(p)$ as well. One can obtain similar extensions of the other results in Sections \ref{sec:divform} and \ref{sec:standardform}. 

Note that, using these extensions of our results, \eqref{eq:Zygmund1} and \eqref{eq:Zygmund2} show why the case $r\in\N$ is special in the results of Sections \ref{sec:divform} and \ref{sec:standardform}.

\bibliographystyle{plain}
\bibliography{Bibliography}

\end{document}